\documentclass[a4paper, 10pt]{article}
\usepackage{mathrsfs}

\usepackage{epsfig}
\usepackage{amsmath}
\usepackage{amssymb}
\usepackage{latexsym}
\usepackage{amsfonts}
\usepackage{amsthm}
\usepackage[numbers,sort&compress]{natbib}
\usepackage{graphicx}
\ifx\pdfoutput\undefined  \else
 \fi
\usepackage[centerlast]{caption2}
\usepackage{color}

\newcommand {\cross } {\nu_D }

\newtheorem{definition}{Definition}

\newtheorem{lemma}{Lemma}
\newtheorem{theorem}{Theorem}
\newtheorem{proposition}{Proposition}
\newtheorem{observation}{Observation}

\numberwithin{figure}{section} \numberwithin{definition}{section}
\numberwithin{observation}{section} \numberwithin{lemma}{section}
\numberwithin{theorem}{section} \numberwithin{proposition}{section}
\numberwithin{conjecture}{section} \numberwithin{table}{section}

\DeclareSymbolFont{largesymbol}{OMX}{yhex}{m}{n}
\DeclareMathAccent{\Widehat}{\mathord}{largesymbol}{"62}

\begin{document}

\title{
{An upper bound for the crossing number of augmented cubes} \footnote{The
research is supported by NSFC (11001035, 60973014, 60803034) and
SRFDP (200801081017)}
\author{
Guoqing Wang$^a$, \ \ Haoli
Wang$^b$, \ \ Yuansheng
Yang$^{b,}$\thanks{Corresponding
author's E-mail : yangys@dlut.edu.cn}, \ \ Xuezhi Yang$^b$, \ \ Wenping Zheng$^c$\\
\\
$^a$Department of Mathematics,\\
Tianjin Polytechnic University, Tianjin 300387, China\\
\\
\\
$^b$Department of Computer Science, \\
Dalian University of Technology, Dalian 116024, China\\
\\
\\
$^c$Key Laboratory of Computational Intelligence and Chinese Information\\
Processing of Ministry of Education,\\
Shanxi University, Taiyuan 030006, China\\
}}

\date{}
\maketitle
\begin{abstract}

A {\it good drawing} of a graph $G$ is a drawing where the edges are
non-self-intersecting and each two edges have at most one point in
common, which is either a common end vertex or a crossing. The {\it
crossing number} of a graph $G$ is the minimum number of pairwise
intersections of edges in a good drawing of $G$ in the plane. The
{\it $n$-dimensional augmented cube} $AQ_n$, proposed by S.A. Choudum and V. Sunitha, is an important interconnection network with good
topological properties and applications. In this paper, we
obtain an upper bound on the crossing number of $AQ_n$ less than
$\frac{26}{32}4^{n}-(2n^2+\frac{7}{2}n-6)2^{n-2}$.

\bigskip

\end{abstract}

\section{Introduction}

Let $G$ be a simple connected graph with vertex set $V(G)$ and edge
set $E(G)$. The {\it crossing number} $cr(G)$ of a graph $G$ is the
minimum possible number of edge crossings in a drawing of $G$ in the
plane. The notion of crossing number is a central one for
Topological Graph Theory and has been studied extensively  by
mathematicians including Erd\H{o}s, Guy, Tur\'{a}n and Tutte, et al.
(see \cite{EG73,Guy60,Turan77,Tutte70}). The study of {\it crossing
number} not only is of theoretical importance, but also has many
applications in, for example,  VLSI theory and wiring layout
problems (see \cite{BL84,S05,L81,L83}).

However, the investigation on the crossing number of graphs is an
extremely difficult problem. In 1973, Erd\H{o}s and Guy \cite{EG73}
wrote, ``{\sl Almost all questions that one can ask about crossing
numbers remain unsolved.}'' Actually, Garey and Johnson \cite{GJ83}
proved that computing the crossing number is NP-complete. Not
surprisingly, there are only a few infinite families of graphs for
which the exact crossing numbers are known (see for example
\cite{LiYaZh,PanRich07,RichterThomassen95}). Therefore,
it is more practical to determine the upper and lower bounds of the
crossing number of a graph. In particular, the bounds of crossing
number of some popular parallel network topologies with good
topological properties and applications in VLSI theory would be of
theoretical importance and practical value. Among all the network
topologies, the hypercube $Q_n$ is one of the most popular
interconnection network because of its attractive properties, such
as strong connectivity, small diameter, symmetry, recursive
construction, relatively small degree, and regularity
\cite{BA84,L92}. Naturally, the crossing number of hypercubes has
attracted many researches for the past several decades (see
\cite{DR95,Egg70,FF00,FFSV08,M91,SV93}).

Concerned with upper bound of crossing number of hypercube, Eggleton
and  Guy \cite{Egg70} in 1970 established a drawing of $Q_n$ to show
\begin{equation}\label{equation conjeture of Erdos and Guy} {\rm
cr}(Q_n)\leq \frac{5}{32}4^n-\lfloor\frac{n^2+1}{2}\rfloor 2^{n-2}.
\end{equation}
However, a gap was found in their constructions. Erd\H{o}s and Guy
\cite{EG73} in 1973 stated the above inequality again as a
conjecture. In fact, Erd\H{o}s and Guy further conjectured the
\emph{equality} of \eqref{equation conjeture of Erdos and Guy}
holds. With
regard to the latest progress of this conjecture, the interested
readers are referred to \cite{FF00,FFSV08,YangWang}.

The {\it $n$-dimensional augmented
cube} $AQ_n$ proposed by S.A. Choudum and V. Sunitha \cite{CS02} in
2002 is an important variation of $Q_n$. The augmented cube not
only retains some favorable properties of $Q_n$ but also processes
some embedding properties that $Q_n$ does not
\cite{HCTH05,HLT07,MLX07}. Thus, it has drawn a great deal of
attention of research \cite{CCWH09,CH12,HH12,HC10,HS07,WMX07,LTTH09,XX07}.
Hence, to determine the bounds of the crossing number of the {\it
$n$-dimensional augmented cube} $AQ_n$ would be highly interesting.

Our main result in this paper is Theorem \ref{Theorem
Upper Bound for general n}, which gives a general upper bound of
$cr(AQ_n)$.

\begin{theorem}\label{Theorem Upper Bound for general n}
For $n\geq 8$,
$$cr(AQ_n)<\frac{26}{32}4^{n}-(2n^2+\frac{7}{2}n-6)2^{n-2}.$$
\end{theorem}

\section{Definitions and tools}

Let $\mathbb{R}$ be the set of real numbers. For any real numbers
$r<s$, let $\mathbb{R}_{r}^s=\{t\in \mathbb{R}: r\leq t\leq s \}.$
For integers $c,d\in \mathbb{Z}$ we set $[c,d]=\{t\in \mathbb{Z}:
c\leq t\leq d\}$. Let $G$ be a graph, and let $A$ and $B$ be sets of
vertices (not necessarily disjoint) of $G$. We denote by $E[A, B]$
the set of edges of $G$ with one end in $A$ and the other end in
$B$. If $A = B$, we simply write $E(A)$ for $E[A,A]$. For any vertex
subset $A\subseteq V(G)$, let $\langle A\rangle$ be the induced
subgraph of $A$. Suppose that the graph $G$ is drawn in the
$2$-dimensional Euclidean plane $\mathbb{R}\times \mathbb{R}$. Let
$u$ be a vertex of $G$. By $X_u$ and $Y_u$ we denote the $X$ and
$Y$-coordinates of $u$ on $\mathbb{R}\times \mathbb{R}$.

A drawing of $G$ is said to be a {\it good} drawing, provided that
no edge crosses itself, no adjacent edges cross each other, no two
edges cross more than once, and no three edges cross in a point. It
is well known that the crossing number of a graph is attained only
in {\it good} drawings of the graph. So, we always assume that all
drawings throughout this paper are good drawings. Let  $A$ and $B$ be two disjoint subsets of $E(G)$. In a drawing $D$ of a graph $G$,
the number of the crossings formed by an edge in $A$ and another
edge in $B$ is denoted by $\cross(A,B)$, the number of the crossings
that involve a pair of edges in $A$ is denoted by $\cross (A)$. For convenience,  $\nu_D(E(G))$ is abbreviated to $\nu_D(G)$.
Then the following statement is straightforward.
\begin{lemma}\label{lemma counting of crossings} Let $A$, $B$, $C$ be mutually disjoint subsets of
$E(G)$. Then,
$$\begin{array}{llll}
\nu_D(C,A\cup B)&=&\nu_D(C,A)+\nu_D(C,B),\\
\nu_D(A\cup B)&=&\nu_D(A)+\nu_D(B)+\nu_D(A,B).\\
\end{array}$$
\end{lemma}

The following observation will be useful for the calculations in
Section 3.

\begin{observation}\label{Observation two bunches cross} For any $m\geq
1$, let $R$ and $S$ be two non-horizontal bunches of $m$ parallel
lines starting from points $(0,0),(1,0),\ldots,(m-1,0)$ respectively
(see Figure 2.1), which are above the real $X$-axis. Then the number
of crossings between $R$ and $S$ is ${a\choose 2}$.
\end{observation}
\begin{figure}[ht]
\centering
\includegraphics[scale=1.0]{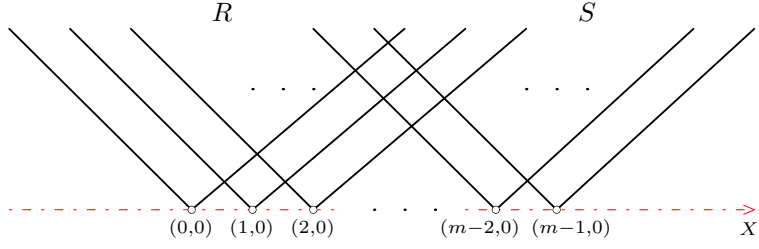}
\caption{\small{The crossings between two bunches of $m$ parallel
lines $R$ and $S$}}
\end{figure}

\medskip


Now we give the definition of the augmented cubes $AQ_n$. As with
augmented hypercubes, there are many ways to describe it, one of
which is as follows.

\begin{definition}\label{definition augmented cube}
The {\it $n$-dimensional augmented cube} $AQ_n$ is defined
recursively as follows: $AQ_1$ is a complete graph $K_2$ with the
vertex set $\{0,1\}$. For $n\geq 2$, $AQ_n$ is obtained by taking
two copies of the augmented cube $AQ_{n-1}$, denoted by $AQ_{n-1}^0$
and $AQ_{n-1}^1$, and adding $2\times 2^{n-1}$ edges between the two
as follows:\\
\indent Let $V(AQ_{n-1}^0)=\{0a_{n-1}\cdots a_2a_1:a_i\in\{0,1\}\}$
and $V(AQ_{n-1}^1)=\{1b_{n-1}\cdots b_2b_1:b_i\in\{0,1\}\}$. A
vertex $a=0a_{n-1}\cdots a_2a_1$ of $AQ_{n-1}^0$ is joined to a
vertex $b=1b_{n-1}\cdots b_2b_1$ of $AQ_{n-1}^1$ if and only if,
either $a_i=b_i$ for all $i\in [1,n-1]$, or $a_i=\overline{b}_i$ for
all $i\in [1,n-1]$.
\end{definition}

The graphs shown in Figure 2.2 are $AQ_1$, $AQ_2$ and $AQ_3$,
respectively.
\begin{figure}[ht]
\centering
\includegraphics[scale=1.0]{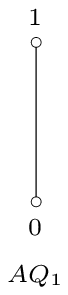} \hspace{50pt}
\includegraphics[scale=1.0]{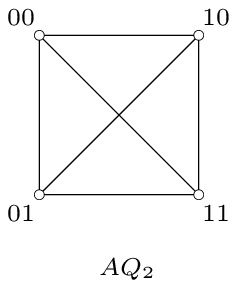} \hspace{50pt}
\includegraphics[scale=1.0]{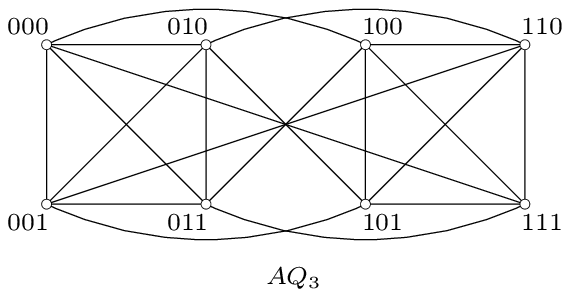}
\caption{\small{Augmented cubes $AQ_1$, $AQ_2$ and $AQ_3$}}
\end{figure}


Let $a=a_na_{n-1}\cdots a_1$ and $b=b_nb_{n-1}\cdots b_1$ be two
vertices of $AQ_n$. We define
$$\theta_i(a)=a_i \ \ \mbox{ for }\  i\in [1,n],$$
and
$$\begin{array}{llll} & \mbox{Dim}(a,b)=\left \{\begin{array}{llll}
               t, & \mbox{ if } t=1 \mbox{ or } \theta_{t-1}(a)=\theta_{t-1}(b);\\
               -t, &  \mbox{ otherwise},\\
              \end{array}
           \right. \\
\end{array}$$\\
where $t$ is the largest integer $i\in[1,n]$ such that
$\theta_i(a)\neq \theta_i(b)$. For convenience, let
$\mbox{Dim}(a,a)=0$. In particular, if $ab\in E(AQ_n)$, let
$\mbox{Dim}(ab)=\mbox{Dim}(a,b)$ and say the edge $ab$ is of
Dimension $\mbox{Dim}(ab)$. It is easy to see that
$${\rm Dim}(e)\in [-n,-2]\cup [1,n] \ \ \mbox{ for any edge }
e\in E(AQ_n).$$ For any $t\in [-n,-2]\cup [1,n]$, we define
$$\mathscr{I}_t: V(AQ_n)\rightarrow E(AQ_n)$$ to be a map such that
 $\mathscr{I}_t(a)\in E(AQ_n)$ is an
edge incident to $a$ with $${\rm Dim}(\mathscr{I}_t(a))=t.$$ In
particular, for any vertex subset $A\subseteq V(AQ_n)$, let
$$\mathscr{I}_t(A)=\{\mathscr{I}_t(a):a\in A\}.$$

Next we shall introduce a partition of $V(AQ_n)$ and $E(AQ_n)$ for
$n\geq 5$, together with a way to obtain such a partition
inductively, which will be used in the drawing of $AQ_n$ in Section
3.

Let
\begin{equation}\label{equation definition Un}
U^n=\{a\in V(AQ_n): \theta_{n-1}(a)=0\}
\end{equation}
 and
\begin{equation}\label{equation definition Vn}
V^n=\{a\in V(AQ_n): \theta_{n-1}(a)=1\}.
\end{equation}
Hence, we have the following partition
$$V(AQ_n)=U^n\cup V^n.$$
Let $\widehat{a}$ be the vertex of $AQ_n$ which is adjacent to $a$
with ${\rm Dim}(\widehat{a},a)=-(n-2)$. We define
$$\pi:V(AQ_n)\rightarrow V(AQ_{n+1})$$ to be a map such that
\begin{equation}\label{equation definition theta}
\begin{array}{llll} & \theta_i(\pi(a))=\left \{\begin{array}{llll}
               \theta_i(a), & \mbox{ if } \ \ i\in [1,n-1];\\
                \theta_{i-1}(a), & \mbox{ if } \ \ i\in \{n,n+1\}.\\
              \end{array}
           \right. \\
\end{array}
\end{equation}
For any subset $A\subseteq V(AQ_n)$,
 we define
\begin{equation}\label{equation definition omega}
\Omega(A)=\{\pi(a):a\in A\}\cup \{\widehat{\pi(a)}:a\in A\}
\end{equation}
 and
$$\Omega^{(m)}(A)=\underbrace{\Omega(\Omega(\cdots
(\Omega}\limits_{m}(A))\cdots))\ \ \mbox{ for }  m\geq 1.$$ For
convenience, let $$\Omega^{(0)}(A)=A.$$ It is easy to show that
\begin{equation}\label{equation a cap b=empty}
\Omega(\{a\})\cap \Omega(\{b\})=\emptyset
\end{equation}
for any two distinct vertices $a,b\in V(AQ_n)$ with $n\geq 5$.
 Since $|U^{n}|=|V^n|=2^{n-1}$, it follows from \eqref{equation definition Un}, \eqref{equation definition Vn}, \eqref{equation definition theta},
 \eqref{equation definition omega} and \eqref{equation a cap
 b=empty} that $$U^{n+1}=\Omega(U^n)$$ and  $$V^{n+1}=\Omega(V^n)$$
 where $n\geq 5$.

Now we define eight vertex subsets
$U_1^n,U_2^n,U_3^n,U_4^n,V_1^n,V_2^n,V_3^n,V_4^n$ of $V(AQ_n)$
inductively for $n\geq 5$ as follows:

If $n=5$, let
$$\begin{array}{rlll}
U_1^5&=\{00100,00011,00111,00000\},\\
U_2^5&=\{00101,00010,00110,00001\},\\
U_3^5&=\{10100,10011,10111,10000\},\\
U_4^5&=\{10101,10010,10110,10001\},\\
V_1^5&=\{01011,01100,01000,01111\},\\
V_2^5&=\{01010,01101,01001,01110\},\\
V_3^5&=\{11011,11100,11000,11111\},\\
V_4^5&=\{11010,11101,11001,11110\}.\\
\end{array}$$
If $n>5$, let
$$U_i^{n}=\Omega(U_i^{n-1})$$
and
$$V_i^{n}=\Omega(V_i^{n-1})$$
where $i\in \{1,2,3,4\}$.

By \eqref{equation a cap b=empty}, we have that the above eight
vertex subsets are pairwise disjoint. Moreover, by \eqref{equation
definition Un}, \eqref{equation definition Vn}, \eqref{equation
definition theta} and \eqref{equation definition omega}, we have
that
$$U^n=\bigcup\limits_{i=1}^{4} U_i^n$$
and $$V^n=\bigcup\limits_{i=1}^{4} V_i^n.$$ Observe
$$E(AQ_n)=E[U^n,V^n]\cup E(U^n)\cup E(V^n).$$

By induction on $n$ and straight verifications, we can prove the
following lemma which gives a partition of $E(AQ_n)$ and a way to
obtain $E(AQ_{n+1})$ from $E(AQ_n)$.

\begin{lemma}\label{Lemma enlarging the edges} Let $n\geq 5$. Then,

1. $E[U^n,V^n]=\{e\in E(AQ_n): {\rm Dim}(e)\in \{-n,-(n-1),n-1\}\}$.
Moreover, for any $uv\in E[U^n,V^n]$,
$$\begin{array}{llll}
E[\Omega(\{u\}),\Omega(\{v\})]=\left \{\begin{array}{llll}
               \{\pi(u)\pi(v),\widehat{\pi(u)}\widehat{\pi(v)},\pi(u)\widehat{\pi(v)},\widehat{\pi(u)}\pi(v)\},  &  \mbox{ if }\  {\rm Dim}(uv)=-(n-1);\\
               \{\pi(u)\pi(v),\widehat{\pi(u)}\widehat{\pi(v)}\},  &  \mbox{ if } \  {\rm Dim}(uv)=-n;\\
               \emptyset,  &  \mbox{ if } \  {\rm Dim}(uv)=n-1,\\
                             \end{array}
           \right. \\
\end{array}$$
with $${\rm Dim}(\widehat{\pi(u)}\widehat{\pi(v)})={\rm
Dim}(\pi(u)\pi(v)) =
               {\rm Dim}(uv)-1 \mbox{ if } {\rm Dim}(uv)\in
               \{-n,-(n-1)\}$$
and
$${\rm
Dim}(\pi(u)\widehat{\pi(v)})={\rm Dim}(\widehat{\pi(u)}\pi(v)) =
               n \mbox{ if } {\rm Dim}(uv)=-(n-1).$$

2. $E(U^n)\cup E(V^n)=\{e\in E(AQ_n): {\rm Dim}(e)\in
[-(n-2),-2]\cup [1,n-2]\cup \{n\}\}$. Moreover, for any $uv\in
E(U^n)\cup E(V^n)$,
$$\begin{array}{llll}
E[\Omega(\{u\}),\Omega(\{v\})]=\left \{\begin{array}{llll}
               \{\pi(u)\pi(v),\widehat{\pi(u)}\widehat{\pi(v)},\pi(u)\widehat{\pi(v)},\widehat{\pi(u)}\pi(v)\},  &  \mbox{ if } \ \ {\rm Dim}(uv)=-(n-2);\\
               \{\pi(u)\pi(v),\widehat{\pi(u)}\widehat{\pi(v)}\},  &  \mbox{ if } \ \ {\rm Dim}(uv)\in [-(n-3),-2]\cup[1,n-2]\cup \{n\},\\
              \end{array}
           \right. \\
\end{array}$$
with $$\begin{array}{llll} {\rm
Dim}(\widehat{\pi(u)}\widehat{\pi(v)})={\rm Dim}(\pi(u)\pi(v))
=\left \{\begin{array}{llll}
               {\rm Dim}(uv), & \mbox{ if } \ \ {\rm Dim}(uv)\in [-(n-2),-2]\cup [1,n-2];\\
               n+1, & \mbox{ if } \ \ {\rm
                Dim}(uv)=n,
              \end{array}
           \right. \\
\end{array}$$
and $${\rm Dim}(\pi(u)\widehat{\pi(v)})={\rm
Dim}(\widehat{\pi(u)}\pi(v)) =
               n-1  \mbox{ if }  {\rm Dim}(uv)=-(n-2).$$

3. For any two distinct vertices $u,v\in V(AQ_n)$ such that $u$ and
$v$ are not adjacent, $$E[\Omega(\{u\}),\Omega(\{v\})]=\emptyset.$$

4. $E(AQ_n)$ has a partition as follows:
$$\begin{array}{llll}
E[U_1^n,U_2^n]\cup E[U_3^n,U_4^n]\cup E[V_1^n,V_2^n]\cup E[V_3^n,V_4^n]&=&\{e\in E(AQ_n):{\rm Dim}(e)\in \{1,2\}\},\\
E[U_1^n,U_3^n]\cup E[U_2^n,U_4^n]\cup E[V_1^n,V_3^n]\cup E[V_2^n,V_4^n]&=&\{e\in E(AQ_n):{\rm Dim}(e)=n\},\\
E[U_1^n,V_3^n]\cup E[U_2^n,V_4^n]\cup E[V_1^n,U_3^n]\cup E[V_2^n,U_4^n]&=&\{e\in E(AQ_n):{\rm Dim}(e)=-n\},\\
\end{array}$$
$$\begin{array}{llll}
(\cup_{i=1}^{4}E(U_i^n))\bigcup (\cup_{i=1}^{4}E(V_i^n))&=&\{e\in E(AQ_n): {\rm Dim}(e)\in[-(n-2),-2]\cup[3,n-2]\},\\
\cup_{i=1}^4 E[U_i^n,V_i^n]&=&\{e\in E(AQ_n):{\rm Dim}(e)\in\{-(n-1),n-1\}\}.\\
\end{array}$$
\end{lemma}

In the rest of this section, we shall introduce a particular drawing
of some induced subgraph of $AQ_n$ together with some properties of
the drawing, which will be useful in Section 3.

Fix a positive integer $$\mathcal {N}\geq 5.$$  Take four vertices
in $V(AQ_{\mathcal {N}})$, denoted $z_1^{\mathcal {N}+(0)}$,
$z_2^{\mathcal {N}+(0)}$, $z_3^{\mathcal {N}+(0)}$,
$z_4^{\mathcal{N}+(0)}$, such that \eqref{equation Initial condition
1} and \eqref{equation Initial condition 2} hold:
\begin{equation}\label{equation Initial condition 1}
\begin{array}{llll}
E(\{z_1^{\mathcal{N}+(0)},z_2^{\mathcal {N}+(0)},z_3^{\mathcal {N}+(0)},z_4^{\mathcal{N}+(0)}\})&=&\{z_1^{\mathcal {N}+(0)}z_2^{\mathcal{N}+(0)},z_3^{\mathcal {N}+(0)}z_4^{\mathcal{N}+(0)},z_1^{\mathcal {N}+(0)}z_4^{\mathcal {N}+(0)}, \\
& &\ \ z_2^{\mathcal {N}+(0)}z_3^{\mathcal {N}+(0)},z_1^{\mathcal{N}+(0)}z_3^{\mathcal {N}+(0)},z_2^{\mathcal{N}+(0)}z_4^{\mathcal{N}+(0)}\}.\\
\end{array}
\end{equation}
\begin{equation}\label{equation
Initial condition 2}
\begin{array}{llll}
{\rm Dim}(z_1^{\mathcal {N}+(0)}z_2^{\mathcal {N}+(0)})&={\rm Dim}(z_3^{\mathcal {N}+(0)}z_4^{\mathcal {N}+(0)})=-(\mathcal{N}-2),\\
{\rm Dim}(z_1^{\mathcal {N}+(0)}z_4^{\mathcal {N}+(0)})&={\rm Dim}(z_2^{\mathcal {N}+(0)}z_3^{\mathcal {N}+(0)})=t_1, \\
{\rm Dim}(z_1^{\mathcal {N}+(0)}z_3^{\mathcal {N}+(0)})&={\rm Dim}(z_2^{\mathcal {N}+(0)}z_4^{\mathcal {N}+(0)})=t_2,\\
\end{array}
\end{equation}
where $$\{t_1,t_2\}=\{-(\mathcal{N}-3),\mathcal {N}-2\}.$$

By \eqref{equation definition Un}, \eqref{equation definition Vn}
and \eqref{equation Initial condition 2}, we see that either
$$z_j^{\mathcal {N}+(0)}\in U^{\mathcal {N}} \mbox{ for all }j\in [1,4],$$ or
$$z_j^{\mathcal {N}+(0)}\in V^{\mathcal {N}} \mbox{ for all }j\in [1,4].$$

For $m\geq 0$, we define $\Upsilon_{(m)}^{\mathcal {N}}$ to be the
drawing of the induced subgraph $\langle\Omega^{(m)}(\{z_1^{\mathcal
{N}+(0)},z_2^{\mathcal {N}+(0)},z_3^{\mathcal {N}+(0)},z_4^{\mathcal
{N}+(0)}\})\rangle$ of $AQ_{\mathcal {N}+m}$ such that all the
vertices are drawn precisely at some axis, say the real X-axis, and
all the edges drawn to be semi-circles above or below the X-axis,
and satisfies Inductive Rule A below. For convenience, we shall
denote all the vertices of $\Omega^{(m)}(\{z_1^{\mathcal
{N}+(0)},z_2^{\mathcal {N}+(0)},z_3^{\mathcal {N}+(0)},z_4^{\mathcal
{N}+(0)}\})$ by $z_1^{\mathcal {N}+(m)},z_2^{\mathcal {N}+(m)},
\ldots, z_{2^{m+2}}^{\mathcal {N}+(m)}$ such that
$$X_{z_1^{\mathcal {N}+(m)}}<X_{z_2^{\mathcal {N}+(m)}}<\cdots<X_{z_{2^{m+2}}^{\mathcal
{N}+(m)}}.$$

\medskip

\noindent\textbf{Inductive Rule A.}

We consider the case when $m=0$. Draw $\langle\{z_1^{\mathcal {N}+(0)},z_2^{\mathcal
{N}+(0)},z_3^{\mathcal {N}+(0)},z_4^{\mathcal {N}+(0)}\}\rangle$
such that Conditions (i) and (ii) hold.

\noindent (i) $$X_{z_i^{\mathcal {N}+(0)}}=i\mbox{ for }\  i\in
[1,4].$$
\noindent (ii) $$\mathcal {O}(z_1^{\mathcal
{N}+(0)}z_2^{\mathcal {N}+(0)})=\mathcal {O}(z_3^{\mathcal
{N}+(0)}z_4^{\mathcal {N}+(0)})=-1$$
 and $$\mathcal {O}(z_1^{\mathcal
{N}+(0)}z_3^{\mathcal {N}+(0)})=\mathcal {O}(z_2^{\mathcal
{N}+(0)}z_4^{\mathcal {N}+(0)})=\mathcal {O}(z_1^{\mathcal
{N}+(0)}z_4^{\mathcal {N}+(0)})=\mathcal {O}(z_2^{\mathcal
{N}+(0)}z_3^{\mathcal {N}+(0)})=1,$$ where
$$\begin{array}{llll} & \mathcal {O}(e)=\left \{\begin{array}{llll}
               1, & \mbox{ if $e$ is drawn above the X-axis};\\
               -1, & \mbox{ if $e$ is drawn below the X-axis},\\
              \end{array}
           \right. \\
\end{array}$$
for any $e\in E(\{z_1^{\mathcal {N}+(0)},z_2^{\mathcal
{N}+(0)},z_3^{\mathcal {N}+(0)},z_4^{\mathcal {N}+(0)}\})$.

Now we consider the case when $m>0$.

We first arrange the vertices of
$\langle\Omega^{(m)}(\{z_1^{\mathcal {N}+(0)},z_2^{\mathcal
{N}+(0)},z_3^{\mathcal {N}+(0)},z_4^{\mathcal {N}+(0)}\})\rangle$.
Take an arbitrary vertex $a= z_i^{\mathcal {N}+(m-1)}$ where $i\in
[1,2^{m+1}]$. We set
$$X_{\pi(a)}=X_{a}$$ and
$$X_{\widehat{\pi(a)}}
=X_{z_i^{\mathcal {N}+(m-1)}}+\frac{X_{z_{i+(-1)^{i-1}}^{\mathcal
{N}+(m-1)}}-X_{z_i^{\mathcal {N}+(m-1)}}}{3}.$$ It follows that
\begin{equation}\label{equation zi enlarging two}
\pi(z_i^{\mathcal
{N}+(m-1)})=z_{2i-\frac{1+(-1)^{i-1}}{2}}^{\mathcal {N}+(m)}
\end{equation}
and
\begin{equation}\label{equation zi enlarging pi} \Widehat{\pi(z_i^{\mathcal
{N}+(m-1)})}=z_{2i-\frac{1+(-1)^{i}}{2}}^{\mathcal {N}+(m)}
\end{equation} for all
$i\in [1,2^{m+1}]$.

Next we arrange the edges of $\langle\Omega^{(m)}(\{z_1^{\mathcal
{N}+(0)},z_2^{\mathcal {N}+(0)},z_3^{\mathcal {N}+(0)},z_4^{\mathcal
{N}+(0)}\})\rangle$. By \eqref{equation Initial condition 1},
\eqref{equation Initial condition 2} and applying Conclusions 2, 3
of Lemma \ref{Lemma enlarging the edges} repeatedly, we can show
that
\begin{eqnarray}\label{equaiton dimensions special drawing}
& &\{{\rm Dim}(e): e\in E\big{(}\Omega^{(m)}(\{z_1^{\mathcal
{N}+(0)},z_2^{\mathcal {N}+(0)},z_3^{\mathcal {N}+(0)},z_4^{\mathcal
{N}+(0)}\})\big{)}\}\nonumber \\
&=&[-(\mathcal {N}+m-2),-(\mathcal {N}-3)]\cup [\mathcal {N}-2,
\mathcal {N}+m-2].
\end{eqnarray}
Combined with \eqref{equation zi enlarging two} and \eqref{equation
zi enlarging pi}, we have
\begin{eqnarray}\label{equaiton gm dim -(N+m-2)}
& &\{e\in E\big{(}\Omega^{(m)}(\{z_1^{\mathcal
{N}+(0)},z_2^{\mathcal {N}+(0)},z_3^{\mathcal {N}+(0)},z_4^{\mathcal
{N}+(0)}\})\big{)}: {\rm Dim}(e)=-(\mathcal {N}+m-2)\}\nonumber\\
&=&\{z_{2i-1}^{\mathcal {N}+(m)} z_{2i}^{\mathcal {N}+(m)}: i\in
[1,2^{m+1}]\}.
\end{eqnarray}
Draw $$\mathcal {O}(z_{2i-1}^{\mathcal {N}+(m)} z_{2i}^{\mathcal
{N}+(m)})=-\mathcal {O}(z_{i}^{\mathcal {N}+(m-1)}
z_{i+(-1)^{i-1}}^{\mathcal {N}+(m-1)}) \mbox{ for all }
i\in[1,2^{m+1}].$$

Let $z_{k}^{\mathcal {N}+(m-1)}z_{\ell}^{\mathcal {N}+(m-1)}$ be an
arbitrary edge of $\langle\Omega^{(m-1)}(\{z_1^{\mathcal
{N}+(0)},z_2^{\mathcal {N}+(0)},z_3^{\mathcal {N}+(0)},z_4^{\mathcal
{N}+(0)}\})\rangle$, where $$1\leq k<\ell\leq 2^{m+1}.$$

Suppose $\ell=k+1$ and $k\equiv 1 \pmod 2$. By  Conclusion 2 of
Lemma \ref{Lemma enlarging the edges}, \eqref{equation zi enlarging
two}, \eqref{equation zi enlarging pi} and \eqref{equaiton gm dim
-(N+m-2)}, we see
$$E[\Omega(\{z_{k}^{\mathcal {N}+(m-1)}\}),\
\Omega(\{z_{\ell}^{\mathcal {N}+(m-1)}\})] = \{z_{2k-1}^{\mathcal
{N}+(m)}z_{2\ell-1}^{\mathcal {N}+(m)}, \ z_{2k-1}^{\mathcal
{N}+(m)}z_{2\ell}^{\mathcal {N}+(m)}, \ z_{2k}^{\mathcal
{N}+(m)}z_{2\ell-1}^{\mathcal {N}+(m)}, \ z_{2k}^{\mathcal
{N}+(m)}z_{2\ell}^{\mathcal {N}+(m)}\},$$ and draw
$$\mathcal {O}(z_{2k-1}^{\mathcal {N}+(m)}z_{2\ell-1}^{\mathcal {N}+(m)})
=\mathcal {O}(z_{2k-1}^{\mathcal {N}+(m)}z_{2\ell}^{\mathcal
{N}+(m)}) =\mathcal {O}(z_{2k}^{\mathcal
{N}+(m)}z_{2\ell-1}^{\mathcal {N}+(m)}) =\mathcal
{O}(z_{2k}^{\mathcal {N}+(m)}z_{2\ell}^{\mathcal {N}+(m)}) =\mathcal
{O}(z_{k}^{\mathcal {N}+(m-1)}z_{\ell}^{\mathcal {N}+(m-1)}).$$

Suppose $\ell=k+2$ and $k\equiv 1,2 \pmod 4$. By  Conclusion 2 of
Lemma \ref{Lemma enlarging the edges}, \eqref{equation zi enlarging
two}, \eqref{equation zi enlarging pi} and \eqref{equaiton
dimensions special drawing}, we see
$$E[\Omega(\{z_{k}^{\mathcal {N}+(m-1)}\}),\
\Omega(\{z_{\ell}^{\mathcal {N}+(m-1)}\})]=
\{z_{2k-\frac{1+(-1)^{k-1}}{2}}^{\mathcal
{N}+(m)}z_{2\ell-\frac{1+(-1)^{\ell-1}}{2}}^{\mathcal {N}+(m)}, \ \
z_{2k-\frac{1+(-1)^{k}}{2}}^{\mathcal
{N}+(m)}z_{2\ell-\frac{1+(-1)^{\ell}}{2}}^{\mathcal {N}+(m)}\},$$
and draw
$$\mathcal {O}(z_{2k-\frac{1+(-1)^{k-1}}{2}}^{\mathcal {N}+(m)}z_{2\ell-\frac{1+(-1)^{\ell-1}}{2}}^{\mathcal {N}+(m)})
=-\mathcal {O}(z_{k}^{\mathcal {N}+(m-1)}z_{\ell}^{\mathcal
{N}+(m-1)})$$ and
$$\mathcal {O}(z_{2k-\frac{1+(-1)^{k}}{2}}^{\mathcal {N}+(m)}z_{2\ell-\frac{1+(-1)^{\ell}}{2}}^{\mathcal {N}+(m)})
=\mathcal {O}(z_{k}^{\mathcal {N}+(m-1)}z_{\ell}^{\mathcal
{N}+(m-1)}).$$

Suppose otherwise. By  Conclusion 2 of Lemma \ref{Lemma enlarging
the edges}, \eqref{equation zi enlarging two}, \eqref{equation zi
enlarging pi}, \eqref{equaiton dimensions special drawing} and
\eqref{equaiton gm dim -(N+m-2)}, we see
$$E[\Omega(\{z_{k}^{\mathcal {N}+(m-1)}\}),\
\Omega(\{z_{\ell}^{\mathcal {N}+(m-1)}\})]=
\{z_{2k-\frac{1+(-1)^{k-1}}{2}}^{\mathcal
{N}+(m)}z_{2\ell-\frac{1+(-1)^{\ell-1}}{2}}^{\mathcal {N}+(m)}, \ \
z_{2k-\frac{1+(-1)^{k}}{2}}^{\mathcal
{N}+(m)}z_{2\ell-\frac{1+(-1)^{\ell}}{2}}^{\mathcal {N}+(m)}\},$$
and draw
$$\mathcal {O}(z_{2k-\frac{1+(-1)^{k-1}}{2}}^{\mathcal {N}+(m)}z_{2\ell-\frac{1+(-1)^{\ell-1}}{2}}^{\mathcal {N}+(m)})
=\mathcal {O}(z_{2k-\frac{1+(-1)^{k}}{2}}^{\mathcal
{N}+(m)}z_{2\ell-\frac{1+(-1)^{\ell}}{2}}^{\mathcal {N}+(m)})
=\mathcal {O}(z_{k}^{\mathcal {N}+(m-1)}z_{\ell}^{\mathcal
{N}+(m-1)}).$$

This completes the characterization of the drawing
$\Upsilon_{(m)}^{\mathcal {N}}$.

\noindent $\bullet$  We say the {\sl initial positive order} of the
above characterized drawing $\Upsilon_{(m)}^{\mathcal {N}}$ is
$(z_1^{\mathcal {N}+(0)},z_2^{\mathcal {N}+(0)},z_3^{\mathcal
{N}+(0)},z_4^{\mathcal {N}+(0)})$.

By the above arguments, we see that the drawing of
$\Upsilon_{(m)}^{\mathcal {N}}$ is independent of the value of
$\mathcal {N}$. Therefore, for the convenience, we shall write
$\Upsilon_{(m)}$ for $\Upsilon_{(m)}^{\mathcal {N}}$,
and write $z_i^{(m)}$ for $z_i^{\mathcal {N}+(m)}$, 
and write $\mathcal
{G}^{(m)}$ for the induced subgraph
$\Omega^{(m)}(\{z_1^{(0)},z_2^{(0)},z_3^{(0)},z_4^{(0)}\})$ when it
is unambiguous. While,
 in the rest of this paper, we always mean
$\{z_1^{(0)},z_2^{(0)},z_3^{(0)},z_4^{(0)}\}\subseteq V(AQ_{\mathcal
{N}})$ when we use $\mathcal {N}$.

To make the above notations clear, we give the drawings
$\Upsilon_{(0)},\Upsilon_{(1)},\Upsilon_{(2)}$ and $\Upsilon_{(3)}$
as examples shown in Figure 2.3.
\begin{figure}[ht]
\centering
\includegraphics[scale=1.0]{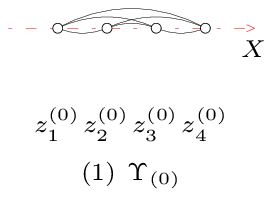}\hspace{4pt}
\includegraphics[scale=1.0]{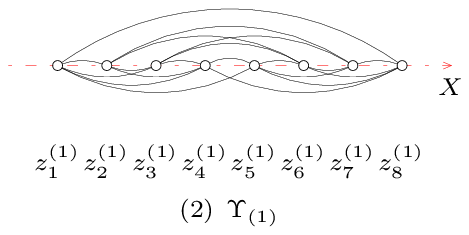}\hspace{4pt}
\includegraphics[scale=1.0]{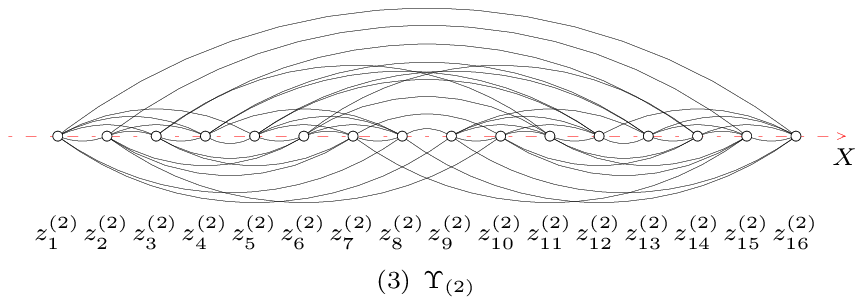}

\vspace{10pt}

\includegraphics[scale=1.0]{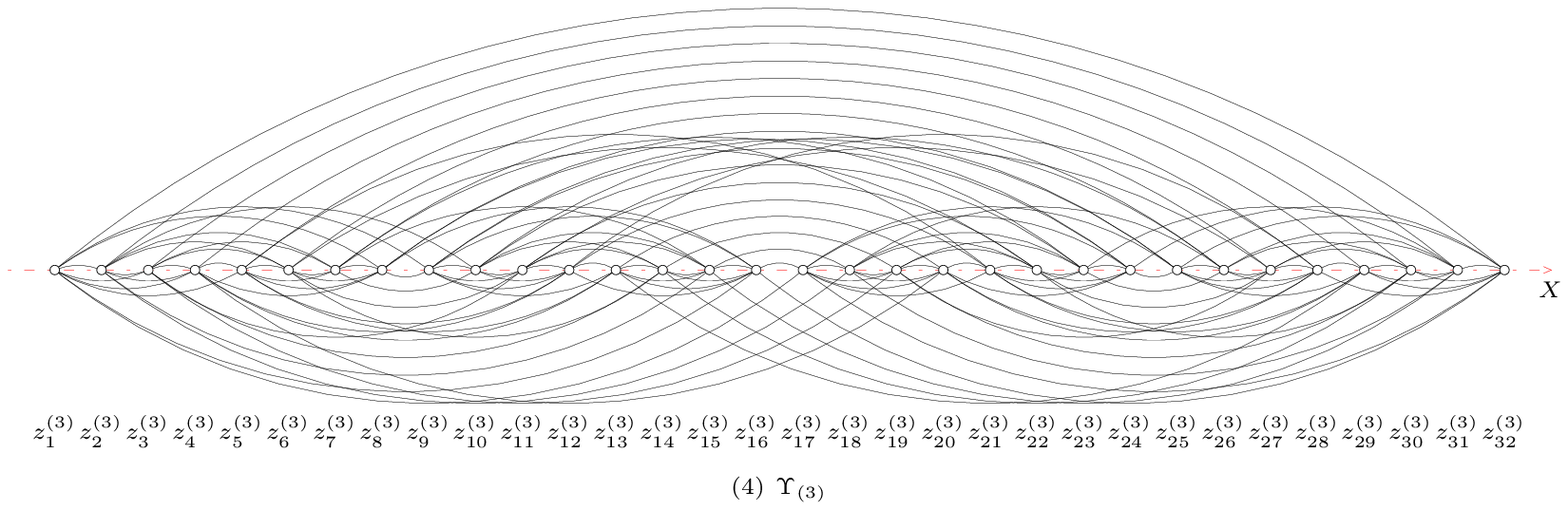}
\caption{\small{The drawings $\Upsilon_{(0)}$, $\Upsilon_{(1)}$,
$\Upsilon_{(2)}$ and $\Upsilon_{(3)}$}}
\end{figure}

Before giving the properties of the drawing $\Upsilon_{(m)}$, we
need to fix some notations.

Let $$\begin{array}{llll}
\mathcal {E}^{(m)}&=&E(\mathcal {G}^{(m)}),\\
\mathcal {E}_{\ell}^{(m)}&=&E(\{z_i^{(m)}: i\in [1,2^{m+1}]\}),\\
\mathcal {E}_{r}^{(m)}&=&E(\{z_i^{(m)}: i\in [2^{m+1}+1,2^{m+2}]\}),\\
\mathcal {H}^{(m)}&=&E[\{z_i^{(m)}: i\in [1,2^{m+1}]\},\{z_i^{(m)}:
i\in [2^{m+1}+1,2^{m+2}]\}].
\end{array}$$

Let $z_i^{(m)}z_j^{(m)}$ be an edge of $\mathcal {E}^{(m)}$, where
$i<j$. We say the edge $z_i^{(m)}z_j^{(m)}$ covers vertex
$z_t^{(m)}$ for any $t\in [i+1,j-1]$. Let $F$ be an edge subset of
$\mathcal {G}^{(m)}$, and let $v$ be a vertex of $\mathcal
{G}^{(m)}$. For any $i\in [1,2^{m+2}]$, we define
$$\begin{array}{llll}
\alpha_i^{(m),F}&=&|\{e\in F: e \mbox{ is incident to } z_i^{(m)} \mbox{ with } \mathcal {O}(e)=1\}|,\\
\beta_i^{(m),F}&=&|\{e\in F: e \mbox{ is incident to } z_i^{(m)} \mbox{ with } \mathcal {O}(e)=-1\}|,\\
\gamma_i^{(m),F}&=&|\{e\in F: e \mbox{ covers } z_i^{(m)} \mbox{ with } \mathcal {O}(e)=1\}|,\\
\xi_i^{(m),F}&=&|\{e\in F: e \mbox{ covers } z_i^{(m)} \mbox{ with }\mathcal {O}(e)=-1\}|.\\
\end{array}$$
We define
$$\mathscr{C}_{+}^{(m)}(F,v)=|\{e\in F: e \mbox{ covers } v \mbox{ under the drawing } \Upsilon_{(m)} \mbox{ with } \mathcal
{O}(e)=1\}|$$ and
$$\mathscr{C}_{-}^{(m)}(F,v)=|\{e\in F: e \mbox{ covers } v \mbox{ under the drawing } \Upsilon_{(m)} \mbox{ with } \mathcal {O}(e)=-1\}|.$$
In particular, let
$$\mathscr{C}_{+}^{(m)}(v)=\mathscr{C}_{+}^{(m)}(\mathcal {E}^{(m)},v)$$
and
$$\mathscr{C}_{-}^{(m)}(v)=\mathscr{C}_{-}^{(m)}(\mathcal {E}^{(m)},v).$$
Moreover, let
$$\mathcal {C}_{+}^{(m)}=\sum\limits_{i\in[1,2^{m+2}]}\mathscr{C}_{+}^{(m)}(z_i^{(m)})$$
and
$$\mathcal {C}_{-}^{(m)}=\sum\limits_{i\in[1,2^{m+2}]}\mathscr{C}_{-}^{(m)}(z_i^{(m)}).$$

Now we are ready to give some properties of the drawing
$\Upsilon_{(m)}$.

By Lemma \ref{Lemma enlarging the edges}, it is not hard to derive
the following two lemmas.

\begin{lemma}\label{Lemma counting upsion} Let $m\geq 0$, and let $i,j$ be two integers of
$[1,2^{m+2}]$. Then,

\noindent 1. Conclusions {\rm (i)}, {\rm (ii)} and {\rm (iii)} are
equivalent:

{\rm (i)} $z_i^{(m)}z_j^{(m)}\in \mathcal {E}^{(m)}$;

{\rm (ii)} $z_{2^{m+2}+1-i}^{(m)}z_{2^{m+2}+1-j}^{(m)}\in \mathcal
{E}^{(m)}$;

{\rm (iii)} $z_i^{(m+1)}z_j^{(m+1)}\in \mathcal {E}^{(m+1)}$.

\noindent Furthermore,

\noindent 2. If $z_i^{(m)}z_j^{(m)}\in \mathcal {E}^{(m)}$, then
$\mathcal {O}(z_i^{(m+1)}z_j^{(m+1)})=-\mathcal
{O}(z_i^{(m)}z_j^{(m)})=-\mathcal
{O}(z_{2^{m+2}+1-i}^{(m)}z_{2^{m+2}+1-j}^{(m)}).$
\end{lemma}

\begin{lemma}\label{lemma m,k} Let $m,k\geq 0$, and
let $z_i^{(m)}z_j^{(m)}\in\mathcal {E}^{(m)}$ such that $i<j-2$. For
any $t\in [1,2^{m+2}]$,
$$\begin{array}{llll}
& & \sum\limits_{v\in\Omega^{(k)}(\{z_t^{(m)}\})}\mathscr{C}_{+}^{(m+k)}\big{(}E[\Omega^{(k)}(\{z_i^{(m)}\}),\Omega^{(k)}(\{z_j^{(m)}\})], \ v\big{)}\\
&=& \left \{\begin{array}{llll}
               4^{k}, & \mbox{ if } \ i<t<j \mbox{ and } \mathcal {O}(z_i^{(m)}z_j^{(m)})=1;\\
               2^{k-1}\cdot (2^k-1), & \mbox{ if } \ t\in \{i,j\} \mbox{ and } \mathcal {O}(z_i^{(m)}z_j^{(m)})=1;\\
               0, & \mbox{ otherwise,}\\
              \end{array}
           \right. \\
\end{array}$$
and
$$\begin{array}{llll}
& & \sum\limits_{v\in\Omega^{(k)}(\{z_t^{(m)}\})}\mathscr{C}_{-}^{(m+k)}(E[\Omega^{(k)}(\{z_i^{(m)}\}),\Omega^{(k)}(\{z_j^{(m)}\})],v)\\
&=& \left \{\begin{array}{llll}
               4^{k}, & \mbox{ if } \ i<t<j \mbox{ and } \mathcal {O}(z_i^{(m)}z_j^{(m)})=-1;\\
               2^{k-1}\cdot (2^k-1), & \mbox{ if } \ t\in \{i,j\} \mbox{ and } \mathcal {O}(z_i^{(m)}z_j^{(m)})=-1;\\
               0, & \mbox{ otherwise.}\\
              \end{array}
           \right. \\
\end{array}$$
\end{lemma}

\begin{lemma}\label{Lemma Claim A}
 Let $m>0$, and let $z_i^{(m)}z_j^{(m)}\in \mathcal
{E}^{(m)}$ where $1\leq i<j\leq 2^{m+2}$. If $j-i=2$ then $i\equiv
1,2 \pmod 4$.
\end{lemma}
\begin{proof} By \eqref{equation zi
enlarging two} and \eqref{equation zi enlarging pi}, we have that
$z_i^{(m)}\in \Omega(\{z_{\lceil\frac{i}{2}\rceil}^{(m-1)}\})$ and
$z_j^{(m)}\in
\Omega(\{z_{\lceil\frac{j}{2}\rceil}^{(m-1)}\})=\Omega(\{z_{\lceil\frac{i}{2}\rceil+1}^{(m-1)}\})$,
and moreover that, either $z_i^{(m)}\in
\pi(\{z_{\lceil\frac{i}{2}\rceil}^{(m-1)}\})$ and $z_j^{(m)}\in
\Widehat{\pi(\{z_{\lceil\frac{i}{2}\rceil+1}^{(m-1)}\})}$, or
$z_i^{(m)}\in
\Widehat{\pi(\{z_{\lceil\frac{i}{2}\rceil}^{(m-1)}\})}$ and
$z_j^{(m)}\in \pi(\{z_{\lceil\frac{i}{2}\rceil+1}^{(m-1)}\})$.
Combined with Lemma \ref{Lemma enlarging the edges}, we have that
$z_{\lceil\frac{i}{2}\rceil}^{(m-1)}z_{\lceil\frac{i}{2}\rceil+1}^{(m-1)}\in
\mathcal {E}^{(m-1)}$ with ${\rm
Dim}(z_{\lceil\frac{i}{2}\rceil}^{(m-1)}z_{\lceil\frac{i}{2}\rceil+1}^{(m-1)})=-((\mathcal{N}+m-1)-2)$.
It follows from \eqref{equaiton gm dim -(N+m-2)} that
$\lceil\frac{i}{2}\rceil\equiv 1\pmod 2$, and so $i\equiv 1,2\pmod
4$. The lemma follows. \end{proof}

By Lemma \ref{Lemma Claim A}, we have that
\begin{equation}\label{equation H(m) cap n+m-2=empty}
|i-j|\neq 2 \  \mbox{ for any } z_i^{(m)}z_j^{(m)}\in \mathcal
{H}^{(m)} \mbox{ with } m>0.
\end{equation}

Let $m\geq 1$. We define
\begin{equation}\label{equation definition Itm}
\mathcal {I}_t^{(m)}=[(t-1)\cdot 2^{m-1}+1, \  t\cdot 2^{m-1}] \
\mbox{ for }t\in [1,8].
\end{equation}

By \eqref{equation H(m) cap n+m-2=empty} and Lemma \ref{Lemma
enlarging the edges}, we can derive that

\begin{equation}\label{equation ai(m)}
\begin{array}{llll} & \alpha_i^{(m),\mathcal {H}^{(m)}}=\left \{\begin{array}{llll}
               1, & i\in \mathcal{I}_1^{(m)}\cup \mathcal{I}_4^{(m)}\cup \mathcal{I}_5^{(m)}\cup \mathcal{I}_8^{(m)};\\
               2, & i\in \mathcal{I}_2^{(m)}\cup \mathcal{I}_3^{(m)}\cup \mathcal{I}_6^{(m)}\cup \mathcal{I}_7^{(m)},\\
              \end{array}
           \right. \\
\end{array}
 \end{equation}
 and \
 \begin{equation}\label{equation bi(m)}
\begin{array}{llll} & \beta_i^{(m),\mathcal {H}^{(m)}}=\left \{\begin{array}{llll}
               1, & i\in \mathcal{I}_1^{(m)}\cup \mathcal{I}_4^{(m)}\cup \mathcal{I}_5^{(m)}\cup \mathcal{I}_8^{(m)};\\
               0, & i\in \mathcal{I}_2^{(m)}\cup \mathcal{I}_3^{(m)}\cup \mathcal{I}_6^{(m)}\cup \mathcal{I}_7^{(m)},\\
              \end{array}
           \right. \\
\end{array}
 \end{equation}
 for all $m\geq 1$.

\begin{lemma}\label{Lemma C+ and C-} For $m\geq 0$, $$\mathcal {C}_{+}^{(m)}=\frac{7}{3}\cdot 4^{m+1}-(2m+\frac{16}{3}+\frac{7\cdot (1+(-1)^{m+1})}{6})\cdot 2^m$$
and
$$\mathcal {C}_{-}^{(m)}=\frac{5}{3}\cdot 4^{m+1}-(2m+\frac{13}{3}+\frac{7\cdot (1+(-1)^{m})}{6})\cdot 2^m$$
\end{lemma}

\begin{proof} If $m=0,1$, the lemma follows from trivial
verifications (see Figure 2.3 (1)-(2) for $\Upsilon_{(0)}$ and
$\Upsilon_{(1)}$). Now we consider the case when
$$m>1.$$

By \eqref{equation ai(m)}, \eqref{equation bi(m)} and trivial
verifications, we have the following
\begin{table}[htbp]\label{Table four indexes 1}
\centering
\begin{tabular}{|c|c|c|c|c|}
  \hline
  $j$                 & \ \ 1 \ \ & \ \ 2 \ \ & \ \ 3 \ \ & \ \ 4 \ \ \\ \hline
  $\alpha_j^{(1),\mathcal {H}^{(1)}}$    & 1 & 2 & 2 & 1 \\ \hline
  $\beta_j^{(1),\mathcal {H}^{(1)}}$     & 1 & 0 & 0 & 1 \\ \hline
  $\gamma_j^{(1),\mathcal {H}^{(1)}}$    & 0 & 1 & 3 & 5 \\ \hline
  $\xi_j^{(1),\mathcal {H}^{(1)}}$       & 0 & 1 & 1 & 1 \\ \hline
  \end{tabular}
\caption{\small{The values of $\alpha_j^{(1),\mathcal {H}^{(1)}}$,
$\beta_j^{(1),\mathcal {H}^{(1)}}$, $\gamma_j^{(1),\mathcal
{H}^{(1)}}$ and $\xi_j^{(1),\mathcal {H}^{(1)}}$ for $j=1,2,3,4$}}
\end{table}

Combined \eqref{equation H(m) cap n+m-2=empty}, Table 2.1, Lemma
\ref{Lemma enlarging the edges}, Lemma \ref{Lemma counting upsion}
and Lemma \ref{lemma m,k}, we have that
\begin{eqnarray}\label{equation C+=2C-}
\mathcal{C}_{+}^{(m)}&=&\sum\limits_{i\in[1,2^{m+2}]}\mathscr{C}_{+}^{(m)}(z_i^{(m)}) \nonumber\\
&=&2\cdot \big{(}\ \sum\limits_{i\in [1,2^{m+1}]}\mathscr{C}_{+}^{(m)}(z_i^{(m)})\ \big{)} \nonumber\\
&=& 2\cdot\big{(}\ \sum\limits_{i\in[1,2^{m+1}]}\mathscr{C}_{+}^{(m)}(\mathcal {E}_{\ell}^{(m)},z_i^{(m)})+\sum\limits_{i\in[1,2^{m+1}]}\mathscr{C}_{+}^{(m)}(\mathcal {H}^{(m)},z_i^{(m)})\ \big{)} \nonumber\\
&=& 2\cdot \big{(}\ \sum\limits_{i\in[1,2^{m+1}]}\mathscr{C}_{-}^{(m-1)}(\mathcal {E}^{(m-1)},z_i^{(m-1)})+\sum\limits_{i\in[1,2^{m+1}]}\mathscr{C}_{+}^{(m)}(\mathcal {H}^{(m)},z_i^{(m)})\ \big{)} \nonumber\\
&=& 2\cdot \big{(}\ \mathcal{C}_{-}^{(m-1)}+\sum\limits_{i\in[1,2^{m+1}]}\mathscr{C}_{+}^{(m)}(\mathcal {H}^{(m)},z_i^{(m)})\ \big{)} \nonumber\\
&=& 2\cdot \big{(}\ \mathcal{C}_{-}^{(m-1)}+4^{m-1}\cdot \sum_{j=1}^{4}\gamma_j^{(1),\mathcal {H}^{(1)}}+2^{m-2}\cdot(2^{m-1}-1)\cdot \sum_{j=1}^{4} \alpha_j^{(1),\mathcal {H}^{(1)}}\ \big{)} \nonumber\\
&=&2\cdot \mathcal{C}_{-}^{(m-1)}+(6\cdot 4^m-3\cdot 2^{m})
\end{eqnarray}
and
\begin{eqnarray}\label{equation C-=2C+}
\mathcal{C}_{-}^{(m)}&=&\sum\limits_{i\in[1,2^{m+2}]}\mathscr{C}_{-}^{(m)}(z_i^{(m)}) \nonumber\\
&=&2\cdot \big{(}\ \sum\limits_{i\in [1,2^{m+1}]}\mathscr{C}_{-}^{(m)}(z_i^{(m)})\ \big{)} \nonumber\\
&=&2\cdot \big{(}\ \sum\limits_{i\in[1,2^{m+1}]}\mathscr{C}_{-}^{(m)}(\mathcal {E}_{\ell}^{(m)},z_i^{(m)})+\sum\limits_{i\in[1,2^{m+1}]}\mathscr{C}_{-}^{(m)}(\mathcal {H}^{(m)},z_i^{(m)})\ \big{)} \nonumber\\
&=&2\cdot \big{(}\ \sum\limits_{i\in[1,2^{m+1}]}\mathscr{C}_{+}^{(m-1)}(\mathcal {E}^{(m-1)},z_i^{(m-1)})+\sum\limits_{i\in[1,2^{m+1}]}\mathscr{C}_{-}^{(m)}(\mathcal {H}^{(m)},z_i^{(m)})\ \big{)} \nonumber\\
&=&2\cdot \big{(}\ \mathcal{C}_{+}^{(m-1)}+\sum\limits_{i\in[1,2^{m+1}]}\mathscr{C}_{-}^{(m)}(\mathcal {H}^{(m)},z_i^{(m)})\ \big{)} \nonumber\\
&=&2\cdot \big{(}\ \mathcal{C}_{+}^{(m-1)}+4^{m-1}\cdot \sum_{j=1}^{4}\xi_j^{(1),\mathcal {H}^{(1)}}+2^{m-2}\cdot(2^{m-1}-1)\cdot \sum_{j=1}^{4} \beta_j^{(1),\mathcal {H}^{(1)}}\ \big{)} \nonumber\\
&=&2\cdot \mathcal{C}_{+}^{(m-1)}+(2\cdot 4^m-2^{m}).
\end{eqnarray}
Then the lemma follows from \eqref{equation C+=2C-} and
\eqref{equation C-=2C+}.
\end{proof}

\begin{lemma} \label{Lemma three conclusions on C+ and C-}
The following three conclusions hold.

{\rm (i)} For $m=1$,
$$\begin{array}{llll} & \sum\limits_{i\in \mathcal {I}_t^{(m)}}\mathscr{C}_{+}^{(m)}(\mathcal {E}^{(m)},z_i^{(m)})=\left \{\begin{array}{llll}
                0, & \mbox{ if } \ \ t\in \{1,8\};\\
                1, & \mbox{ if } \ \ t\in \{2,7\};\\
                3, & \mbox{ if } \ \ t\in \{3,6\};\\
                5, & \mbox{ if } \ \ t\in \{4,5\},\\
              \end{array}
           \right. \\
\end{array}$$
and
$$\begin{array}{llll} & \sum\limits_{i\in \mathcal {I}_t^{(m)}}\mathscr{C}_{-}^{(m)}(\mathcal {E}^{(m)},z_i^{(m)})=\left \{\begin{array}{llll}
                0, & \mbox{ if } \ \ t\in \{1,8\};\\
                3, & \mbox{ if } \ \ t\in \{2,7\};\\
                3, & \mbox{ if } \ \ t\in \{3,6\};\\
                1, & \mbox{ if } \ \ t\in \{4,5\}.\\
              \end{array}
           \right. \\
\end{array}$$

{\rm (ii)} For $m=2$,
$$\begin{array}{llll} & \sum\limits_{i\in \mathcal {I}_t^{(m)}}\mathscr{C}_{+}^{(m)}(\mathcal {E}^{(m)},z_i^{(m)})=\left \{\begin{array}{llll}
                4, & \mbox{ if } \ \ t\in \{1,8\};\\
                10, & \mbox{ if } \ \ t\in \{2,7\};\\
                18, & \mbox{ if } \ \ t\in \{3,6\};\\
                24, & \mbox{ if } \ \ t\in \{4,5\},\\
              \end{array}
           \right. \\
\end{array}$$
and
$$\begin{array}{llll} & \sum\limits_{i\in \mathcal {I}_t^{(m)}}\mathscr{C}_{-}^{(m)}(\mathcal {E}^{(m)},z_i^{(m)})=\left \{\begin{array}{llll}
                2, & \mbox{ if } \ \ t\in \{1,8\};\\
                12, & \mbox{ if } \ \ t\in \{2,7\};\\
                12, & \mbox{ if } \ \ t\in \{3,6\};\\
                6, & \mbox{ if } \ \ t\in \{4,5\}.\\
              \end{array}
           \right. \\
\end{array}$$

{\rm (iii)} For $m\geq 3$,
$$\begin{array}{llll} & \sum\limits_{i\in \mathcal {I}_t^{(m)}}\mathscr{C}_{+}^{(m)}(\mathcal {E}^{(m)},z_i^{(m)})=\left \{\begin{array}{llll}
                \frac{98}{3}\cdot4^{m-3}-(2m+\frac{13}{3}+\frac{7\cdot(1+(-1)^{m+1})}{6})\cdot2^{m-3}, & \mbox{ if } \ \ t\in \{1,8\};\\
                \frac{182}{3}\cdot4^{m-3}-(2m+\frac{19}{3}+\frac{7\cdot(1+(-1)^{m+1})}{6})\cdot2^{m-3}, & \mbox{ if } \ \ t\in \{2,7\};\\
                \frac{278}{3}\cdot4^{m-3}-(2m+\frac{19}{3}+\frac{7\cdot(1+(-1)^{m+1})}{6})\cdot2^{m-3}, & \mbox{ if } \ \ t\in \{3,6\};\\
                \frac{338}{3}\cdot4^{m-3}-(2m+\frac{13}{3}+\frac{7\cdot(1+(-1)^{m+1})}{6})\cdot2^{m-3}, & \mbox{ if } \ \ t\in \{4,5\},\\
              \end{array}
           \right. \\
\end{array}$$
and
$$\begin{array}{llll} & \sum\limits_{i\in \mathcal {I}_t^{(m)}}\mathscr{C}_{-}^{(m)}(\mathcal {E}^{(m)},z_i^{(m)})=\left \{\begin{array}{llll}
                \frac{94}{3}\cdot4^{m-3}-(2m+\frac{16}{3}+\frac{7\cdot(1+(-1)^{m})}{6})\cdot2^{m-3}, & \mbox{ if } \ \ t\in \{1,8\};\\
                \frac{202}{3}\cdot4^{m-3}-(2m+\frac{10}{3}+\frac{7\cdot(1+(-1)^{m})}{6})\cdot2^{m-3}, & \mbox{ if } \ \ t\in \{2,7\};\\
                \frac{202}{3}\cdot4^{m-3}-(2m+\frac{10}{3}+\frac{7\cdot(1+(-1)^{m})}{6})\cdot2^{m-3}, & \mbox{ if } \ \ t\in \{3,6\};\\
                \frac{142}{3}\cdot4^{m-3}-(2m+\frac{16}{3}+\frac{7\cdot(1+(-1)^{m})}{6})\cdot2^{m-3}, & \mbox{ if } \ \ t\in \{4,5\}.\\
              \end{array}
           \right. \\
\end{array}$$
\end{lemma}

\begin{proof} (i) and (ii) follows from straight verifications (see Figure 2.3). Now we
let $m\geq 3$ and prove (iii).

For $t\in [1,8]$, let
$$\mathcal {E}_t^{(m)}=E(\{z_i^{(m)}:i\in \mathcal {I}_t^{(m)}\}),$$
$$\overline{\mathcal {E}_t^{(m)}}=\mathcal {E}^{(m)}\setminus \mathcal {E}_t^{(m)}.$$
Let $$\mathcal {K}^{(m)}=\mathcal {E}^{(m)}\setminus \cup_{i=1}^8
\mathcal {E}_i^{(m)}.$$

By \eqref{equation definition Itm} and Lemma \ref{Lemma Claim A}, we
have that
\begin{equation}\label{equation F cap n+m-2=empty}
|i-j|\neq 2 \  \mbox{ for any } z_i^{(m)}z_j^{(m)}\in \mathcal
{K}^{(m)} \mbox{ with } m\geq 3.
\end{equation}

Combined \eqref{equation F cap n+m-2=empty}, Lemma \ref{Lemma
enlarging the edges}, Lemma \ref{Lemma counting upsion}, Lemma
\ref{lemma m,k} and Lemma \ref{Lemma C+ and C-}, we have that for
any $t\in [1,8]$,
\begin{eqnarray}\label{equation +, t in [1,8]}
& &\sum\limits_{i\in
\mathcal {I}_t^{(m)}}\mathscr{C}_{+}^{(m)}(\mathcal {E}^{(m)},z_i^{(m)})\nonumber \\
&=&\sum\limits_{i\in \mathcal
{I}_t^{(m)}}\mathscr{C}_{+}^{(m)}(\mathcal
{E}_t^{(m)},z_i^{(m)})+\sum\limits_{i\in
\mathcal {I}_t^{(m)}}\mathscr{C}_{+}^{(m)}(\overline{\mathcal {E}_t^{(m)}},z_i^{(m)})\nonumber \\
&=&\sum\limits_{i\in \mathcal
{I}_t^{(m)}}\mathscr{C}_{+}^{(m)}(\mathcal
{E}_t^{(m)},z_i^{(m)})+\sum\limits_{i\in
\mathcal {I}_t^{(m)}}\mathscr{C}_{+}^{(m)}(\mathcal {K}^{(m)},z_i^{(m)})\nonumber \\
&=&\mathcal {C}_{-}^{(m-3)}+\sum\limits_{i\in
\mathcal {I}_t^{(m)}}\mathscr{C}_{+}^{(m)}(\mathcal {K}^{(m)},z_i^{(m)})\nonumber \\
&=&\mathcal {C}_{-}^{(m-3)}+\big{(}\ 4^{m-3}\cdot
\sum_{j=4(t-1)+1}^{4t} \gamma_j^{(3),\mathcal
{K}^{(3)}}+2^{m-4}\cdot (2^{m-3}-1)\cdot \sum_{j=4(t-1)+1}^{4t}
\alpha_j^{(3),\mathcal {K}^{(3)}}\ \big{)}
\end{eqnarray}
and similarly,
\begin{eqnarray}\label{equation -, t in [1,8]}
\sum\limits_{i\in \mathcal
{I}_t^{(m)}}\mathscr{C}_{-}^{(m)}(\mathcal {E}^{(m)},z_i^{(m)})
&=&\mathcal {C}_{+}^{(m-3)}+\big{(}\ 4^{m-3}\cdot
\sum_{j=4(t-1)+1}^{4t} \xi_j^{(3),\mathcal {K}^{(3)}}+2^{m-4}\cdot
(2^{m-3}-1)\cdot \sum_{j=4(t-1)+1}^{4t}
\beta_j^{(3),\mathcal {K}^{(3)}}\ \big{)}.\nonumber\\
\end{eqnarray}

It is easy to get the values in Table 2.2.
\begin{table}[htbp]\label{Table four indexes 2}
\centering
\begin{tabular}{|c|c|c|c|c|c|c|c|c|c|c|c|c|c|c|c|c|}
  \hline
  $j$                & 1 & 2 & 3 & 4  & 5  & 6  & 7  & 8  & 9  & 10 & 11 & 12 & 13 & 14 & 15 & 16 \\ \hline
  $\alpha_j^{(3),\mathcal {K}^{(3)}}$   & 3 & 4 & 3 & 2  & 3  & 4  & 5  & 4  & 4  & 5  & 4  & 3  & 2  & 3  & 4  & 3  \\ \hline
  $\beta_j^{(3),\mathcal {K}^{(3)}}$    & 3 & 2 & 3 & 4  & 3  & 2  & 1  & 2  & 2  & 1  & 2  & 3  & 4  & 3  & 2  & 3  \\ \hline
  $\gamma_j^{(3),\mathcal {K}^{(3)}}$   & 0 & 3 & 7 & 10 & 11 & 11 & 11 & 13 & 15 & 17 & 21 & 25 & 27 & 26 & 24 & 23 \\ \hline
  $\xi_j^{(3),\mathcal {K}^{(3)}}$      & 0 & 3 & 5 & 8  & 11 & 13 & 15 & 15 & 15 & 15 & 13 & 11 & 9  & 8  & 8  & 7  \\ \hline
  \end{tabular}
\caption{\small{The values of $\alpha_j^{(3),\mathcal {K}^{(3)}}$,
$\beta_j^{(3),\mathcal {K}^{(3)}}$, $\gamma_j^{(3),\mathcal
{K}^{(3)}}$ and $\xi_j^{(3),\mathcal {K}^{(3)}}$ for $j\in [1,16]$}}
\end{table}

Then Conclusion (iii) follows from \eqref{equation +, t in [1,8]},
\eqref{equation -, t in [1,8]}, Table 2.2 and Lemma \ref{Lemma C+
and C-} readily.
\end{proof}

In the following lemma, we shall count the crossings of
$\Upsilon_{(m)}$ for all $m\geq 1$.

\begin{lemma}\label{Lemma counting nu H(m)}
For $m\geq 1$, $\nu_{\Upsilon_{(m)}}(\mathcal {H}^{(m)})=6\cdot
4^{m-1}-2^{m+1}.$
\end{lemma}

\begin{proof} By \eqref{equation zi enlarging two}, \eqref{equation zi enlarging pi}, \eqref{equation H(m) cap n+m-2=empty} and Lemma
\ref{Lemma enlarging the edges}, we have the following

\textbf{Claim A.} Let $z_i^{(m)}z_j^{(m)}$ and $z_i^{(m)}z_k^{(m)}$
be two distinct edges of $\mathcal {H}^{(m)}$ incident to
$z_i^{(m)}$. Then
$$\nu_{\Upsilon_{(m+1)}}\big{(}\ E[\Omega(\{z_i^{(m)}\}), \Omega(\{z_j^{(m)}\})],\ E[\Omega(\{z_i^{(m)}\}),\Omega(\{z_k^{(m)}\})]\ \big{)}=1.$$

Now we show

\textbf{Claim B.} Let $z_i^{(m)}z_j^{(m)}$ be an arbitrary edge of
$\mathcal {H}^{(m)}$. Then
$$\begin{array}{llll}
\nu_{\Upsilon_{(m+1)}}\big{(}\ E[\Omega(\{z_i^{(m)}\}),
\Omega(\{z_j^{(m)}\})]\ \big{)}= \left \{\begin{array}{llll}
               1, & \mbox{ if } i\equiv j\pmod 2;\\
               0, & \mbox{ if } i\not\equiv j\pmod 2.\\
              \end{array}
           \right. \\
\end{array}$$

{\sl Proof of Claim B.} \ Say $i<j$. By \eqref{equation zi enlarging
two}, \eqref{equation zi enlarging pi}, \eqref{equation H(m) cap
n+m-2=empty} and Lemma \ref{Lemma enlarging the edges}, we have that
$$\begin{array}{llll}
& &E[\Omega(\{z_i^{(m)}\}), \Omega(\{z_j^{(m)}\})]\\
&=& \{\pi(z_i^{(m)})\pi(z_j^{(m)}), \
\Widehat{\pi(z_i^{(m)})}\Widehat{\pi(z_j^{(m)})}\}\\
&=&\{z_{2i-\frac{1+(-1)^{i-1}}{2}}^{(m+1)}z_{2j-\frac{1+(-1)^{j-1}}{2}}^{(m+1)},
\
z_{2i-\frac{1+(-1)^{i}}{2}}^{(m+1)}z_{2j-\frac{1+(-1)^{j}}{2}}^{(m+1)}\}
\end{array}$$
Also, since $\mathcal
{O}(z_{2i-\frac{1+(-1)^{i-1}}{2}}^{(m+1)}z_{2j-\frac{1+(-1)^{j-1}}{2}}^{(m+1)})=\mathcal
{O}(z_{2i-\frac{1+(-1)^{i}}{2}}^{(m+1)}z_{2j-\frac{1+(-1)^{j}}{2}}^{(m+1)})=\mathcal
{O}(z_{i}^{(m)}z_{j}^{(m)})$, we have Claim B proved. \qed

In a similar argument as Claim B, by induction on $m$, we have
\begin{equation}\label{equation sharp= 2exp(m+1)}
|\{z_i^{(m)}z_j^{(m)}\in \mathcal {H}^{(m)}:i\equiv j\pmod
2\}|=2^{m+1}.
\end{equation}

Now we proceed with the proof of the lemma.

If $m=1$, the lemma holds trivially (see Figure 2.3). Suppose $m>1$.
By \eqref{equation H(m) cap n+m-2=empty}, \eqref{equation ai(m)},
\eqref{equation bi(m)}, \eqref{equation sharp= 2exp(m+1)}, Claim A,
Claim B and Lemma \ref{Lemma enlarging the edges}, we have that
$$\begin{array}{llll}
& &\nu_{\Upsilon_{(m)}}(\mathcal {H}^{(m)})\\
&=&4\cdot
\nu_{\Upsilon_{(m-1)}}(\mathcal
{H}^{(m-1)})+|\{z_i^{(m-1)}z_j^{(m-1)}\in
\mathcal {H}^{(m-1)}:i\equiv j\pmod 2\}|\\
& &+\sum\limits_{i=1}^{2^{m+1}}{\alpha_i^{(m-1),\mathcal {H}^{(m-1)}}\choose 2}+\sum\limits_{i=1}^{2^{m+1}}{\beta_i^{(m-1),\mathcal {H}^{(m-1)}}\choose 2}\\
&=&4\cdot \nu_{\Upsilon_{(m-1)}}(\mathcal
{H}^{(m-1)})+2^{m}\\
& &+\sum\limits_{i\in \mathcal {I}_1^{(m-1)}\cup \mathcal
{I}_4^{(m-1)}\cup\mathcal {I}_5^{(m-1)}\cup\mathcal {I}_8^{(m-1)}
}{\alpha_i^{(m-1),\mathcal {H}^{(m-1)}}\choose 2}+\sum\limits_{i\in
\mathcal {I}_2^{(m-1)}\cup \mathcal {I}_3^{(m-1)}\cup\mathcal
{I}_6^{(m-1)}\cup\mathcal {I}_7^{(m-1)} }{\alpha_i^{(m-1),\mathcal
{H}^{(m-1)}}\choose 2}\\
& & +\sum\limits_{i\in \mathcal {I}_1^{(m-1)}\cup \mathcal
{I}_4^{(m-1)}\cup\mathcal {I}_5^{(m-1)}\cup\mathcal {I}_8^{(m-1)}
}{\beta_i^{(m-1),\mathcal {H}^{(m-1)}}\choose 2} + \sum\limits_{i\in
\mathcal {I}_2^{(m-1)}\cup \mathcal {I}_3^{(m-1)}\cup\mathcal
{I}_6^{(m-1)}\cup\mathcal
{I}_7^{(m-1)} }{\beta_i^{(m-1),\mathcal {H}^{(m-1)}}\choose 2}\\
&=&4\cdot \nu_{\Upsilon_{(m-1)}}(\mathcal
{H}^{(m-1)})+2^{m}\\
& &+\sum\limits_{i\in \mathcal {I}_1^{(m-1)}\cup \mathcal
{I}_4^{(m-1)}\cup\mathcal {I}_5^{(m-1)}\cup\mathcal {I}_8^{(m-1)}
}{1\choose 2}+\sum\limits_{i\in \mathcal {I}_2^{(m-1)}\cup \mathcal
{I}_3^{(m-1)}\cup\mathcal {I}_6^{(m-1)}\cup\mathcal {I}_7^{(m-1)}
}{2\choose 2}\\
& & +\sum\limits_{i\in \mathcal {I}_1^{(m-1)}\cup \mathcal
{I}_4^{(m-1)}\cup\mathcal {I}_5^{(m-1)}\cup\mathcal {I}_8^{(m-1)}
}{1\choose 2} + \sum\limits_{i\in \mathcal {I}_2^{(m-1)}\cup
\mathcal {I}_3^{(m-1)}\cup\mathcal {I}_6^{(m-1)}\cup\mathcal
{I}_7^{(m-1)} }{0\choose 2}\\
&=&4\cdot \nu_{\Upsilon_{(m-1)}}(\mathcal
{H}^{(m-1)})+2^{m}+2^{m}\\
&=&4\cdot \nu_{\Upsilon_{(m-1)}}(\mathcal {H}^{(m-1)})+2^{m+1},\\
\end{array}$$
and thus, the lemma follows immediately.
\end{proof}

\medskip

\begin{lemma}\label{Lemma nu(H,El)}
$$\begin{array}{llll}
\nu_{\Upsilon_{(m)}}(\mathcal {H}^{(m)},\mathcal {E}_{\ell}^{(m)})=
\left \{\begin{array}{llll}
               24, & \mbox{ if } m=2;\\
               128, & \mbox{ if } m=3;\\
               \frac{158}{3}\cdot 4^{m-2}-(8m+\frac{38}{3}+\frac{7\cdot(1+(-1)^{m-1})}{3})\cdot 2^{m-2}, & \mbox{ if } m\geq 4.\\
              \end{array}
           \right. \\
\end{array}$$
\end{lemma}

\begin{proof}  By \eqref{equation ai(m)},
\eqref{equation bi(m)} and Lemma \ref{Lemma counting upsion}, we
have that, for all $m\geq 2$, \begin{eqnarray}
& &\nu_{\Upsilon_{(m)}}(\mathcal {H}^{(m)},\mathcal {E}_{\ell}^{(m)})\nonumber \\
&=&\sum\limits_{i\in \bigcup\limits_{t=1}^4\mathcal
{I}_t^{(m)}}\alpha_i^{(m),\mathcal
{H}^{(m)}}\cdot\mathscr{C}_{+}^{(m)}(\mathcal
{E}_{\ell}^{(m)},z_i^{(m)})+ \sum\limits_{i\in
\bigcup\limits_{t=1}^4\mathcal
{I}_t^{(m)}}\beta_i^{(m),\mathcal {H}^{(m)}}\cdot\mathscr{C}_{-}^{(m)}(\mathcal {E}_{\ell}^{(m)},z_i^{(m)})\nonumber \\
&=&\sum\limits_{i\in \mathcal {I}_1^{(m)}}\alpha_i^{(m),\mathcal
{H}^{(m)}}\cdot\mathscr{C}_{+}^{(m)}(\mathcal
{E}_{\ell}^{(m)},z_i^{(m)})+ \sum\limits_{i\in \mathcal
{I}_2^{(m)}}\alpha_i^{(m),\mathcal {H}^{(m)}}\cdot\mathscr{C}_{+}^{(m)}(\mathcal {E}_{\ell}^{(m)},z_i^{(m)})\nonumber \\
& &+\sum\limits_{i\in \mathcal {I}_3^{(m)}}\alpha_i^{(m),\mathcal
{H}^{(m)}}\cdot\mathscr{C}_{+}^{(m)}(\mathcal
{E}_{\ell}^{(m)},z_i^{(m)})+ \sum\limits_{i\in \mathcal
{I}_4^{(m)}}\alpha_i^{(m),\mathcal {H}^{(m)}}\cdot\mathscr{C}_{+}^{(m)}(\mathcal {E}_{\ell}^{(m)},z_i^{(m)})\nonumber\\
& &+\sum\limits_{i\in \mathcal {I}_1^{(m)}}\beta_i^{(m),\mathcal
{H}^{(m)}}\cdot\mathscr{C}_{-}^{(m)}(\mathcal
{E}_{\ell}^{(m)},z_i^{(m)})+ \sum\limits_{i\in \mathcal
{I}_2^{(m)}}\beta_i^{(m),\mathcal {H}^{(m)}}\cdot\mathscr{C}_{-}^{(m)}(\mathcal {E}_{\ell}^{(m)},z_i^{(m)})\nonumber\\
& &+ \sum\limits_{i\in \mathcal {I}_3^{(m)}}\beta_i^{(m),\mathcal
{H}^{(m)}}\cdot\mathscr{C}_{-}^{(m)}(\mathcal
{E}_{\ell}^{(m)},z_i^{(m)})+ \sum\limits_{i\in \mathcal
{I}_4^{(m)}}\beta_i^{(m),\mathcal {H}^{(m)}}\cdot\mathscr{C}_{-}^{(m)}(\mathcal {E}_{\ell}^{(m)},z_i^{(m)})\nonumber\\
&=&\sum\limits_{i\in \mathcal
{I}_1^{(m)}}\mathscr{C}_{+}^{(m)}(\mathcal
{E}_{\ell}^{(m)},z_i^{(m)})+ 2\cdot\sum\limits_{i\in \mathcal
{I}_2^{(m)}}\mathscr{C}_{+}^{(m)}(\mathcal
{E}_{\ell}^{(m)},z_i^{(m)})\nonumber\\
& &+\ 2\cdot \sum\limits_{i\in \mathcal
{I}_3^{(m)}}\mathscr{C}_{+}^{(m)}(\mathcal
{E}_{\ell}^{(m)},z_i^{(m)})+\sum\limits_{i\in \mathcal
{I}_4^{(m)}}\mathscr{C}_{+}^{(m)}(\mathcal
{E}_{\ell}^{(m)},z_i^{(m)})\nonumber\\
& &+\sum\limits_{i\in \mathcal
{I}_1^{(m)}}\mathscr{C}_{-}^{(m)}(\mathcal
{E}_{\ell}^{(m)},z_i^{(m)})+ \sum\limits_{i\in
\mathcal{I}_4^{(m)}}\mathscr{C}_{-}^{(m)}(\mathcal {E}_{\ell}^{(m)},z_i^{(m)})\nonumber\\
&=&\sum\limits_{i\in \mathcal {I}_1^{(m-1)}\cup \mathcal
{I}_2^{(m-1)} }\mathscr{C}_{-}^{(m-1)}(\mathcal
{E}^{(m-1)},z_i^{(m-1)})+ 2\cdot\sum\limits_{i\in \mathcal
{I}_3^{(m-1)}\cup \mathcal
{I}_4^{(m-1)}}\mathscr{C}_{-}^{(m-1)}(\mathcal {E}^{(m-1)},z_i^{(m-1)})\nonumber\\
& &+\ 2\cdot \sum\limits_{i\in \mathcal {I}_5^{(m-1)}\cup \mathcal
{I}_6^{(m-1)}}\mathscr{C}_{-}^{(m-1)}(\mathcal
{E}^{(m-1)},z_i^{(m-1)}) + \sum\limits_{i\in \mathcal
{I}_7^{(m-1)}\cup \mathcal
{I}_8^{(m-1)}}\mathscr{C}_{-}^{(m-1)}(\mathcal {E}^{(m-1)},z_i^{(m-1)})\nonumber\\
& &+\sum\limits_{i\in \mathcal {I}_1^{(m-1)}\cup \mathcal
{I}_2^{(m-1)}}\mathscr{C}_{+}^{(m-1)}(\mathcal
{E}^{(m-1)},z_i^{(m-1)})+ \sum\limits_{i\in \mathcal
{I}_7^{(m-1)}\cup \mathcal
{I}_8^{(m-1)}}\mathscr{C}_{+}^{(m-1)}(\mathcal {E}^{(m-1)},z_i^{(m-1)})\nonumber\\
&=&2\cdot \big{(}\ \sum\limits_{i\in \mathcal {I}_1^{(m-1)}\cup
\mathcal {I}_2^{(m-1)} }\mathscr{C}_{-}^{(m-1)}(\mathcal
{E}^{(m-1)},z_i^{(m-1)})+ 2\cdot\sum\limits_{i\in \mathcal
{I}_3^{(m-1)}\cup \mathcal
{I}_4^{(m-1)}}\mathscr{C}_{-}^{(m-1)}(\mathcal {E}^{(m-1)},z_i^{(m-1)})\ \big{)}\nonumber\\
& &+\ 2\cdot \sum\limits_{i\in \mathcal {I}_1^{(m-1)}\cup \mathcal
{I}_2^{(m-1)}}\mathscr{C}_{+}^{(m-1)}(\mathcal {E}^{(m-1)},z_i^{(m-1)})\nonumber\\
&=&2\cdot \big{(}\ \sum\limits_{i\in \mathcal
{I}_1^{(m-1)}}\mathscr{C}_{-}^{(m-1)}(\mathcal
{E}^{(m-1)},z_i^{(m-1)})+ \sum\limits_{i\in \mathcal
{I}_2^{(m-1)}}\mathscr{C}_{-}^{(m-1)}(\mathcal {E}^{(m-1)},z_i^{(m-1)})\ \big{)}\nonumber\\
& &+\ 4\cdot\big{(}\ \sum\limits_{i\in \mathcal {I}_3^{(m-1)}}\mathscr{C}_{-}^{(m-1)}(\mathcal {E}^{(m-1)},z_i^{(m-1)})+\sum\limits_{i\in \mathcal {I}_4^{(m-1)}}\mathscr{C}_{-}^{(m-1)}(\mathcal {E}^{(m-1)},z_i^{(m-1)})\ \big{)}\nonumber\\
& &+\ 2\cdot \big{(}\ \sum\limits_{i\in \mathcal
{I}_1^{(m-1)}}\mathscr{C}_{+}^{(m-1)}(\mathcal
{E}^{(m-1)},z_i^{(m-1)})+\sum\limits_{i\in \mathcal
{I}_2^{(m-1)}}\mathscr{C}_{+}^{(m-1)}(\mathcal
{E}^{(m-1)},z_i^{(m-1)})\ \big{)}.
\end{eqnarray}
Then the lemma follows from Lemma \ref{Lemma three conclusions on C+
and C-} and trivial verifications.
\end{proof}

\begin{lemma}\label{Lemma crossing of E-black}
For $m\geq 3$,
$$\nu_{\Upsilon_{(m)}}(\mathcal {E}^{(m)})=\frac{194}{3}\cdot4^{m-1}-(4m^2+23m+\frac{101}{3}-\frac{7\cdot(1+(-1)^m)}{6})\cdot2^{m-1}.$$
\end{lemma}

\begin{proof}
By Lemma \ref{lemma counting of crossings}, Lemma \ref{Lemma
counting upsion} and Lemma \ref{Lemma counting nu H(m)}, we have
that for all $k\geq 1$,
\begin{eqnarray}\label{equation deductive for nu(E(i))}
\nu_{\Upsilon_{(k)}}(\mathcal {E}^{(k)})
&=&\nu_{\Upsilon_{(k)}}(\mathcal
{E}_{\ell}^{(k)})+\nu_{\Upsilon_{(k)}}(\mathcal {E}_{r}^{(k)})
+\nu_{\Upsilon_{(k)}}(\mathcal {H}^{(k)},\mathcal {E}_{\ell}^{(k)})
+\nu_{\Upsilon_{(k)}}(\mathcal {H}^{(k)},\mathcal {E}_{r}^{(k)})+\nu_{\Upsilon_{(k)}}(\mathcal {H}^{(k)})\nonumber\\
&=&2\cdot \big{(}\ \nu_{\Upsilon_{(k)}}(\mathcal
{E}_{\ell}^{(k)})+\nu_{\Upsilon_{(k)}}(\mathcal {H}^{(k)},\mathcal
{E}_{\ell}^{(k)})\ \big{)}
+\nu_{\Upsilon_{(k)}}(\mathcal {H}^{(k)})\nonumber\\
&=&2\cdot \big{(}\ \nu_{\Upsilon_{(k-1)}}(\mathcal
{E}^{(k-1)})+\nu_{\Upsilon_{(k)}}(\mathcal {H}^{(k)},\mathcal
{E}_{\ell}^{(k)})\ \big{)}
+\nu_{\Upsilon_{(k)}}(\mathcal {H}^{(k)})\nonumber\\
&=&2\cdot \nu_{\Upsilon_{(k-1)}}(\mathcal {E}^{(k-1)})+2\cdot
\nu_{\Upsilon_{(k)}}(\mathcal {H}^{(k)},\mathcal {E}_{\ell}^{(k)})
+(6\cdot 4^{k-1}-2^{k+1}).
\end{eqnarray}

Now we shall prove this lemma by induction on $m$. We first consider the case when
$m=3$. Observe
\begin{equation}\label{equation nu(E(1))=4}
\nu_{\Upsilon_{(1)}}(\mathcal {E}^{(1)})=4.
\end{equation}
By  \eqref{equation deductive for nu(E(i))}, \eqref{equation
nu(E(1))=4} and Lemma \ref{Lemma nu(H,El)},  we have that
$$\begin{array}{llll}\label{equation nu(E(3))}
\nu_{\Upsilon_{(3)}}(\mathcal {E}^{(3)})&=&2\cdot
\nu_{\Upsilon_{(2)}}(\mathcal {E}^{(2)})+2\cdot 128 +(6\cdot
4^{2}-2^{4})\\
&=&2\cdot \big{(}\ 2\cdot \nu_{\Upsilon_{(1)}}(\mathcal
{E}^{(1)})+2\cdot 24+6\cdot
4-2^3\ \big{)}+336\\
&=& 480\\
&=&\frac{194}{3}\cdot4^{3-1}-(4\cdot 3^2+23\cdot
3+\frac{101}{3}-\frac{7\cdot(1+(-1)^3)}{6})\cdot2^{3-1},
\end{array}$$
we are done. Hence, we need to consider the case when $m>3$. By \eqref{equation deductive
for nu(E(i))} and Lemma \ref{Lemma nu(H,El)},  we have that
$$\begin{array}{llll}
\nu_{\Upsilon_{(m)}}(\mathcal {E}^{(m)}) &=&2\cdot
\nu_{\Upsilon_{(m-1)}}(\mathcal {E}^{(m-1)})+2\cdot
\nu_{\Upsilon_{(m)}}(\mathcal {H}^{(m)},\mathcal {E}_{\ell}^{(m)})
+(6\cdot
4^{m-1}-2^{m+1})\\
&=&2\cdot \big{(}\
\frac{194}{3}\cdot4^{m-2}-(4(m-1)^2+23(m-1)+\frac{101}{3}-\frac{7\cdot(1+(-1)^{m-1})}{6})\cdot2^{m-2}\ \big{)}\\
& &+\ 2\cdot \big{(}\ \frac{158}{3}\cdot
4^{m-2}-(8m+\frac{38}{3}+\frac{7\cdot(1+(-1)^{m-1})}{3})\cdot
2^{m-2}\ \big{)} +(6\cdot
4^{m-1}-2^{m+1})\\
&=&\frac{194}{3}\cdot4^{m-1}-(4m^2+23m+\frac{101}{3}-\frac{7\cdot(1+(-1)^m)}{6})\cdot2^{m-1}.
\end{array}$$
This completes the proof of the lemma.
\end{proof}

\section{Proof of Theorem \ref{Theorem Upper Bound for general n}}

We begin this section by introducing a couple of necessary notations.\\
Let
$$\begin{array}{llll}
E_{black}^n&=&\bigcup\limits_{i=1}^{4}E(U_i^n\cup V_i^n),\\
E_{red}^n&=&E[U_1,U_2]\cup E[U_3,U_4]\cup E[V_1,V_2]\cup E[V_3,V_4],\\
E_{blue}^n&=&E[U_1,U_3]\cup E[U_2,U_4]\cup E[V_1,V_3]\cup E[V_2,V_4]\cup E[U_1,V_3]\cup E[U_2,V_4]\cup E[V_1,U_3]\cup E[V_2,U_4].\\
\end{array}$$
By Lemma \ref{Lemma enlarging the edges}, we have the following
decomposition \begin{equation} \label{equation decomposition of
edges of AQn} E(AQ_n)=E_{black}^n\cup E_{red}^n\cup E_{blue}^n
\end{equation}
with
\begin{equation}\label{equation Eb Er Eb}
\begin{array}{llll}
E_{black}^n&=\{e\in E(AQ_n): {\rm Dim}(e)\in [-(n-1),-2]\cup [3,n-1]\},\\
E_{red}^n&=\{e\in E(AQ_n): {\rm Dim}(e)\in \{1,2\}\},\\
E_{blue}^n&=\{e\in E(AQ_n): {\rm Dim}(e)\in \{-n,n\}\}.\\
\end{array}
\end{equation}

Let $\mathscr{R}\subseteq \mathbb{R}\times \mathbb{R}$ be a region,
and let $F$ be an edge subset of $E(AQ_n)$. We denote by ${\rm
In}_{\mathscr{R}}(F)$ the part of $F$ which is drawn at the inner of
the region $\mathscr{R}$, and by ${\rm Out}_{\mathscr{R}}(F)$ the
part of $F$ which is drawn outer the region $\mathscr{R}$. If
$F=\{e\}$ is a singleton, we simply write ${\rm
In}_{\mathscr{R}}(e)$ and ${\rm Out}_{\mathscr{R}}(e)$ for ${\rm
In}_{\mathscr{R}}(\{e\})$ and ${\rm Out}_{\mathscr{R}}(\{e\})$
respectively.

\noindent {\sl $\bullet$ For convenience, in the rest of this paper,
we shall write ${\rm In}(\cdot)$ to mean ${\rm
In}_{\mathbb{R}_{-5}^0\times \mathbb{R}_{5/2}^5}(\cdot)$ for short.}


Now we are in a position to prove Theorem \ref{Theorem Upper Bound
for general n}. For the convenience of readers, we split the proof
into two parts, which are put into Subsection 3.1 and Subsection
3.2, respectively. In Subsection 3.1, we give a good drawing
$\Gamma_n$ of $cr(AQ_n)$ for all $n\geq 8$. In Subsection 3.2, we
shall verify that the crossings of $\Gamma_n$ is no more than the
upper bound in Theorem \ref{Theorem Upper Bound for general n}, and
therefore completes the proof of Theorem \ref{Theorem Upper Bound
for general n}.

\subsection{The drawing $\Gamma_n$ for $n\geq 8$}

We first give the drawing of $E_{black}^n$ under $\Gamma_n$. In
general, $E_{black}^n$ are drawn at the inner of the region
$\mathbb{R}_{-5}^{5}\times \mathbb{R}_{-5}^{5}$. Let
$$\mathcal {W}^{n}\in
\{U_1^n,U_2^n,U_3^n,U_4^n,V_1^n,V_2^n,V_3^n,V_4^n\}.$$ Take
$\mathcal{N}=5$ and
\begin{equation}\label{equation value of z1-z4}
\begin{array}{llll} &
(z_1^{\mathcal{N}+(0)},z_2^{\mathcal{N}+(0)},z_3^{\mathcal{N}+(0)},z_4^{\mathcal{N}+(0)})=\left
\{\begin{array}{llll}
               (00100, 00011, 00111, 00000),  &  \mbox{ if } \ \ \mathcal {W}^{n}=U_1^n;\\
               (00001, 00110, 00010, 00101),  &  \mbox{ if } \ \ \mathcal {W}^{n}=U_2^n;\\
               (10100, 10011, 10111, 10000),  &  \mbox{ if } \ \ \mathcal {W}^{n}=U_3^n;\\
               (10101, 10010, 10110, 10001),  &  \mbox{ if } \ \ \mathcal {W}^{n}=U_4^n;\\
               (01111, 01000, 01100, 01011),  &  \mbox{ if } \ \ \mathcal {W}^{n}=V_1^n;\\
               (01010, 01101, 01001, 01110),  &  \mbox{ if } \ \ \mathcal {W}^{n}=V_2^n;\\
               (11111, 11000, 11100, 11011),  &  \mbox{ if } \ \ \mathcal {W}^{n}=V_3^n;\\
               (11110, 11001, 11101, 11010),  &  \mbox{ if } \ \ \mathcal {W}^{n}=V_4^n.\\
                             \end{array}
           \right. \\
\end{array}
\end{equation}
It is easy to verify that the above values of
$z_1^{\mathcal{N}+(0)},z_2^{\mathcal{N}+(0)},z_3^{\mathcal{N}+(0)},z_4^{\mathcal{N}+(0)}$
satisfy \eqref{equation Initial condition 1} and \eqref{equation
Initial condition 2} with $t_1=3=\mathcal{N}-2$ and
$t_2=-2=-(\mathcal{N}-3)$. Hence, we can draw $\langle\mathcal
{W}^{n}\rangle=\Omega^{(n-\mathcal{N})}(\{z_1^{\mathcal{N}+(0)},z_2^{\mathcal{N}+(0)},z_3^{\mathcal{N}+(0)},z_4^{\mathcal{N}+(0)}\})$
as the drawing $\Upsilon_{(n-5)}$ with the initial positive order as
$(z_1^{\mathcal{N}+(0)},z_2^{\mathcal{N}+(0)},z_3^{\mathcal{N}+(0)},z_4^{\mathcal{N}+(0)})$
and satisfying that
\begin{equation}\label{equation location X}
\begin{array}{llll} & X_{z_1^{(0)}}=X_{z_2^{(0)}}=X_{z_3^{(0)}}=X_{z_4^{(0)}}=\left \{\begin{array}{llll}
               -2,  &  \mbox{ if } \ \ \mathcal {W}^{n}\in \{U_1^n,U_3^n\};\\
               -1,  &  \mbox{ if } \ \ \mathcal {W}^{n}\in\{V_1^n,V_3^n\};\\
               1,  &  \mbox{ if } \ \ \mathcal {W}^{n}\in\{V_2^n,V_4^n\};\\
               2,  &  \mbox{ if } \ \ \mathcal {W}^{n}\in\{U_2^n,U_4^n\},\\
                             \end{array}
           \right. \\
\end{array}
\end{equation}
and
\begin{equation}\label{equation location Y}
\begin{array}{llll} & (Y_{z_1^{(0)}},Y_{z_2^{(0)}},Y_{z_3^{(0)}},Y_{z_4^{(0)}})=\left \{\begin{array}{llll}
               (1,2,3,4),  &  \mbox{ if } \ \ \mathcal {W}^{n}\in \{U_1^n,V_2^n\};\\
               (4,3,2,1),  &  \mbox{ if } \ \ \mathcal {W}^{n}\in \{V_1^n,U_2^n\};\\
               (-4,-3,-2,-1),  &  \mbox{ if } \ \ \mathcal {W}^{n}\in \{U_3^n,V_4^n\};\\
               (-1,-2,-3,-4),  &  \mbox{ if } \ \ \mathcal {W}^{n}\in \{V_3^n,U_4^n\}.\\
                             \end{array}
           \right. \\
\end{array}
\end{equation}
We remark that the edges $z_i^{(n-5)}z_j^{(n-5)}\in
\mathcal{E}^{(n-5)}$ will be drawn to be at the left or at the right
of the lines $x=\pm 1, \pm 2$ dependent on the locations of
$z_1^{\mathcal{N}+(0)},z_2^{\mathcal{N}+(0)},z_3^{\mathcal{N}+(0)},z_4^{\mathcal{N}+(0)}$
and $\mathcal {O}(z_i^{(n-5)}z_j^{(n-5)})$ in $\Upsilon_{(n-5)}$.

To proceed with this subsection, we need to rename the vertices of
$\mathcal {W}^{n}$ as follows:
$$\begin{array}{llll} & \mathcal
{W}^{n}=\left \{\begin{array}{llll}
               \{u_{1,1}^n,u_{1,2}^n,\ldots,u_{1,2^{n-3}}^n\},  &  \mbox{ if } \ \ \mathcal {W}^{n}=U_1^n;\\
               \{u_{2,1}^n,u_{2,2}^n,\ldots,u_{2,2^{n-3}}^n\},  &  \mbox{ if } \ \ \mathcal {W}^{n}=U_2^n;\\
               \{u_{3,1}^n,u_{3,2}^n,\ldots,u_{3,2^{n-3}}^n\},  &  \mbox{ if } \ \ \mathcal {W}^{n}=U_3^n;\\
               \{u_{4,1}^n,u_{4,2}^n,\ldots,u_{4,2^{n-3}}^n\},  &  \mbox{ if } \ \ \mathcal {W}^{n}=U_4^n;\\
               \{v_{1,1}^n,v_{1,2}^n,\ldots,v_{1,2^{n-3}}^n\},  &  \mbox{ if } \ \ \mathcal {W}^{n}=V_1^n;\\
               \{v_{2,1}^n,v_{2,2}^n,\ldots,v_{2,2^{n-3}}^n\},  &  \mbox{ if } \ \ \mathcal {W}^{n}=V_2^n;\\
               \{v_{3,1}^n,v_{3,2}^n,\ldots,v_{3,2^{n-3}}^n\},  &  \mbox{ if } \ \ \mathcal {W}^{n}=V_3^n;\\
               \{v_{4,1}^n,v_{4,2}^n,\ldots,v_{4,2^{n-3}}^n\},  &  \mbox{ if } \ \ \mathcal {W}^{n}=V_4^n;\\                    \end{array}
           \right. \\
\end{array}$$
with
\begin{equation}\label{equation order of Yu}
|Y_{u_{i,1}^n}|>|Y_{u_{i,2}^n}|>\cdots >|Y_{u_{i,2^{n-3}}^n}|
\end{equation}
 and
\begin{equation}\label{equation order of Yv}
|Y_{v_{i,1}^n}|>|Y_{v_{i,2}^n}|>\cdots >|Y_{v_{i,2^{n-3}}^n}|
\end{equation} where $i\in [1,4]$.

Note that
\begin{equation}\label{equation Yu=Yv}
|Y_{u_{1,j}^n}|=|Y_{v_{1,j}^n}|=|Y_{u_{2,j}^n}|=|Y_{v_{2,j}^n}|=|Y_{u_{3,j}^n}|=|Y_{v_{3,j}^n}|=|Y_{u_{4,j}^n}|=|Y_{v_{4,j}^n}|
\end{equation}
 for all $j\in
[1,2^{n-3}]$.

By \eqref{equation zi enlarging two} and \eqref{equation zi
enlarging pi}, we derive that for any $n\geq 8$,
\begin{equation}\label{equation orginial for uij}
u_{i,j}^n\in
\Omega^{(n-8)}(\{u_{i,\lceil\frac{j}{2^{n-8}}\rceil}^8\})
\end{equation}
and
\begin{equation}\label{equation orginial for vij}
v_{i,j}^n\in
\Omega^{(n-8)}(\{v_{i,\lceil\frac{j}{2^{n-8}}\rceil}^8\}),
\end{equation}
where $i\in [1,4]$ and $j\in [1,2^{n-3}]$.

\medskip

By applying Lemma \ref{Lemma enlarging the edges} repeatedly, we
have that
\begin{equation}\label{equation Dim -(n-1)}
\{e\in E(AQ_n):{\rm Dim}(e)=-(n-1)\}=\{u_{i,j}^n v_{i,j}^n: i\in
[1,4],j\in [1,2^{n-3}]\}
\end{equation}
and \begin{equation} \label{equation Dim (n-1)} \{e\in E(AQ_n):{\rm
Dim}(e)=(n-1)\}=\{u_{i,j}^n v_{i,j+(-1)^{j-1}}^n: i\in [1,4],j\in
[1,2^{n-3}]\}.
\end{equation}
By \eqref{equation order of Yu}, \eqref{equation order of Yv},
\eqref{equation Yu=Yv}, \eqref{equation Dim -(n-1)} and
\eqref{equation Dim (n-1)}, we can draw the corresponding edges in
$\{e\in E(AQ_n):{\rm Dim}(e)\in\{-(n-1),n-1\}\}$ to be straight
lines. Therefore, this completes the characterization of the drawing
of $E_{black}^n$.

\bigskip

By applying \eqref{equaiton gm dim -(N+m-2)}, we see that
\begin{eqnarray}\label{equation Dim -(n-2)}
\{e\in E(AQ_n): {\rm Dim}(e)=-(n-2)\}&=&\{u_{i,2j-1}^n u_{i,2j}^n:
i\in [1,4], j\in [1,2^{n-4}]\}\nonumber\\
& &\cup \ \{v_{i,2j-1}^n v_{i,2j}^n: i\in [1,4], j\in [1,2^{n-4}]\}.
\end{eqnarray}

\bigskip

Now we give the drawing of $E_{red}^n$ under $\Gamma_n$. In general,
$E_{red}^n$ are drawn at the inner of the region
$\mathbb{R}_{-5}^{5}\times \mathbb{R}_{-5}^{5}$.

By \eqref{equation zi enlarging two} and \eqref{equation zi
enlarging pi}, applying  Lemma \ref{Lemma enlarging the edges}
repeatedly, we have that
\begin{eqnarray} \label{equation Dim 1}
\{e\in E(AQ_n):{\rm Dim}(e)=1\}&=&\{u_{1,j}^n u_{2,j}^n: j\in
[1,2^{n-3}]\}\cup \{v_{1,j}^n v_{2,j}^n: j\in [1,2^{n-3}]\}\nonumber\\
& &\cup \ \{u_{3,j}^n u_{4,j}^n: j\in [1,2^{n-3}]\}\cup \{v_{3,j}^n
v_{4,j}^n: j\in [1,2^{n-3}]\}
\end{eqnarray}
and
\begin{eqnarray} \label{equation Dim 2}
\{e\in E(AQ_n):{\rm Dim}(e)=2\}&=&\{u_{1,j}^n u_{2,j+2^{n-4}}^n:
j\in [1,2^{n-3}]\}\cup \{v_{1,j}^n v_{2,j+2^{n-4}}^n: j\in
[1,2^{n-3}]\}\nonumber\\
& &\cup \ \{u_{3,j}^n u_{4,j+2^{n-4}}^n: j\in [1,2^{n-3}]\}\cup
\{v_{3,j}^n v_{4,j+2^{n-4}}^n: j\in [1,2^{n-3}]\},
\end{eqnarray}
where the subscripts $j+2^{n-4}$ are taken to be the least positive
residues modulo $2^{n-3}$.

In general, by \eqref{equation Eb Er Eb}, \eqref{equation order of
Yu}, \eqref{equation order of Yv}, \eqref{equation Yu=Yv},
\eqref{equation Dim 1} and \eqref{equation Dim 2}, we can draw
$E_{red}^n$ satisfying Conditions (i), (ii) and (iii).

(i)  The drawing of $E_{red}^n$ is symmetric with respect to
$Y$-axis;

(ii) The drawing of $E_{red}^n$ is symmetric with respect to
$X$-axis;

(iii) ${\rm In}_{\mathbb{R}_{-5}^5\times
\mathbb{R}_{0}^5}(E_{red}^n)$ is symmetric with respect to
$y=\frac{5}{2}$.

Take an arbitrary edge $e\in E_{red}^n$. By \eqref{equation Dim 1},
\eqref{equation Dim 2} and Conditions (i)-(iii), we may consider
only the case
$$e\in \bigcup\limits_ {t\in \{1,2\}}\mathscr{I}_t(U_{1,1}^n\cup
V_{1,1}^n).$$ Then we draw $e$ according to the following inductive
rule.

\textbf{Inductive rule for the drawing of $E_{red}^n$:} \ \ {\sl Let
$j$ be an arbitrary integer of $[1,2^{n-4}]$. By \eqref{equation
orginial for uij} and \eqref{equation orginial for vij}, we can draw
the arcs ${\rm In}(\mathscr{I}_1(u_{1,j}^n))$, ${\rm
In}(\mathscr{I}_1(v_{1,j}^n))$, ${\rm In}(\mathscr{I}_2(u_{1,j}^n))$
and ${\rm In}(\mathscr{I}_2(v_{1,j}^n))$ in $\Gamma_n$ to be along
the original ways of ${\rm
In}(\mathscr{I}_1(u_{1,{\lceil\frac{j}{2^{n-8}}\rceil}}^8))$, ${\rm
In}(\mathscr{I}_1(v_{1,{\lceil\frac{j}{2^{n-8}}\rceil}}^8))$, ${\rm
In}(\mathscr{I}_2(u_{1,{\lceil\frac{j}{2^{n-8}}\rceil}}^8))$ and
${\rm In}(\mathscr{I}_2(v_{1,{\lceil\frac{j}{2^{n-8}}\rceil}}^8))$
in $\Gamma_8$, respectively.}

Therefore, this completes the characterization of the drawing of
$E_{red}^n$.

\bigskip

Next, we give the drawing of $E_{blue}^n$ under $\Gamma_n$.

By \eqref{equation zi enlarging two} and \eqref{equation zi
enlarging pi}, applying  Lemma \ref{Lemma enlarging the edges}
repeatedly, we have that
\begin{eqnarray} \label{equation Dim -n}
\{e\in E(AQ_n):{\rm Dim}(e)=-n\}&=&\{u_{1,j}^n v_{3,j}^n: j\in
[1,2^{n-3}]\}\cup \{v_{1,j}^n u_{3,j}^n: j\in
[1,2^{n-3}]\}\nonumber\\
& &\cup \ \{u_{2,j}^n v_{4,j}^n: j\in [1,2^{n-3}]\}\cup \{v_{2,j}^n
u_{4,j}^n: j\in [1,2^{n-3}]\}
\end{eqnarray}
and
\begin{eqnarray} \label{equation Dim n}
\{e\in E(AQ_n):{\rm Dim}(e)=n\}&=&\{u_{1,j}^n u_{3,j}^n: j\in
[1,2^{n-3}]\}\cup \{v_{1,j}^n v_{3,j}^n: j\in
[1,2^{n-3}]\}\nonumber\\
& &\cup \ \{u_{2,j}^n u_{4,j}^n: j\in [1,2^{n-3}]\}\cup \{v_{2,j}^n
v_{4,j}^n: j\in [1,2^{n-3}]\}.
\end{eqnarray}

In general, by \eqref{equation Eb Er Eb}, \eqref{equation order of
Yu}, \eqref{equation order of Yv}, \eqref{equation Yu=Yv},
\eqref{equation Dim -n} and \eqref{equation Dim n}, we can draw
$E_{blue}^n$ satisfying Conditions (iv), (v) and (vi).

(iv) The drawing of $E_{blue}^n$ is symmetric with respect to
$Y$-axis;

(v) The drawing of $E_{blue}^n$ is symmetric with respect to
$X$-axis;

(vi) ${\rm In}_{\mathbb{R}_{-5}^5\times
\mathbb{R}_{0}^5}(E_{blue}^n)$ is symmetric with respect to
$y=\frac{5}{2}$.

Take an arbitrary edge $e\in E_{blue}^n$. By \eqref{equation Dim
-n}, \eqref{equation Dim n} and Condition (iv), we may assume $e\in
\bigcup\limits_ {t\in \{-n,n\}}\mathscr{I}_t(U_1^n\cup V_1^n).$
 We draw $e$ to be consisting of three nonempty parts, say
${\rm In}_{\mathbb{R}_{-5}^{5}\times \mathbb{R}_{0}^{5}}(e)$, ${\rm
In}_{\mathbb{R}_{-5}^{5}\times \mathbb{R}_{-5}^{0}}(e)$ and ${\rm
Out}_{\mathbb{R}_{-5}^{5}\times \mathbb{R}_{-5}^{5}}(e)$, of which
${\rm In}_{\mathbb{R}_{-5}^{5}\times \mathbb{R}_{0}^{5}}(e)$ and
${\rm In}_{\mathbb{R}_{-5}^{5}\times \mathbb{R}_{-5}^{0}}(e)$ `meet'
the line $x=-5$ at right angles, and of which ${\rm
Out}_{\mathbb{R}_{-5}^{5}\times \mathbb{R}_{-5}^{5}}(e)$ connects
the two `points' which are `produced' by ${\rm
In}_{\mathbb{R}_{-5}^{5}\times \mathbb{R}_{0}^{5}}(e)$ and  ${\rm
In}_{\mathbb{R}_{-5}^{5}\times \mathbb{R}_{-5}^{0}}(e)$ with the
line $x=-5$.

It remains to show the drawing of the other two parts of $e$, that
is, ${\rm In}_{\mathbb{R}_{-5}^{5}\times \mathbb{R}_{0}^{5}}(e)$ and
${\rm In}_{\mathbb{R}_{-5}^{5}\times \mathbb{R}_{-5}^{0}}(e)$. By
Conditions (v) and (vi), i.e., the above two kinds of symmetries of the drawing of $E_{blue}^n$, we can suppose without loss of generality that
$$e\in
\bigcup\limits_ {t\in \{-n,n\}}\mathscr{I}_t(U_{1,1}^n\cup
V_{1,1}^n),$$ and draw ${\rm In}(e)$ satisfying the following
inductive rule.

\textbf{Inductive rule for the drawing of $E_{blue}^n$:} \ \ {\sl
Let $j$ be an arbitrary integer of $[1,2^{n-4}]$. By \eqref{equation
orginial for uij} and \eqref{equation orginial for vij}, we can draw
the arc ${\rm In}(\mathscr{I}_n(v_{1,j}^n))$ in $\Gamma_n$ to be
along the original way of ${\rm
In}(\mathscr{I}_n(v_{1,{\lceil\frac{j}{2^{n-8}}\rceil}}^8))$ in
$\Gamma_8$, and we can draw the three arcs ${\rm
In}(\mathscr{I}_n(u_{1,j}^n))$, ${\rm
In}(\mathscr{I}_{-n}(u_{1,j}^n))$ and ${\rm
In}(\mathscr{I}_{-n}(v_{1,j}^n))$ to be {\rm a bunch} of three arcs
in $\Gamma_n$ which is along the line $y=Y_{u_{1,j}^n}$ and
satisfies Property A.}

\noindent \textbf{Property A:}\ \ The arc ${\rm
In}(\mathscr{I}_{-n}(u_{1,j}^n))$ will be drawn {\sl precisely} at
the line $y=Y_{u_{1,j}^n}$, one of ${\rm
In}(\mathscr{I}_n(u_{1,j}^n))$ and ${\rm
In}(\mathscr{I}_{-n}(v_{1,j}^n))$ lying flat above the line
$y=Y_{u_{1,j}^n}$ and another below the line $y=Y_{u_{1,j}^n}$. In
detailed, for $j\in [1,2^{n-6}-1]\cup
[2^{n-6}+1,2^{n-6}+2^{n-7}-1]\cup [2^{n-6}+2^{n-7}+1,2^{n-5}-1]\cup
[2^{n-5}+1,2^{n-5}+2^{n-7}-1]\cup
[2^{n-5}+2^{n-7}+1,2^{n-5}+2^{n-6}-1]\cup \{2^{n-5}+2^{n-6}+1\}$,
${\rm In}(\mathscr{I}_{-n}(v_{1,j}^n))$ is drawn above the line
$y=Y_{u_{1,j}^n}$, for otherwise, ${\rm
In}(\mathscr{I}_{-n}(v_{1,j}^n))$ is drawn below the line
$y=Y_{u_{1,j}^n}$.

Therefore, this completes the characterization of the drawing of
$E_{blue}^n$.

\bigskip

Combined \eqref{equation value of z1-z4}, \eqref{equation location
X}, \eqref{equation location Y} and Conditions (i)-(vi), we conclude
that $\Gamma_n$ has Property B.

\noindent \textbf{Property B:}\ \

$\bullet$ $\Gamma_n$ is symmetric with respect to $Y$-axis;

$\bullet$  $\Gamma_n$ is symmetric with respect to $X$-axis;

$\bullet$ ${\rm In}_{\mathbb{R}_{-5}^5\times
\mathbb{R}_{0}^5}(E(AQ_n))$ is symmetric with respect to
$y=\frac{5}{2}$.

\medskip

This completes the characterizations of the drawing $\Gamma_n$. To
make the above process clear, we give the drawings of ${\rm
In}(AQ_n)$ under $\Gamma_n$ for $n=8,9$ in Figure 3.1 and 3.2,
respectively.

\noindent $\bullet$ For the clearness of composition, in the rest of
this paper, any vertex $a=a_na_{n-1}\cdots a_1\in V(AQ_n)$ in
figures will be represented by the corresponding decimal number
$a_n\cdot 2^{n-1}+a_{n-1}\cdot 2^{n-2}+\cdots +a_1$.

\begin{figure}
\centering
\includegraphics[scale=1.0]{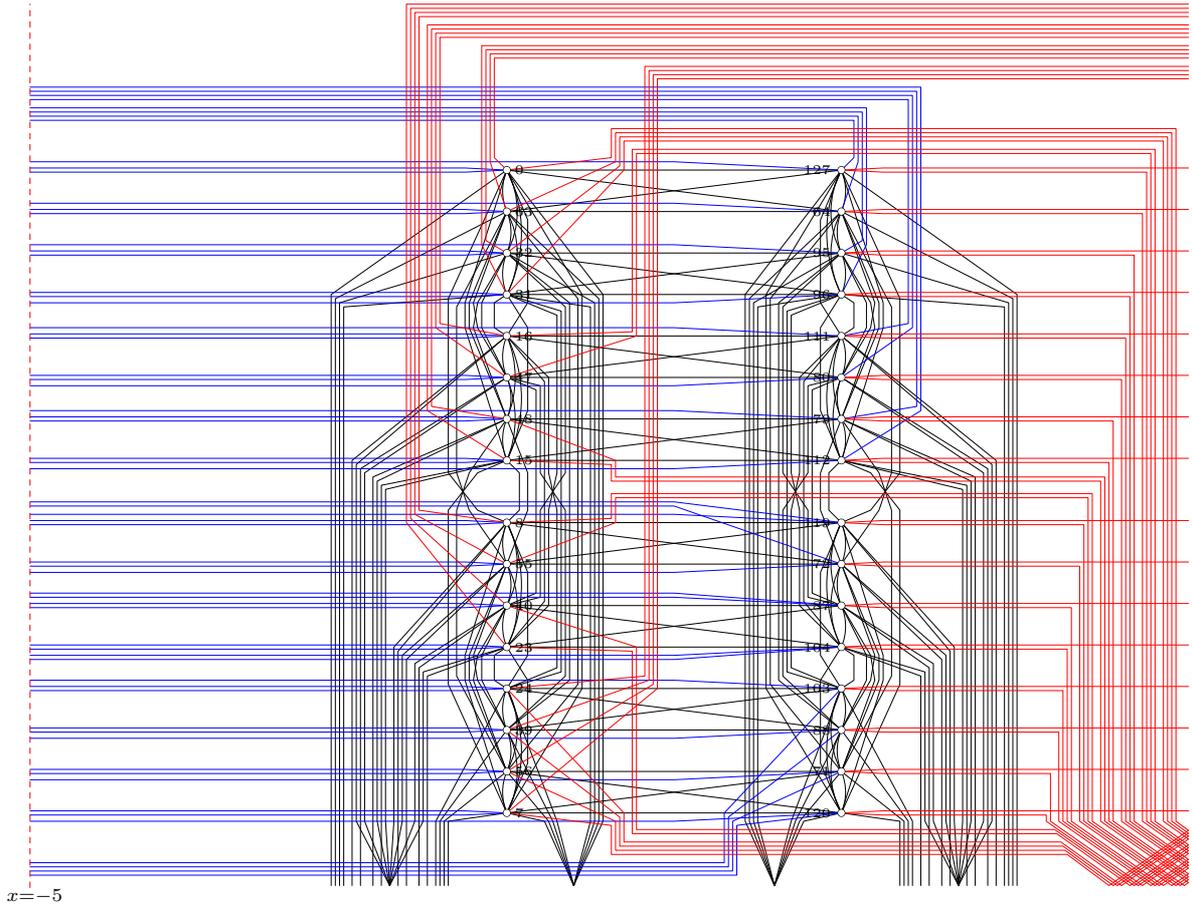}
\caption{\small{${\rm In}(AQ_8)$ in the drawing $\Gamma_8$}}
\end{figure}

\begin{figure}
\hspace{-10pt}
\includegraphics[scale=1.0]{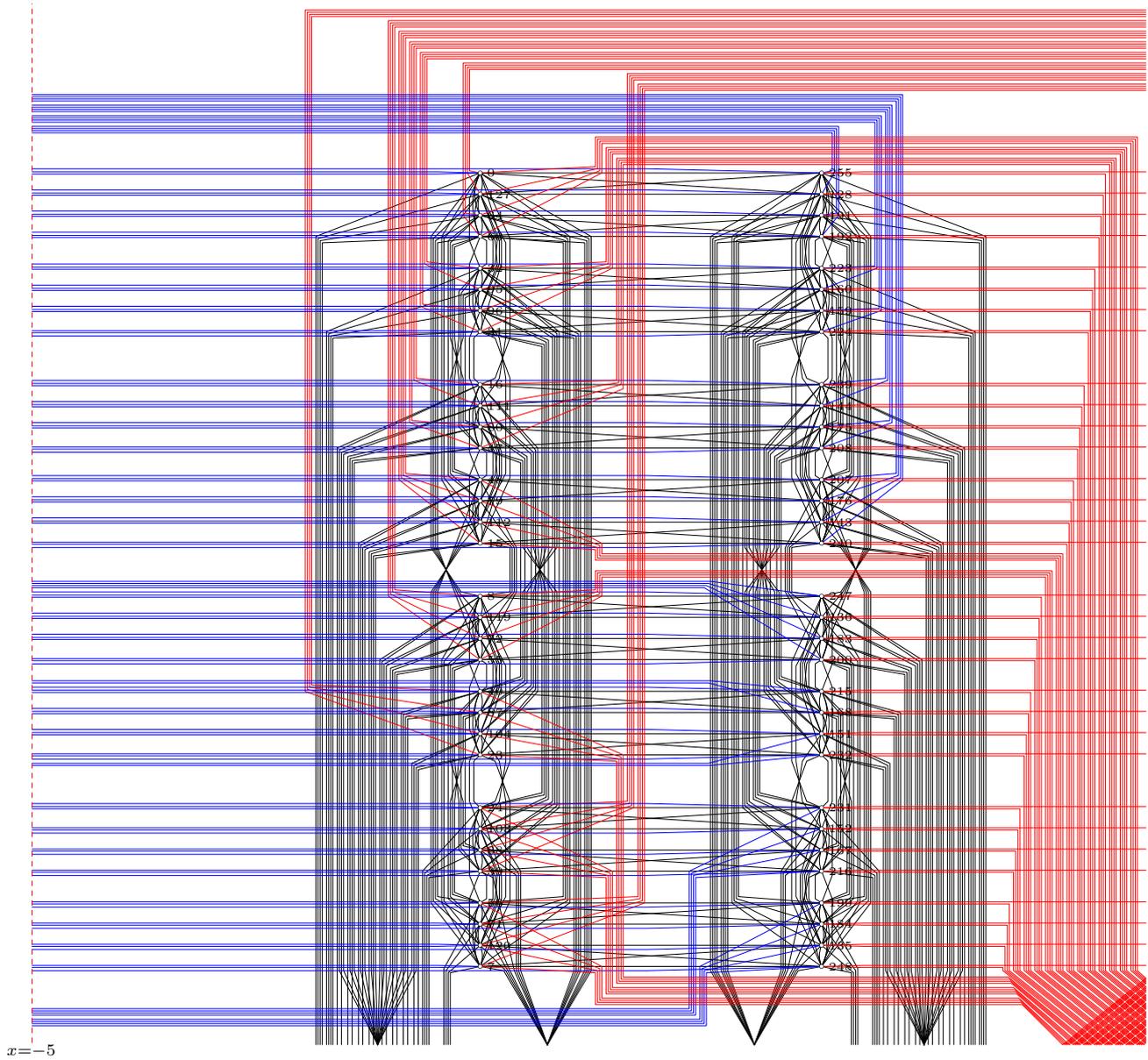}
\caption{\small{${\rm In}(AQ_9)$ in the drawing $\Gamma_9$}}
\end{figure}

\newpage

\subsection{The calculation of crossings in $\Gamma_n$}

In this subsection, we shall calculate the number of crossings of
$AQ_n$ in the drawing $\Gamma_n$ for $n\geq 8$. By \eqref{equation
decomposition of edges of AQn} and Lemma \ref{lemma counting of
crossings}, we have that
\begin{eqnarray} \label{equation crossings of AQn}
\nu_{\Gamma_n}(AQ_n)&=&\nu_{\Gamma_n}(E_{blue}^n)+\nu_{\Gamma_n}(E_{red}^n)+\nu_{\Gamma_n}(E_{black}^n) \nonumber\\
                    & &+\ \nu_{\Gamma_n}(E_{blue}^n, E_{red}^n)+\nu_{\Gamma_n}(E_{blue}^n, E_{black}^n)+\nu_{\Gamma_n}(E_{red}^n, E_{black}^n).
\end{eqnarray}
By \eqref{equation crossings of AQn}, the process of calculation can
be divided into six parts. To begin with the six parts, we  need
some preliminaries.

For $i\in [1,4]$, let
$$\begin{array}{rlll}
U_{i,1}^n&=&\{u_{i,1}^n,u_{i,2}^n,\ldots,u_{i,2^{n-4}}^n\},\\
U_{i,2}^n&=&\{u_{i,2^{n-4}+1}^n,u_{i,2^{n-4}+2}^n,\ldots,u_{i,2^{n-3}}^n\},\\
V_{i,1}^n&=&\{v_{i,1}^n,v_{i,2}^n,\ldots,v_{i,2^{n-4}}^n\},\\
V_{i,2}^n&=&\{v_{i,2^{n-4}+1}^n,v_{i,2^{n-4}+2}^n,\ldots,v_{i,2^{n-3}}^n\}.\\
\end{array}$$

For any edge $e\in \bigcup\limits_{t\in
\{-n,n\}}\mathscr{I}_t(V_{1,1}^n)$, let $Y^{e}$ be the
$Y$-coordinate of the `point' at which ${\rm In}(e)$ and the line
$x=-2$ cross each other, and let $\mathcal {L}({\rm In}(e))$ be the
least positive integer $j\in [1,2^{n-4}+1]$ such that
$Y^{e}>Y_{u_{1,j}^n}$.

Let $$H_{red}^n=\bigcup\limits_{t\in
[1,2]}\mathscr{I}_t(U_{1,1}^n)$$ and
$$H_{black}^n=\bigcup\limits_{t\in [-(n-1),-2]\cup
[3,n-1]}\mathscr{I}_t(U_{1,1}^n).$$

Take an edge $e\in H_{red}^n\cup H_{black}^n$. We say $e$ covers
vertex $u_{1,j}^n$ for some $j\in [1,2^{n-4}]$ provided that ${\rm
In}(e)$ crosses the line $y=Y_{u_{1,j}^n}$. Suppose $e$ is incident
to $u_{1,t}^n$ for some $t\in [1,2^{n-4}]$. Let
$\varepsilon(e,u_{1,t}^n)=1$ ($\varepsilon(e,u_{1,t}^n)=-1$) if $e$
shots upward (downward) of the line $y=Y_{u_{1,t}^n}$. For example,
$\varepsilon(\mathscr{I}_2(u_{1,j}^8),u_{1,j}^8)=1$ for $j\in [1,6]$
and $\varepsilon(\mathscr{I}_2(u_{1,k}^8),u_{1,k}^8)=-1$ for $k\in
[7,8]$ (see Figure 3.1).

Take an edge subset $H\subseteq H_{red}^n\cup H_{black}^n$. Then we
define
$$\varpi_j^{n,H}=|\{e\in H: e \mbox{ covers } u_{1,j}^n\}|,$$
$$\vartheta_j^{n,H}=|\{e\in H: e \mbox{ is incident to } u_{1,j}^n \mbox{ with } \varepsilon(e,u_{1,j}^n)=1\}|$$
and $$\varsigma_j^{n,H}=|\{e\in H: e \mbox{ is incident to }
u_{1,j}^n \mbox{ with } \varepsilon(e,u_{1,j}^n)=-1\}|.$$

Therefore, it is easy to derive the following

\noindent\textbf{Assertion C.} \ Let $e^{'}\in \bigcup\limits_{t\in
\{-n,n\}}\mathscr{I}_t(V_{1,1}^n)$, and let $H$ be an edge subset of
$H_{red}^n\cup H_{black}^n$. If ${\rm In}(e^{'})$ is drawn along the
line $y=Y_{u_{1,\mathcal {L}({\rm In}(e))}^n}$ then
\begin{equation}\label{equation along upward}
|\{e\in H: \nu_{\Gamma_n}\big{(} \ {\rm In}(e),\ {\rm In}(e^{'})\
\big{)}=1  \}|=\varpi_{\mathcal {L}({\rm
In}(e^{'}))}^{n,H}+\vartheta_{\mathcal {L}({\rm In}(e^{'}))}^{n,H},
\end{equation}

if ${\rm In}(e^{'})$ is drawn along the line $y=Y_{u_{1,\mathcal
{L}({\rm In}(e^{'}))-1}^n}$ then
\begin{equation}\label{equation along downward}
|\{e\in H: \nu_{\Gamma_n}\big{(} \ {\rm In}(e),\ {\rm In}(e^{'})\
\big{)}=1  \}|=\varpi_{\mathcal {L}({\rm
In}(e^{'}))-1}^{n,H}+\varsigma_{\mathcal {L}({\rm
In}(e^{'}))-1}^{n,H}.
\end{equation}

The following definitions will be key for the computations of
$\nu_{\Gamma_n}(E_{blue}^n, E_{red}^n)$ and
$\nu_{\Gamma_n}(E_{blue}^n, E_{black}^n)$ later.

Define
\begin{equation}\label{equation definition s i}
s_{n,j}=\varpi_j^{n,H_{red}^n}+\vartheta_j^{n,H_{red}^n},
\end{equation}
\begin{equation}\label{equation definition t i}
t_{n,j}=\varpi_j^{n,H_{black}^n}+\vartheta_j^{n,H_{black}^n}
\end{equation}
and
\begin{equation}\label{equation t' i}
t_{n,j}^{'}=\varpi_j^{n,E(U_1^n)}+\vartheta_j^{n,E(U_1^n)},
\end{equation}
where $j\in [1,2^{n-4}]$, and define
\begin{equation}\label{equation definition s 2exp(n-4)+1}
s_{n,2^{n-4}+1}=\varpi_{2^{n-4}}^{n,H_{red}^n}+\varsigma_{2^{n-4}}^{n,H_{red}^n},
\end{equation}
\begin{equation}\label{equation definition t 2exp(n-4)+1}
t_{n,2^{n-4}+1}=\varpi_{2^{n-4}}^{n,H_{black}^n}+\varsigma_{2^{n-4}}^{n,H_{black}^n}
\end{equation}
and
\begin{equation}\label{equation t' 2exp(n-4)+1}
t_{n,2^{n-4}+1}^{'}=\varpi_{2^{n-4}}^{n,E(U_1^n)}+\varsigma_{2^{n-4}}^{n,E(U_1^n)}.
\end{equation}

Since $H_{black}^n\setminus E(U_1^n)=\bigcup\limits_{t\in
\{-(n-1),n-1\}}\mathscr{I}_t(U_{1,1}^n)$, it follows from
\eqref{equation Dim -(n-1)} and \eqref{equation Dim (n-1)} that
\begin{equation}\label{equation t and t'}
\begin{array}{llll} & t_{n,j}=\left \{\begin{array}{llll}
               t_{n,j}^{'},  &  \mbox{ if } j\equiv 1 \pmod 2;\\
               t_{n,j}^{'}+1,  &  \mbox{ if } j\equiv 0 \pmod 2,\\
               \end{array}
           \right. \\
\end{array}
\end{equation}
and
\begin{equation}\label{equation varsigma odd and even}
\begin{array}{llll} & \varsigma_j^{n,H_{black}^n}=\left \{\begin{array}{llll}
               \varsigma_j^{n,E(U_1^n)}+1,  &  \mbox{ if } j\equiv 1 \pmod 2;\\
               \varsigma_j^{n,E(U_1^n)},  &  \mbox{ if } j\equiv 0 \pmod 2,\\
               \end{array}
           \right. \\
\end{array}
\end{equation}
for all $j\in [1,2^{n-4}+1]$.

\begin{lemma}\label{Lemma s and t} Let $n\geq 8$ be an integer. Then the following conclusions hold.

1. $s_{n,j+1}=\varpi_j^{n,H_{red}^n}+\varsigma_j^{n,H_{red}^n}$ \ \
for all $j\in [1,2^{n-4}]$;

2. $s_{n,2j-1}=2s_{n-1,j}$  \ \ for all $j\in [1,2^{n-5}+1]$;

3. $s_{n,2j}=s_{n-1,j}+s_{n-1,j+1}$ \ \ for all $j\in [1,2^{n-5}]$;

4. $t_{n,j+1}=\varpi_j^{n,H_{black}^n}+\varsigma_j^{n,H_{black}^n}$
\ \ for all $j\in [1,2^{n-4}-1]$;

5. $t_{n,2j-1}=2t_{n-1,j}$ \ \ for all $j\in [1,2^{n-5}+1]$;

6. $t_{n,2j}=t_{n-1,j}+t_{n-1,j+1}+2$ \ \ for all $j\in
[1,2^{n-5}]$.
\end{lemma}

\begin{proof} To prove Conclusion 1, 2 and 3, we shall need the following
notations. Let
$$s_{n,j}^d=\varpi_j^{n,H_{red}^n}+\varsigma_j^{n,H_{red}^n}\mbox{ \ \ for }j\in [1,2^{n-4}].$$
We  have that for $n>8$ and for $j\in [1,2^{n-5}]$,
\begin{eqnarray}\label{equation s d odd}
s_{n,2j-1}^d&=&\varpi_{2j-1}^{n,H_{red}^n}+\varsigma_{2j-1}^{n,H_{red}^n} \nonumber\\
&=&(2\cdot
\varpi_{j}^{n-1,H_{red}^{n-1}}+\vartheta_{2j}^{n,H_{red}^{n}})+\varsigma_{2j-1}^{n,H_{red}^{n}} \nonumber\\
&=&(2\cdot
\varpi_{j}^{n-1,H_{red}^{n-1}}+\vartheta_{j}^{n-1,H_{red}^{n-1}})+\varsigma_{j}^{n-1,H_{red}^{n-1}} \nonumber\\
&=&(
\varpi_{j}^{n-1,H_{red}^{n-1}}+\vartheta_{j}^{n-1,H_{red}^{n-1}})+(
\varpi_{j}^{n-1,H_{red}^{n-1}}+\varsigma_{j}^{n-1,H_{red}^{n-1}}) \nonumber\\
&=&s_{n-1,j}+s_{n-1,j}^d,
\end{eqnarray}
\begin{eqnarray}\label{equation s d even}
s_{n,2j}^d&=&\varpi_{2j}^{n,H_{red}^n}+\varsigma_{2j}^{n,H_{red}^n} \nonumber\\
&=&(2\cdot
\varpi_{j}^{n-1,H_{red}^{n-1}}+\varsigma_{2j-1}^{n,H_{red}^{n}})+\varsigma_{j}^{n-1,H_{red}^{n-1}} \nonumber\\
&=&(2\cdot
\varpi_{j}^{n-1,H_{red}^{n-1}}+\varsigma_{j}^{n-1,H_{red}^{n-1}})+\varsigma_{j}^{n-1,H_{red}^{n-1}} \nonumber\\
&=&2\cdot(\varpi_{j}^{n-1,H_{red}^{n-1}}+\varsigma_{j}^{n-1,H_{red}^{n-1}}) \nonumber\\
&=&2\cdot s_{n-1,j}^d,
\end{eqnarray}
\begin{eqnarray}\label{equation s odd}
s_{n,2j-1}&=&\varpi_{2j-1}^{n,H_{red}^n}+\vartheta_{2j-1}^{n,H_{red}^n} \nonumber\\
&=&(2\cdot
\varpi_{j}^{n-1,H_{red}^{n-1}}+\vartheta_{2j}^{n,H_{red}^{n}})+\vartheta_{j}^{n-1,H_{red}^{n-1}} \nonumber\\
&=&(2\cdot \varpi_{j}^{n-1,H_{red}^{n-1}}+\vartheta_{j}^{n-1,H_{red}^{n-1}})+\vartheta_{j}^{n-1,H_{red}^{n-1}} \nonumber\\
&=&2\cdot(\varpi_{j}^{n-1,H_{red}^{n-1}}+\vartheta_{j}^{n-1,H_{red}^{n-1}})\nonumber\\
&=&2\cdot s_{n-1,j}
\end{eqnarray}
and
\begin{eqnarray}\label{equation s even}
s_{n,2j}&=&\varpi_{2j}^{n,H_{red}^n}+\vartheta_{2j}^{n,H_{red}^n} \nonumber\\
&=&(2\cdot
\varpi_{j}^{n-1,H_{red}^{n-1}}+\varsigma_{2j-1}^{n,H_{red}^{n}})+\vartheta_{j}^{n-1,H_{red}^{n-1}} \nonumber\\
&=&(2\cdot
\varpi_{j}^{n-1,H_{red}^{n-1}}+\varsigma_{j}^{n-1,H_{red}^{n-1}})+\vartheta_{j}^{n-1,H_{red}^{n-1}} \nonumber\\
&=&(
\varpi_{j}^{n-1,H_{red}^{n-1}}+\vartheta_{j}^{n-1,H_{red}^{n-1}})+(
\varpi_{j}^{n-1,H_{red}^{n-1}}+\varsigma_{j}^{n-1,H_{red}^{n-1}}) \nonumber\\
&=&s_{n-1,j}+s_{n-1,j}^d.
\end{eqnarray}

Now we prove Conclusion 1 by induction on $n$. If $n=8$, it follows
from trivial verifications. Next we need only to consider the case when $n>8$. Conclusion 1 is equivalent
to show that $$s_{n,j+1}=s_{n,j}^d\mbox{ \ \ for all }j\in
[1,2^{n-4}-1].$$  By \eqref{equation s d odd}, \eqref{equation s d
even}, \eqref{equation s odd}, \eqref{equation s even} and the
induction hypothesis, we have that for all $j\in [1,2^{n-5}-1]$,
\begin{eqnarray}
s_{n,2j+1}&=&2\cdot s_{n-1,j+1} \nonumber\\
&=& 2\cdot s_{n-1,j}^d \nonumber\\
&=& s_{n,2j}^d\nonumber,
\end{eqnarray}
and that for all $j\in [1,2^{n-5}]$,
\begin{eqnarray}
s_{n,2j}&=&s_{n-1,j}+s_{n-1,j}^d \nonumber\\
&=&  s_{n,2j-1}^d \nonumber,
\end{eqnarray}
and Conclusion 1 follows immediately.

Furthermore, Conclusions 2 and 3 follow from \eqref{equation s odd},
\eqref{equation s even} and Conclusion 1.

Now we prove Conclusions 4, 5 and 6. By \eqref{equation location X},
we can infer that
\begin{equation}\label{equation varpi+=varpi-}
\varpi_{j+1}^{n,E(U_1^n)}+\vartheta_{j+1}^{n,E(U_1^n)}=\varpi_j^{n,E(U_1^n)}+\varsigma_j^{n,E(U_1^n)}\mbox{
\ \ for all} \ j\in [1,2^{n-4}].
\end{equation}
 It follows from \eqref{equation t and
t'}, \eqref{equation varsigma odd and even} and \eqref{equation
varpi+=varpi-} that for all $j\in [1,2^{n-4}-1]$,

if $j\equiv 1\pmod 2$ then
\begin{eqnarray}
t_{n,j+1}&=&t_{n,j+1}^{'}+1\nonumber\\
 &=&\varpi_{j+1}^{n,E(U_1^n)}+\vartheta_{j+1}^{n,E(U_1^n)}+1\nonumber\\
&=&\varpi_j^{n,E(U_1^n)}+\varsigma_j^{n,E(U_1^n)}+1\nonumber\\
&=&\varpi_j^{n,H_{black}^n}+\varsigma_j^{n,E(U_1^n)}+1\nonumber\\
&=&\varpi_j^{n,H_{black}^n}+\varsigma_j^{n,H_{black}^n}
\end{eqnarray}

if $j\equiv 0\pmod 2$ then
\begin{eqnarray}
t_{n,j+1}&=&t_{n,j+1}^{'}\nonumber\\
 &=&\varpi_{j+1}^{n,E(U_1^n)}+\vartheta_{j+1}^{n,E(U_1^n)}\nonumber\\
&=&\varpi_j^{n,E(U_1^n)}+\varsigma_j^{n,E(U_1^n)}\nonumber\\
&=&\varpi_j^{n,H_{black}^n}+\varsigma_j^{n,E(U_1^n)}\nonumber\\
&=&\varpi_j^{n,H_{black}^n}+\varsigma_j^{n,H_{black}^n},
\end{eqnarray}
and thus, Conclusion 4 follows.

It remains to show Conclusions 5 and 6. By \eqref{equation Dim
-(n-1)}, \eqref{equation Dim (n-1)}, \eqref{equation Dim -(n-2)},
\eqref{equation t and t'}, \eqref{equation varpi+=varpi-},
Conclusion 4 and Lemma \ref{Lemma enlarging the edges}, we conclude
that for all $j\in[1,2^{n-5}]$,

if $j\equiv 1\pmod 2$ then $$\begin{array}{llll}
t_{n,2j-1}&=&t_{n,2j-1}^{'}\\
&=& \varpi_{2j-1}^{n,E(U_1^n)}+\vartheta_{2j-1}^{n,E(U_1^n)}\\
&=& (2\cdot \varpi_{j}^{n-1,E(U_1^{n-1})}+\vartheta_{2j}^{n,E(U_1^n)}-1)+\vartheta_{j}^{n-1,E(U_1^{n-1})}\\
&=& (2\cdot \varpi_{j}^{n-1,E(U_1^{n-1})}+(\vartheta_{j}^{n-1,E(U_1^{n-1})}+1)-1)+\vartheta_{j}^{n-1,E(U_1^{n-1})}\\
&=& 2\cdot(\varpi_{j}^{n-1,E(U_1^{n-1})}+\vartheta_{j}^{n-1,E(U_1^{n-1})})\\
&=& 2\cdot t_{n,j}^{'}\\
&=& 2\cdot t_{n,j}\\
\end{array}$$
and
\begin{eqnarray}
t_{n,2j}&=&t_{n,2j}^{'}+1\nonumber\\
&=& \varpi_{2j}^{n,E(U_1^n)}+\vartheta_{2j}^{n,E(U_1^n)}+1 \nonumber\\
&=& (2\cdot \varpi_{j}^{n-1,E(U_1^{n-1})}+\varsigma_{2j-1}^{n,E(U_1^n)}-1)+(\vartheta_{j}^{n-1,E(U_1^{n-1})}+1)+1\nonumber\\
&=& (2\cdot \varpi_{j}^{n-1,E(U_1^{n-1})}+(\varsigma_{j}^{n-1,E(U_1^{n-1})}+2)-1)+(\vartheta_{j}^{n-1,E(U_1^{n-1})}+1)+1\nonumber\\
&=& (\varpi_{j}^{n-1,E(U_1^{n-1})}+\vartheta_{j}^{n-1,E(U_1^{n-1})})+(\varpi_{j}^{n-1,E(U_1^{n-1})}+\varsigma_{j}^{n-1,E(U_1^{n-1})})+3\nonumber\\
&=& (\varpi_{j}^{n-1,E(U_1^{n-1})}+\vartheta_{j}^{n-1,E(U_1^{n-1})})+(\varpi_{j+1}^{n-1,E(U_1^{n-1})}+\vartheta_{j+1}^{n-1,E(U_1^{n-1})})+3\nonumber\\
&=& t_{n,j}^{'}+(t_{n,j+1}^{'}+1)+2\nonumber\\
&=& t_{n,j}+t_{n,j+1}+2,
\end{eqnarray}

if $j\equiv 0\pmod 2$ then
\begin{eqnarray}
t_{n,2j-1}&=&t_{n,2j-1}^{'}\nonumber\\
&=& \varpi_{2j-1}^{n,E(U_1^n)}+\vartheta_{2j-1}^{n,E(U_1^n)} \nonumber\\
&=& (2\cdot \varpi_{j}^{n-1,E(U_1^{n-1})}+\vartheta_{2j}^{n,E(U_1^n)}-1)+(\vartheta_{j}^{n-1,E(U_1^{n-1})}+1)\nonumber\\
&=& (2\cdot \varpi_{j}^{n-1,E(U_1^{n-1})}+(\vartheta_{j}^{n-1,E(U_1^{n-1})}+2)-1)+(\vartheta_{j}^{n-1,E(U_1^{n-1})}+1)\nonumber\\
&=& 2\cdot(\varpi_{j}^{n-1,E(U_1^{n-1})}+\vartheta_{j}^{n-1,E(U_1^{n-1})})+2\nonumber\\
&=& 2\cdot t_{n,j}^{'}+2\nonumber\\
&=& 2\cdot (t_{n,j}^{'}+1) \nonumber\\
&=& 2\cdot t_{n,j}
\end{eqnarray}
and
\begin{eqnarray}
t_{n,2j}&=&t_{n,2j}^{'}+1\nonumber\\
&=& \varpi_{2j}^{n,E(U_1^n)}+\vartheta_{2j}^{n,E(U_1^n)}+1 \nonumber\\
&=& (2\cdot \varpi_{j}^{n-1,E(U_1^{n-1})}+\varsigma_{2j-1}^{n,E(U_1^n)}-1)+(\vartheta_{j}^{n-1,E(U_1^{n-1})}+2)+1\nonumber\\
&=& (2\cdot \varpi_{j}^{n-1,E(U_1^{n-1})}+(\varsigma_{j}^{n-1,E(U_1^{n-1})}+1)-1)+(\vartheta_{j}^{n-1,E(U_1^{n-1})}+2)+1\nonumber\\
&=& (\varpi_{j}^{n-1,E(U_1^{n-1})}+\vartheta_{j}^{n-1,E(U_1^{n-1})})+(\varpi_{j}^{n-1,E(U_1^{n-1})}+\varsigma_{j}^{n-1,E(U_1^{n-1})})+3\nonumber\\
&=& (\varpi_{j}^{n-1,E(U_1^{n-1})}+\vartheta_{j}^{n-1,E(U_1^{n-1})})+(\varpi_{j+1}^{n-1,E(U_1^{n-1})}+\vartheta_{j+1}^{n-1,E(U_1^{n-1})})+3\nonumber\\
&=& t_{n,j}^{'}+t_{n,j+1}^{'}+3\nonumber\\
&=& t_{n,j}^{'}+(t_{n,j+1}^{'}+1)+2\nonumber\\
&=& t_{n,j}+t_{n,j+1}+2,
\end{eqnarray}
and thus, Conclusions 5 and 6 follow immediately.
\end{proof}

\textbf{1. Calculation of $\nu_{\Gamma_n}(E_{blue}^n)$.}

By \eqref{equation Eb Er Eb}, \eqref{equation Dim -n},
\eqref{equation Dim n} and Property B, we have that
\begin{eqnarray}\label{equation counting Eblue}
& &\nu_{\Gamma_n}(E_{blue}^n) \nonumber\\
&=&\nu_{\Gamma_n}\big{(}\
{\rm Out}_{\mathbb{R}_{-5}^{5}\times
\mathbb{R}_{-5}^{5}}(E_{blue}^n)\ \big{)}+\nu_{\Gamma_n}\big{(}\
{\rm In}_{\mathbb{R}_{-5}^{5}\times
\mathbb{R}_{-5}^{5}}(E_{blue}^n)\ \big{)}\nonumber\\
&=&\nu_{\Gamma_n}\big{(}\ {\rm Out}_{\mathbb{R}_{-5}^{5}\times
\mathbb{R}_{-5}^{5}}(E_{blue}^n)\ \big{)}+ 8\cdot
\nu_{\Gamma_n}\big{(}\ {\rm In}(E_{blue}^n)\ \big{)}\nonumber\\
&=& 2\cdot \nu_{\Gamma_n}\big{(}\ {\rm
Out}_{\mathbb{R}_{-5}^{5}\times
\mathbb{R}_{-5}^{5}}(\bigcup\limits_{t\in
\{-n,n\}}\mathscr{I}_{t}(U_1^n\cup V_1^n))\ \big{)}+ 8\cdot
\nu_{\Gamma_n}\big{(}\ {\rm In}(\bigcup\limits_{t\in
\{-n,n\}}\mathscr{I}_{t}(U_{1,1}^n\cup V_{1,1}^n))\ \big{)}.
\end{eqnarray}

Now we calculate $\nu_{\Gamma_n}\big{(}\ {\rm
Out}_{\mathbb{R}_{-5}^{5}\times
\mathbb{R}_{-5}^{5}}(\bigcup\limits_{t\in
\{-n,n\}}\mathscr{I}_{t}(U_1^n\cup V_1^n))\ \big{)}$.

By \eqref{equation Dim -n}, \eqref{equation Dim n}, Inductive rule
for the drawing of $E_{blue}^n$ and Property B, we have that
\begin{equation}\label{equation Out=0 first}
\nu_{\Gamma_n}\big{(}\ {\rm Out}_{\mathbb{R}_{-5}^{5}\times
\mathbb{R}_{-5}^{5}}(\mathscr{I}_{n}(V_1^n))\ \big{)}=0,
\end{equation}
\begin{equation}\label{equation Out=0 second}
\nu_{\Gamma_n}\big{(}\ {\rm Out}_{\mathbb{R}_{-5}^{5}\times
\mathbb{R}_{-5}^{5}}(\mathscr{I}_{n}(V_1^n)),\ {\rm
Out}_{\mathbb{R}_{-5}^{5}\times
\mathbb{R}_{-5}^{5}}(\mathscr{I}_{-n}(V_1^n)\cup
\mathscr{I}_{-n}(U_1^n)\cup \mathscr{I}_{n}(U_1^n))\ \big{)}=0,
\end{equation}
and that for any two distinct integers $j,k\in [1,2^{n-3}]$,
\begin{equation}\label{equation Out=0 third}
\nu_{\Gamma_n}\big{(}\ {\rm Out}_{\mathbb{R}_{-5}^{5}\times
\mathbb{R}_{-5}^{5}}(\mathscr{I}_{-n}(v_{1,j}^n)\cup
\mathscr{I}_{-n}(u_{1,j}^n)\cup \mathscr{I}_{n}(u_{1,j}^n)),\ {\rm
Out}_{\mathbb{R}_{-5}^{5}\times
\mathbb{R}_{-5}^{5}}(\mathscr{I}_{-n}(v_{1,k}^n)\cup
\mathscr{I}_{-n}(u_{1,k}^n)\cup \mathscr{I}_{n}(u_{1,k}^n)) \
\big{)}=0,
\end{equation}
and furthermore that for every $j\in [1,2^{n-3}]$,
\begin{equation}\label{equation Out=1}
\nu_{\Gamma_n}\big{(}\ {\rm Out}_{\mathbb{R}_{-5}^{5}\times
\mathbb{R}_{-5}^{5}}(\mathscr{I}_{-n}(v_{1,j}^n)\cup
\mathscr{I}_{-n}(u_{1,j}^n)\cup \mathscr{I}_{n}(u_{1,j}^n))\
\big{)}=1
\end{equation}
which are shown in Figure 3.3 for example. In Figure 3.3,
$P_{u_{1,j}^n}^{\epsilon}$   ($P_{v_{1,j}^n}^{\epsilon}$) denotes
the `point' at which ${\rm In}_{\mathbb{R}_{-5}^{5}\times
\mathbb{R}_{0}^{5}}(\mathscr{I}_{\epsilon}(u_{1,j}^n))$ (${\rm
In}_{\mathbb{R}_{-5}^{5}\times
\mathbb{R}_{0}^{5}}(\mathscr{I}_{\epsilon}(v_{1,j}^n))$) and the
line $x=-5$ cross each other, and $P_{u_{3,j}^n}^{\epsilon}$
($P_{v_{3,j}^n}^{\epsilon}$) denotes the `point' at which ${\rm
In}_{\mathbb{R}_{-5}^{5}\times
\mathbb{R}_{-5}^{0}}(\mathscr{I}_{\epsilon}(u_{1,j}^n))$ \ (${\rm
In}_{\mathbb{R}_{-5}^{5}\times
\mathbb{R}_{-5}^{0}}(\mathscr{I}_{\epsilon}(v_{1,j}^n))$) and the
line $x=-5$ cross each other, where $\epsilon\in \{-n,n\}$.

\begin{figure}[ht]
\centering
\includegraphics[scale=1.0]{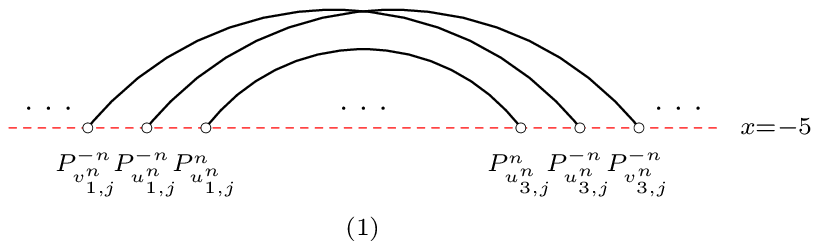}

\vspace{10pt}

\includegraphics[scale=1.0]{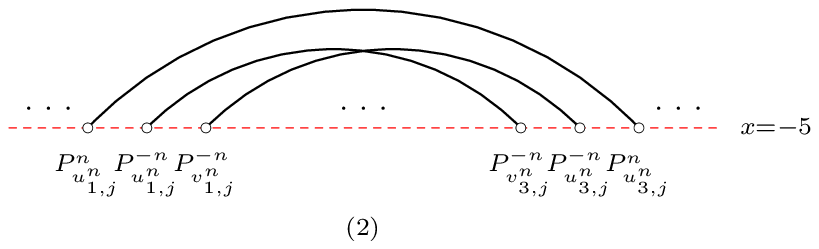}
\caption{\small{An auxiliary graph for the drawing of ${\rm
Out}_{\mathbb{R}_{-5}^{5}\times \mathbb{R}_{-5}^{5}}(E_{blue}^n)$}}
\end{figure}

Combined \eqref{equation Out=0 first}, \eqref{equation Out=0
second}, \eqref{equation Out=0 third} and \eqref{equation Out=1}, we
conclude that
\begin{equation}\label{equation Out(Elue)}
\nu_{\Gamma_n}\big{(}\ {\rm Out}_{\mathbb{R}_{-5}^{5}\times
\mathbb{R}_{-5}^{5}}(\bigcup\limits_{t\in
\{-n,n\}}\mathscr{I}_{t}(U_1^n\cup V_1^n))\ \big{)}=2^{n-3}.
\end{equation}

Next we calculate $\nu_{\Gamma_n}\big{(}\ {\rm
In}(\bigcup\limits_{t\in \{-n,n\}}\mathscr{I}_{t}(U_{1,1}^n\cup
V_{1,1}^n))\ \big{)}$.

By Inductive rule for the drawing of $E_{blue}^n$ and Observation
\ref{Observation two bunches cross}, we have that
\begin{equation}\label{equation In(Eblue) first}
\nu_{\Gamma_n}\big{(}\ {\rm In}(\bigcup\limits_{t\in
\{-n,n\}}\mathscr{I}_{t}(U_{1,1}^n))\ \big{)}=0,
\end{equation}
\begin{equation}\label{equation In(Eblue) second}
\nu_{\Gamma_n}\big{(}\ {\rm In}(\bigcup\limits_{t\in
\{-n,n\}}\mathscr{I}_{t}(U_{1,1}^n)),\ {\rm In}(\bigcup\limits_{t\in
\{-n,n\}}\mathscr{I}_{t}(V_{1,1}^n))\ \big{)}=0,
\end{equation}
 and that
\begin{equation}\label{equation In(Eblue) third}
\nu_{\Gamma_n}\big{(}\ {\rm In}(\bigcup\limits_{t\in
\{-n,n\}}\mathscr{I}_{t}(V_{1,1}^n))\ \big{)}={2^{n-6}\choose
2}+{2^{n-7}\choose 2}+2\cdot{2^{n-8}\choose 2}
\end{equation} where

${2^{n-6}\choose 2}$ crossings are produced between $\{{\rm
In}(\mathscr{I}_{n}(v_{1,j}^n)):j\in [2^{n-4}-2^{n-6}+1,2^{n-4}]\}$
and $\{{\rm In}(\mathscr{I}_{-n}(v_{1,j}^n)):j\in
[2^{n-4}-2^{n-6}+1,2^{n-4}]\}$,

${2^{n-7}\choose 2}$ crossings are produced between $\{{\rm
In}(\mathscr{I}_{n}(v_{1,j}^n)):j\in [2^{n-5}+1,2^{n-5}+2^{n-7}]\}$
and $\{{\rm In}(\mathscr{I}_{-n}(v_{1,j}^n)):j\in
[2^{n-5}+1,2^{n-5}+2^{n-7}]\}$,

$2\cdot{2^{n-8}\choose 2}$ crossings are produced between $\{{\rm
In}(\mathscr{I}_{n}(v_{1,j}^n)):j\in
[2^{n-5}+2^{n-7}+1,2^{n-5}+2^{n-6}]\}$ and $\{{\rm
In}(\mathscr{I}_{-n}(v_{1,j}^n)):j\in
[2^{n-5}+2^{n-7}+1,2^{n-5}+2^{n-6}]\}$.

Combined \eqref{equation In(Eblue) first}, \eqref{equation In(Eblue)
second} and  \eqref{equation In(Eblue) third}, we have
\begin{equation}\label{equation In(Elue)}
\nu_{\Gamma_n}\big{(}\ {\rm In}(\bigcup\limits_{t\in
\{-n,n\}}\mathscr{I}_{t}(U_{1,1}^n\cup V_{1,1}^n))\
\big{)}={2^{n-6}\choose 2}+{2^{n-7}\choose 2}+2\cdot{2^{n-8}\choose
2}.
\end{equation}

Therefore, it follows from \eqref{equation counting Eblue},
\eqref{equation Out(Elue)} and \eqref{equation In(Elue)} that
\begin{equation}\label{equation E-blue}
\nu_{\Gamma_n}(E_{blue}^n)=2\cdot 2^{n-3}+8\cdot\big{(}\
{2^{n-6}\choose 2}+{2^{n-7}\choose 2}+2\cdot{2^{n-8}\choose 2}\
\big{)}=11\cdot2^{2n-13}+2^{n-3}.
\end{equation}

\medskip

\textbf{2. Calculation of $\nu_{\Gamma_n}(E_{red}^n)$.}

By \eqref{equation Eb Er Eb}, \eqref{equation Dim -n},
\eqref{equation Dim n}, Inductive rule for the drawing of
$E_{red}^n$ and Property B, we have that
\begin{equation}\label{equation counting Ered}
\nu_{\Gamma_n}(E_{red}^n)=8\cdot\nu_{\Gamma_n}\big{(}\ {\rm
In}(\bigcup\limits_{t\in \{1,2\}}\mathscr{I}_{t}(U_{1,1}^n\cup
V_{1,1}^n))\ \big{)}+2\cdot \nu_{\Gamma_n}\big{(}\
\mathscr{I}_{2}(U_{1,1}^n\cup V_{1,1}^n), \
\mathscr{I}_{2}(U_{1,2}^n\cup V_{1,2}^n)\ \big{)}.
\end{equation}

Now we calculate $\nu_{\Gamma_n}\big{(}\ {\rm
In}(\bigcup\limits_{t\in \{1,2\}}\mathscr{I}_{t}(U_{1,1}^n\cup
V_{1,1}^n))\ \big{)}$.

By Inductive rule for the drawing of $E_{red}^n$ and Observation
\ref{Observation two bunches cross}, we have that
\begin{equation} \label{equation Ared-u}
\nu_{\Gamma_n}\big{(}\ {\rm In}(\bigcup\limits_{t\in
\{1,2\}}\mathscr{I}_{t}(U_{1,1}^n))\
\big{)}={2^{n-6}\choose2}+3\cdot2^{n-6}\cdot2^{n-6}
\end{equation}
where

${2^{n-6}\choose2}$ crossings are produced between $\{{\rm
In}(\mathscr{I}_1(u_{1,j}^n)):j\in [2^{n-4}-2^{n-6}+1,2^{n-4}]\}$
and $\{{\rm In}(\mathscr{I}_2(u_{1,j}^n)):j\in
[2^{n-4}-2^{n-6}+1,2^{n-4}]\}$,

$3\cdot2^{n-6}\cdot2^{n-6}$ crossings are produced between $\{{\rm
In}(\mathscr{I}_1(u_{1,j}^n)):j\in [2^{n-4}-2^{n-6}+1,2^{n-4}]\}$
and $\{{\rm In}(\mathscr{I}_2(u_{1,j}^n)):j\in
[1,2^{n-4}-2^{n-6}]\}$,

and that
\begin{equation} \label{equation Ared-v}
\nu_{\Gamma_n}\big{(}\ {\rm In}(\bigcup\limits_{t\in
\{1,2\}}\mathscr{I}_{t}(V_{1,1}^n))\ \big{)}={2^{n-4}\choose 2}
\end{equation}
where ${2^{n-4}\choose 2}$ crossings are produced between $\{{\rm
In}(\mathscr{I}_1(v_{1,j}^n)):j\in [1,2^{n-4}]\}$ and $\{{\rm
In}(\mathscr{I}_2(v_{1,j}^n)):j\in [1,2^{n-4}]\}$,

and that
\begin{equation} \label{equation Ared-u and Ared-v}
\nu_{\Gamma_n}\big{(}\ {\rm In}(\bigcup\limits_{t\in
\{1,2\}}\mathscr{I}_{t}(U_{1,1}^n)),\ {\rm In}(\bigcup\limits_{t\in
\{1,2\}}\mathscr{I}_{t}(V_{1,1}^n))\
\big{)}=3\cdot2^{n-7}\cdot2^{n-4}+2^{n-6}\cdot2^{n-5}
\end{equation}
where

$3\cdot2^{n-7}\cdot2^{n-4}$ crossings are produced between $\{{\rm
In}(\mathscr{I}_2(u_{1,j}^n)):j\in [1,2^{n-6}+2^{n-7}]\}$ and
$\{{\rm In}(\mathscr{I}_1(v_{1,j}^n)):j\in [1,2^{n-4}]\}$,

$2^{n-6}\cdot2^{n-5}$ crossings are produced between $\{{\rm
In}(\mathscr{I}_2(u_{1,j}^n)):j\in
[2^{n-6}+2^{n-7}+1,2^{n-5}+2^{n-7}]\}$ and $\{{\rm
In}(\mathscr{I}_1(v_{1,j}^n)):j\in [2^{n-5}+1,2^{n-4}]\}.$

Combined \eqref{equation Ared-u}, \eqref{equation Ared-v} and
\eqref{equation Ared-u and Ared-v}, we have
\begin{eqnarray}\label{equation counting Ered first}
& &\nu_{\Gamma_n}\big{(}\ {\rm In}(\bigcup\limits_{t\in
\{1,2\}}\mathscr{I}_{t}(U_{1,1}^n\cup
V_{1,1}^n))\ \big{)}\nonumber\\
&=& \big{(}\ {2^{n-6}\choose2}+3\cdot2^{n-6}\cdot2^{n-6}\
\big{)}+{2^{n-4}\choose
2}+\big{(}\ 3\cdot2^{n-7}\cdot2^{n-4}+2^{n-6}\cdot2^{n-5}\ \big{)}\nonumber\\
&=&39\cdot2^{2n-13}-5\cdot2^{n-7}.
\end{eqnarray}

On the other hand, it is easy to see that
\begin{equation}\label{equation counting Ered second}\nu_{\Gamma_n}\big{(}\ \mathscr{I}_{2}(U_{1,1}^n\cup V_{1,1}^n), \
\mathscr{I}_{2}(U_{1,2}^n\cup V_{1,2}^n)\ \big{)}=2^{n-3}\cdot
2^{n-3}=4^{n-3}. \end{equation}

Therefore, it follows from \eqref{equation counting Ered},
\eqref{equation counting Ered first} and \eqref{equation counting
Ered second} that
\begin{equation} \label{equation E-red}
\nu_{\Gamma_n}(E_{red}^n)=8\cdot \big{(}\
39\cdot2^{2n-13}-5\cdot2^{n-7}\ \big{)}+2\cdot
4^{n-3}=71\cdot2^{2n-10}-5\cdot2^{n-4}.
\end{equation}

\medskip

\textbf{3. Calculation of $\nu_{\Gamma_n}(E_{black}^n)$.}

By Property B and Lemma \ref{lemma counting of crossings}, we have
\begin{eqnarray}\label{equation Eblack first}
\nu_{\Gamma_n}(E_{black}^n)&=& 4\cdot \nu_{\Gamma_n}( E(U_1^n\cup V_1^n))\nonumber\\
&=&4\cdot \big{(}\
\nu_{\Gamma_n}(E(U_1^n))+\nu_{\Gamma_n}(E(V_1^n))+
\nu_{\Gamma_n}(E[U_1^n, \ V_1^n])\nonumber\\
& & \ \ \ \ \ + \ \nu_{\Gamma_n}(E[U_1^n,V_1^n], \
E(U_1^n))+\nu_{\Gamma_n}(E[U_1^n,V_1^n], \ E(V_1^n))\ \big{)}.
\end{eqnarray}

By  Lemma \ref{Lemma crossing of E-black}, we have
\begin{eqnarray}\label{equation Eblack second}
& &\nu_{\Gamma_n}(E(U_1^n))=\nu_{\Gamma_n}(E(V_1^n))\nonumber\\
&=&\nu_{\Upsilon_{(n-5)}}(\mathcal {E}^{(n-5)})\nonumber\\
&=&\frac{194}{3}\cdot4^{n-6}-(4\cdot (n-5)^2+23\cdot
(n-5)+\frac{101}{3}-\frac{7\cdot(1+(-1)^{n-5})}{6})\cdot2^{n-6}\nonumber\\
&=&\frac{97}{3}\cdot2^{2n-11}-(4n^2-17n+\frac{56}{3}-\frac{7\cdot(1+(-1)^{n-1})}{6})\cdot2^{n-6}.
\end{eqnarray}

By \eqref{equation order of Yu}, \eqref{equation order of Yv},
\eqref{equation Yu=Yv}, \eqref{equation Dim -(n-1)}, \eqref{equation
Dim (n-1)} and Conclusion 4 in Lemma \ref{Lemma enlarging the
edges}, we have
\begin{equation}\label{equation Eblack third}
\nu_{\Gamma_n}(E[U_1^n,V_1^n])=2^{n-4},
\end{equation}
and by \eqref{equation location Y} and Lemma \ref{Lemma C+ and C-},
we have
\begin{eqnarray}\label{equation Eblack fourth}
& &\nu_{\Gamma_n}(E[U_1^n,V_1^n], \
E(U_1^n))=\nu_{\Gamma_n}(E[U_1^n,V_1^n], \ E(V_1^n))\nonumber\\
&=&2\cdot \mathcal {C}_{-}^{(n-5)}\nonumber\\
&=&2\cdot\big{(}\ \frac{5}{3}\cdot
4^{n-4}-(2\cdot(n-5)+\frac{13}{3}+\frac{7\cdot
(1+(-1)^{n-5})}{6})\cdot 2^{n-5}\ \big{)}\nonumber\\
&=&\frac{5}{3}\cdot 2^{2n-7}-(2n-\frac{17}{3}+\frac{7\cdot
(1+(-1)^{n-1})}{6})\cdot 2^{n-4}.
\end{eqnarray}

Therefore, it follows from \eqref{equation Eblack first},
\eqref{equation Eblack second}, \eqref{equation Eblack third} and
\eqref{equation Eblack fourth} that
\begin{eqnarray}\label{equation general E-black}
\nu_{\Gamma_n}(E_{black}^n)
&=&4\cdot\big{(}\ \frac{97}{3}\cdot2^{2n-10}-(4n^2-17n+\frac{56}{3}-\frac{7\cdot(1+(-1)^{n-1})}{6})\cdot2^{n-5}+2^{n-4}\nonumber\\
& &\ \ \ \ \ +\ \frac{5}{3}\cdot2^{2n-6}-(2n-\frac{17}{3}+\frac{7\cdot (1+(-1)^{n-1})}{6})\cdot2^{n-3}\ \big{)}\nonumber\\
&=&59\cdot 2^{2n-8}-\big{(}4n^2-9n-6\big{)}\cdot
2^{n-3}-7\cdot(1+(-1)^{n-1})\cdot2^{n-4}.
\end{eqnarray}

\medskip

\textbf{4. Calculation of $\nu_{\Gamma_n}(E_{red}^n,E_{black}^n)$.}

By  Property B and Lemma \ref{lemma counting of crossings}, we have
that
\begin{eqnarray} \label{equation E-red and E-black general}
& &\nu_{\Gamma_n}(E_{red}^n,E_{black}^n) \nonumber\\
&=&8\cdot \nu_{\Gamma_n}\big{(}\ {\rm In}(E_{red}^n),\ {\rm In}(E_{black}^n) \ \big{)}\nonumber\\
&=&8\cdot \nu_{\Gamma_n}\big{(}\ {\rm In}(\bigcup\limits_{t\in
\{1,2\}}\mathscr{I}_t(U_{1,1}^n\cup V_{1,1}^n)), \ {\rm In}(E(U_1^n\cup V_1^n))\ \big{)} \nonumber\\
&=&8\cdot\big{(}\ \nu_{\Gamma_n}({\rm In}(\bigcup\limits_{t\in
\{1,2\}}\mathscr{I}_t(U_{1,1}^n)),\ {\rm
In}(E(U_1^n)))+\nu_{\Gamma_n}({\rm In}(\bigcup\limits_{t\in
\{1,2\}}\mathscr{I}_t(U_{1,1}^n)),\ E[U_{1,1}^n,V_{1,1}^n])\nonumber\\
& & \ \ \ \ \ +\ \nu_{\Gamma_n}({\rm In}(\bigcup\limits_{t\in
\{1,2\}}\mathscr{I}_t(U_{1,1}^n)),\ {\rm In}(E(V_1^n)))
+\nu_{\Gamma_n}({\rm In}(\bigcup\limits_{t\in
\{1,2\}}\mathscr{I}_t(V_{1,1}^n)),\ {\rm In}(E(U_1^n)))\nonumber\\
& &\ \ \ \ \ +\ \nu_{\Gamma_n}({\rm In}(\bigcup\limits_{t\in
\{1,2\}}\mathscr{I}_t(V_{1,1}^n)),\
E[U_{1,1}^n,V_{1,1}^n])+\nu_{\Gamma_n}({\rm In}(\bigcup\limits_{t\in
\{1,2\}}\mathscr{I}_t(V_{1,1}^n)),\ {\rm In}(E(V_1^n))) \ \big{)}
\end{eqnarray}

By Observation \ref{Observation two bunches cross}, Lemma \ref{Lemma
C+ and C-}, Lemma \ref{Lemma three conclusions on C+ and C-} and
Inductive rule for the drawing of $E_{red}^n$, we have that
\begin{eqnarray} \label{equation Ared-u and blackl}
& &\nu_{\Gamma_n}\big{(}\ {\rm In}(\bigcup\limits_{t\in
\{1,2\}}\mathscr{I}_t(U_{1,1}^n)),\
{\rm In}(E(U_1^n))\ \big{)}\nonumber\\
&=&\sum\limits_{j\in \mathcal{I}_6^{(n-5)}\cup\mathcal{I}_7^{(n-5)}\cup\mathcal{I}_8^{(n-5)}}\big{(}\ \mathscr{C}_{+}^{(n-5)}(\mathcal {E}^{(n-5)},z_j^{(n-5)})+\mathscr{C}_{-}^{(n-5)}(\mathcal {E}^{(n-5)},z_j^{(n-5)})\ \big{)}\nonumber\\
& &+\ 2\sum\limits_{j\in \mathcal{I}_5^{(n-5)}}\mathscr{C}_{-}^{(n-5)}(\mathcal{E}^{(n-5)},z_j^{(n-5)})\nonumber\\
&=&\frac{335}{3}\cdot4^{n-7}-(16n-\frac{100}{3}+\frac{7\cdot(1+(-1)^{n-1})}{3})\cdot2^{n-8},
\end{eqnarray}

and that
\begin{equation} \label{equation Ared-u and blacklr}
\nu_{\Gamma_n}\big{(}\ {\rm In}(\bigcup\limits_{t\in
\{1,2\}}\mathscr{I}_t(U_{1,1}^n)),\ E[U_{1,1}^n,V_{1,1}^n]\
\big{)}=2\cdot2\cdot{3\cdot2^{n-7}\choose2}+2\cdot2\cdot{2^{n-7}\choose2}+2\cdot{2^{n-6}\choose2}+2\cdot3\cdot2^{n-6}\cdot2^{n-6}
\end{equation}
where

$2\cdot2\cdot{3\cdot2^{n-7}\choose2}$ crossings are produced between
$\{{\rm In}(\mathscr{I}_2(u_{1,j}^n)):j\in [1,2^{n-6}+2^{n-7}]\cup
[2^{n-5}+2^{n-7}+1,2^{n-4}]\}$ and $E[U_{1,1}^n,V_{1,1}^n]$;

$2\cdot2\cdot{2^{n-7}\choose2}$ crossings are produced between
$\{{\rm In}(\mathscr{I}_2(u_{1,j}^n)):j\in
[2^{n-6}+2^{n-7}+1,2^{n-5}+2^{n-7}]\}$ and $E[U_{1,1}^n,V_{1,1}^n]$;

$2\cdot{2^{n-6}\choose2}+2\cdot3\cdot2^{n-6}\cdot2^{n-6}$ crossings
are produced between $\{{\rm In}(\mathscr{I}_1(u_{1,j}^n)):j\in
[2^{n-5}+2^{n-6}+1,2^{n-4}]\}$ and $E[U_{1,1}^n,V_{1,1}^n]$,

and that
\begin{equation} \label{equation Ared-u and blackr}
\nu_{\Gamma_n}\big{(}\ {\rm In}(\bigcup\limits_{t\in
\{1,2\}}\mathscr{I}_t(U_{1,1}^n)),\ {\rm In}(E(V_1^n))\
\big{)}=4\cdot2^{n-5}\cdot2^{n-6}+2\cdot2^{n-4}\cdot3\cdot2^{n-7}
\end{equation}
where

$4\cdot2^{n-5}\cdot2^{n-6}$ crossings are produced between $\{{\rm
In}(\mathscr{I}_2(u_{1,j}^n)):j\in
[2^{n-6}+2^{n-7}+1,2^{n-5}+2^{n-7}]\}$ and ${\rm In}(E(V_1^n))$;

$2\cdot2^{n-4}\cdot3\cdot2^{n-7}$ crossings are produced between
$\{{\rm In}(\mathscr{I}_2(u_{1,j}^n)):j\in
[2^{n-5}+2^{n-7}+1,2^{n-4}]\}$ and ${\rm In}(E(V_1^n))$,

and that
\begin{equation} \label{equation Ared-v and blackr}
\nu_{\Gamma_n}\big{(}\ {\rm In}(\bigcup\limits_{t\in
\{1,2\}}\mathscr{I}_t(V_{1,1}^n)),\ {\rm In}(E(V_1^n))\
\big{)}=2\cdot \frac{1}{2}\cdot\mathcal
{C}_{+}^{(n-5)}=\frac{7}{3}\cdot4^{n-4}-(2n-\frac{14}{3}+\frac{7(1+(-1)^n)}{6})\cdot2^{n-5},
\end{equation}

and that
\begin{equation} \label{equation Ared-v and black1,blacklr first}
\nu_{\Gamma_n}\big{(}\ {\rm In}(\bigcup\limits_{t\in
\{1,2\}}\mathscr{I}_t(V_{1,1}^n)),\ {\rm In}(E(U_1^n))\ \big{)}=0,
\end{equation}

and that
\begin{equation} \label{equation Ared-v and black1,blacklr}
\nu_{\Gamma_n}\big{(}\ {\rm In}(\bigcup\limits_{t\in
\{1,2\}}\mathscr{I}_t(V_{1,1}^n)),\ E[U_{1,1}^n,V_{1,1}^n]\
\big{)}=0.
\end{equation}

Therefore, it follows from \eqref{equation E-red and E-black
general}, \eqref{equation Ared-u and blackl}, \eqref{equation Ared-u
and blacklr}, \eqref{equation Ared-u and blackr}, \eqref{equation
Ared-v and blackr}, \eqref{equation Ared-v and black1,blacklr first}
and \eqref{equation Ared-v and black1,blacklr} that
\begin{eqnarray} \label{equation E-red and E-black}
\nu_{\Gamma_n}(E_{red}^n,E_{black}^n)
&=&8\cdot\big{(}
 389\cdot2^{2n-14}-(n-1)\cdot2^{n-3}+7\cdot(1+(-1)^{n-1})\cdot2^{n-8}\ \big{)}\nonumber\\
&=&389\cdot2^{2n-11}-(n-1)\cdot2^{n}+7\cdot(1+(-1)^{n-1})\cdot2^{n-5}.
\end{eqnarray}

\medskip

\textbf{5. Calculation of $\nu_{\Gamma_n}(E_{blue}^n,E_{red}^n)$.}

By \eqref{equation Eb Er Eb}, Property B and Lemma \ref{lemma
counting of crossings}, we have that
\begin{eqnarray} \label{equation E-blue and E-red general}
& &\nu_{\Gamma_n}(E_{blue}^n,E_{red}^n)  \nonumber\\
&=& 8\cdot \nu_{\Gamma_n}\big{(}\ {\rm In}(E_{blue}^n),{\rm
In}(E_{red}^n)\ \big{)}\nonumber\\
&=& 8\cdot \nu_{\Gamma_n}\big{(}\ {\rm In}(\bigcup\limits_{t\in
\{-n,n\}}\mathscr{I}_t(U_{1,1}^n\cup V_{1,1}^n)), \ {\rm
In}(\bigcup\limits_{t\in
\{1,2\}}\mathscr{I}_t(U_{1,1}^n\cup V_{1,1}^n))\ \big{)}\nonumber\\
&=& 8\cdot \nu_{\Gamma_n}\big{(}\ {\rm In}(\bigcup\limits_{t\in
\{-n,n\}}\mathscr{I}_t(U_{1,1}^n))\cup {\rm In}(\bigcup\limits_{t\in
\{-n,n\}}\mathscr{I}_t(V_{1,1}^n)), \ {\rm In}(\bigcup\limits_{t\in
\{1,2\}}\mathscr{I}_t(U_{1,1}^n))\cup {\rm In}(\bigcup\limits_{t\in
\{1,2\}}\mathscr{I}_t(V_{1,1}^n))\ \big{)}\nonumber\\
&=&8\cdot\big{(} \ \nu_{\Gamma_n}( {\rm In}(\bigcup\limits_{t\in
\{-n,n\}}\mathscr{I}_t(U_{1,1}^n)),\ {\rm In}(\bigcup\limits_{t\in
\{1,2\}}\mathscr{I}_t(U_{1,1}^n)) )+\nu_{\Gamma_n}( {\rm
In}(\bigcup\limits_{t\in \{-n,n\}}\mathscr{I}_t(U_{1,1}^n)),\  {\rm
In}(\bigcup\limits_{t\in \{1,2\}}\mathscr{I}_t(V_{1,1}^n)))  \nonumber\\
& & \ \ \ \ +\ \nu_{\Gamma_n} ( {\rm In}(\bigcup\limits_{t\in
\{-n,n\}}\mathscr{I}_t(V_{1,1}^n)),\  {\rm In}(\bigcup\limits_{t\in
\{1,2\}}\mathscr{I}_t(U_{1,1}^n)) )+\nu_{\Gamma_n}( {\rm
In}(\bigcup\limits_{t\in \{-n,n\}}\mathscr{I}_t(V_{1,1}^n)),\  {\rm
In}(\bigcup\limits_{t\in \{1,2\}}\mathscr{I}_t(V_{1,1}^n)) ) \ \big{)}\nonumber\\
\end{eqnarray}

By Inductive rule for the drawing of $E_{red}^n$, Inductive rule for
the drawing of $E_{blue}^n$ and Observation \ref{Observation two
bunches cross}, we have that
\begin{equation} \label{equation Ablue-u and Ared-v}
\nu_{\Gamma_n}\big{(} \ {\rm In}(\bigcup\limits_{t\in
\{-n,n\}}\mathscr{I}_t(U_{1,1}^n)), \ {\rm In}(\bigcup\limits_{t\in
\{1,2\}}\mathscr{I}_t(V_{1,1}^n))\ \big{)}=0,
\end{equation}
and that
\begin{equation} \label{equation Ablue-u and Ared-u}
\nu_{\Gamma_n}\big{(} \ {\rm In}(\bigcup\limits_{t\in
\{-n,n\}}\mathscr{I}_t(U_{1,1}^n)), \ {\rm In}(\bigcup\limits_{t\in
\{1,2\}}\mathscr{I}_t(U_{1,1}^n))\
\big{)}=2\cdot{2^{n-5}+2^{n-6}\choose2}
\end{equation}
where $2\cdot{2^{n-5}+2^{n-6}\choose2}$ crossings are produced
between ${\rm In}(\bigcup\limits_{t\in
\{-n,n\}}\mathscr{I}_t(U_{1,1}^n))$ and $\{{\rm
In}(\mathscr{I}_1(u_{1,j}^n)):j\in [1,2^{n-5}+2^{n-6}]\}$, \\
\\
and that
\begin{equation} \label{equation Ablue-v and Ared-v}
\nu_{\Gamma_n}\big{(}\ {\rm In}(\bigcup\limits_{t\in
\{-n,n\}}\mathscr{I}_t(V_{1,1}^n)), \ {\rm In}(\bigcup\limits_{t\in
\{1,2\}}\mathscr{I}_t(V_{1,1}^n))\ \big{)}=2\cdot{2^{n-5}\choose2}
\end{equation}
where $2\cdot{2^{n-5}\choose2}$ crossings are produced between
$\{{\rm In}(\mathscr{I}_n(v_{1,j}^{n})):j\in [1,2^{n-5}]\}$ and
$\{{\rm In}(\bigcup\limits_{t\in
\{1,2\}}\mathscr{I}_t(v_{1,j}^{n})):j\in [1,2^{n-5}]\}$.

It remains to calculate $\nu_{\Gamma_n}\big{(} \ {\rm
In}(\bigcup\limits_{t\in \{-n,n\}}\mathscr{I}_t(V_{1,1}^n)), \ {\rm
In}(\bigcup\limits_{t\in \{1,2\}}\mathscr{I}_t(U_{1,1}^n)) \
\big{)}$.

By Lemma \ref{Lemma s and t}, we have that
\begin{equation} \label{equation sni1}
s_{n,j\times 2^{n-7}+1}=2\cdot s_{n-1,j\times 2^{n-8}+1}=2^{2}\cdot
s_{n-2,j\times 2^{n-9}+1}=\cdots=2^{n-8}\cdot s_{8,2j+1}
\end{equation}
for all $j\in [0,2^3]$,

 and that \begin{eqnarray}\label{equation sni2}
s_{n,j\times 2^{n-7}}&=&s_{n-1,j\times 2^{n-8}}+s_{n-1,j\times 2^{n-8}+1} \nonumber\\
                       &=&s_{n-2,j\times 2^{n-9}}+s_{n-2,j\times 2^{n-9}+1}+2\cdot s_{n-2,j\times 2^{n-9}+1} \nonumber\\
                       &=&s_{n-2,j\times 2^{n-9}}+(2^{2}-1)\cdot s_{n-2,j\times 2^{n-9}+1} \nonumber\\
                       &=&s_{n-3,j\times 2^{n-10}}+(2^{3}-1)\cdot s_{n-3,j\times 2^{n-10}+1} \nonumber\\
                       &\vdots&\nonumber\\
                       &=&s_{8,2j}+(2^{n-8}-1)\cdot s_{8,2j+1}
\end{eqnarray}
for all $j\in [1,2^3]$.

By \eqref{equation sni1} and Lemma \ref{Lemma s and t}, we have that
\begin{eqnarray} \label{equation sni3}
s_{n,j\times 2^{n-7}+2}&=&s_{n,2\cdot (j\times 2^{n-8}+1)}\nonumber\\
                       &=&s_{n-1,j\times 2^{n-8}+1}+s_{n-1,j\times 2^{n-8}+2} \nonumber\\
                       &=&2\cdot s_{n-2,j\times 2^{n-9}+1}+s_{n-2,j\times 2^{n-9}+1}+s_{n-2,j\times 2^{n-9}+2} \nonumber\\
                       &=&(2^{2}-1)\cdot s_{n-2,j\times 2^{n-9}+1}+s_{n-2,j\times 2^{n-9}+2} \nonumber\\
                       &=&(2^{3}-1)\cdot s_{n-3,j\times 2^{n-10}+1}+s_{n-3,j\times 2^{n-10}+2} \nonumber\\
                       &\vdots&\nonumber\\
                       &=&(2^{n-8}-1)\cdot s_{8,2j+1}+s_{8,2j+2}
\end{eqnarray}
for all $j\in [1,2^3-1]$.

By \eqref{equation sni1}, \eqref{equation sni2} and Lemma \ref{Lemma
s and t}, we have that
\begin{eqnarray} \label{equation sum of sni}
\sum_{j=1}^{2^{n-4}+1}s_{n,i}&=&2\sum_{j=1}^{2^{n-5}+1}s_{n-1,j}+\sum_{j=1}^{2^{n-5}}(s_{n-1,j}+s_{n-1,j+1}) \nonumber\\
                           &=&4\sum_{j=1}^{2^{n-5}+1}s_{n-1,j}-s_{n-1,2^{n-5}+1}-s_{n-1,1} \nonumber\\
                           &=&4\sum_{j=1}^{2^{n-5}+1}s_{n-1,j}-2^{n-9}\cdot s_{8,17}-2^{n-9}\cdot s_{8,1} \nonumber\\
                           &=&4\sum_{j=1}^{2^{n-5}+1}s_{n-1,j}-(s_{8,17}+ s_{8,1})\cdot 2^{n-9} \nonumber\\
                           &=&4^{2}\sum_{j=1}^{2^{n-6}+1}s_{n-2,j}-(s_{8,17}+ s_{8,1})\sum_{j=0}^1 2^{n-9+j} \nonumber\\
                           &\vdots& \nonumber\\
                           &=&4^{n-8}\sum_{j=1}^{2^{4}+1}s_{8,j}-(s_{8,17}+ s_{8,1})\sum_{j=0}^{n-9}2^{n-9+j} \nonumber\\
                           &=&4^{n-8}\sum_{j=1}^{2^{4}+1}s_{8,j}-7\cdot2^{2n-15}+7\cdot2^{n-7}.
\end{eqnarray}

It is easy to verify the following
\begin{table}[htbp]\label{Table 3}
\centering
{\small\begin{tabular}{|c|c|c|c|c|c|c|c|c|c|c|c|c|c|c|c|c|}
  \hline
  $s_{8,1}$ & $s_{8,2}$ & $s_{8,3}$ & $s_{8,4}$  & $s_{8,5}$  & $s_{8,6}$  & $s_{8,7}$  & $s_{8,8}$  & $s_{8,9}$  & $s_{8,10}$ & $s_{8,11}$ & $s_{8,12}$ & $s_{8,13}$ & $s_{8,14}$ & $s_{8,15}$ & $s_{8,16}$ & $s_{8,17}$ \\ \hline
  22 & 20 & 18 & 16 & 14 & 12 & 10 & 10 & 10 & 8 & 6 & 6 & 6 & 6 & 6 & 6 & 6 \\ \hline
  \end{tabular}}
  \caption{\small{The values of $s_{8,j}$ for $j\in [1,2^4+1]$}}
\end{table}

By \eqref{equation sni1}, \eqref{equation sni2}, \eqref{equation sum
of sni}, \eqref{equation definition s i}, \eqref{equation definition
s 2exp(n-4)+1},  Property A , Assertion C, Lemma \ref{Lemma s and t}
and Table 3.1, we conclude that
\begin{eqnarray} \label{equation Ablue-v and Ared-u}
& &\nu_{\Gamma_n}\big{(} \ {\rm In}(\bigcup\limits_{t\in
\{-n,n\}}\mathscr{I}_t(V_{1,1}^n)), \ {\rm In}(\bigcup\limits_{t\in
\{1,2\}}\mathscr{I}_t(U_{1,1}^n)) \
\big{)} \nonumber\\
&=&\nu_{\Gamma_n}\big{(} \ {\rm In}(\mathscr{I}_{-n}(V_{1,1}^n)), \
{\rm In}(\bigcup\limits_{t\in \{1,2\}}\mathscr{I}_t(U_{1,1}^n)) \
\big{)}+\big{(} \ {\rm In}(\mathscr{I}_n(V_{1,1}^n)), \ {\rm
In}(\bigcup\limits_{t\in \{1,2\}}\mathscr{I}_t(U_{1,1}^n)) \
\big{)}\nonumber\\
&=&\big{(}\
\sum_{j=1}^{2^{n-4}}s_{n,j}+(s_{n,2^{n-6}+1}-s_{n,2^{n-6}})+
(s_{n,2^{n-6}+2^{n-7}+1}-s_{n,2^{n-6}+2^{n-7}})+ (s_{n,2^{n-5}+1}-s_{n,2^{n-5}})\nonumber\\
& &\ \ \ +\  (s_{n,2^{n-5}+2^{n-7}+1}-s_{n,2^{n-5}+2^{n-7}}) +
(s_{n,2^{n-5}+2^{n-6}+1}-s_{n,2^{n-5}+2^{n-6}}) \nonumber\\
& &\ \ \ +\ \sum\limits_{j=2^{n-5}+2^{n-6}+2}^{2^{n-4}}(s_{n,j+1}-s_{n,j})\ \big{)}\nonumber\\
& &+\ \big{(}\ 2^{n-5}\cdot s_{n,1}+2^{n-7}\cdot s_{n,2^{n-5}+1}+2^{n-8}\cdot s_{n,2^{n-5}+2^{n-7}+1}+2^{n-8}\cdot s_{n,2^{n-5}+2^{n-6}+1}+2^{n-6}\cdot s_{n,2^{n-4}+1}\ \big{)}\nonumber\\
&=&\big{(}\ \sum_{j=1}^{2^{n-4}+1}s_{n,j}+\sum_{j=2}^{6}s_{n,j\times 2^{n-7}+1}-\sum_{j=2}^{6}s_{n,j\times 2^{n-7}}-s_{n,6\times 2^{n-7}+2}\ \big{)}\nonumber\\
& &+\ \big{(}\ 2^{n-5}\cdot s_{n,1}+2^{n-7}\cdot s_{n,4\times2^{n-7}+1}+2^{n-8}\cdot s_{n,5\times2^{n-7}+1}+2^{n-8}\cdot s_{n,6\times2^{n-7}+1}+2^{n-6}\cdot s_{n,8\times2^{n-7}+1}\ \big{)} \nonumber\\
&=&\big{(}\ (4^{n-8}\sum_{j=1}^{2^{4}+1}s_{8,j}-7\cdot2^{2n-15}+7\cdot2^{n-7})+\sum_{j=2}^{6}2^{n-8}\cdot s_{8,2j+1}-\sum_{j=2}^{6}(s_{8,2j}+(2^{n-8}-1)\cdot s_{8,2j+1}) \nonumber\\
& &\ \ \ -\ ((2^{n-8}-1)\cdot s_{8,13}+s_{8,14})\ \big{)}\nonumber\\
& &+\ \big{(}\ 2^{n-5}\cdot2^{n-8}\cdot s_{8,1}+2^{n-7}\cdot2^{n-8}\cdot s_{8,9}+2^{n-8}\cdot2^{n-8}\cdot s_{8,11}+2^{n-8}\cdot2^{n-8}\cdot s_{8,13}+2^{n-6}\cdot2^{n-8}\cdot s_{8,17}\ \big{)} \nonumber\\
&=&\big{(}\ (182\cdot 2^{2n-16}-7\cdot2^{2n-15}+7\cdot2^{n-7})+46\cdot2^{n-8}-(52+46\cdot(2^{n-8}-1))-(6\cdot(2^{n-8}-1)+6)\ \big{)} \nonumber\\
& &+\ \big{(}\ 22\cdot2^{2n-13}+10\cdot2^{2n-15}+6\cdot2^{2n-16}+6\cdot2^{2n-16}+6\cdot2^{2n-14}\ \big{)} \nonumber\\
&=&25\cdot2^{2n-12}+2^{n-5}-6.
\end{eqnarray}

Therefore, it follows from \eqref{equation E-blue and E-red
general}, \eqref{equation Ablue-u and Ared-v}, \eqref{equation
Ablue-u and Ared-u}, \eqref{equation Ablue-v and Ared-v} and
\eqref{equation Ablue-v and Ared-u} that
\begin{eqnarray} \label{equation E-blue and E-red}
\nu_{\Gamma_n}(E_{blue}^n,E_{red}^n)
&=&8\cdot\big{(}\ 2\cdot{3\cdot2^{n-6}\choose2}+0+2\cdot{2^{n-5}\choose2}+(25\cdot2^{2n-12}+2^{n-5}-6)\ \big{)}\nonumber\\
&=&19\cdot2^{2n-8}-3\cdot2^{n-3}-48.
\end{eqnarray}

\medskip

\textbf{6. Calculation of $\nu_{\Gamma_n}(E_{blue}^n,E_{black}^n)$.}

By Property B and Lemma \ref{lemma counting of crossings}, we have
that \begin{eqnarray} \label{equation E-blue and E-black general}
& &\nu_{\Gamma_n}(E_{blue}^n,E_{black}^n)\nonumber\\
 &=&8\cdot
\nu_{\Gamma_n}\big{(} \ {\rm In}(E_{blue}^n), \ {\rm
In}(E_{black}^n) \ \big{)}\nonumber\\
 &=&8\cdot \nu_{\Gamma_n}\big{(} \ {\rm In}(\bigcup\limits_{t\in
\{-n,n\}}\mathscr{I}_t(U_{1,1}^n\cup V_{1,1}^n)), \ {\rm
In}(E(U_1^n\cup V_1^n)) \ \big{)}\nonumber\\
 &=&8\cdot \big{(} \ \nu_{\Gamma_n}({\rm In}(\bigcup\limits_{t\in
\{-n,n\}}\mathscr{I}_t(U_{1,1}^n)), \ {\rm In}(E(U_1^n\cup
V_1^n)))+\nu_{\Gamma_n}({\rm In}(\bigcup\limits_{t\in
\{-n,n\}}\mathscr{I}_t(V_{1,1}^n)), \ {\rm
In}(E(U_1^n\cup V_1^n))) \ \big{)}\nonumber\\
 &=&8\cdot \big{(} \ \nu_{\Gamma_n}({\rm In}(\bigcup\limits_{t\in
\{-n,n\}}\mathscr{I}_t(U_{1,1}^n)), \ {\rm In}(E(U_1^n\cup
V_1^n)))\nonumber\\
& & \ \ \ \ \ +\ \nu_{\Gamma_n}({\rm In}(\bigcup\limits_{t\in
\{-n,n\}}\mathscr{I}_t(V_{1,1}^n)),\ {\rm In}(E(U_1^n)\cup
E[U_1^n,V_1^n])) +\nu_{\Gamma_n}({\rm In}(\bigcup\limits_{t\in
\{-n,n\}}\mathscr{I}_t(V_{1,1}^n)),\ {\rm In}(E(V_1^n)))\ \big{)}\nonumber\\
\end{eqnarray}

By Lemma \ref{Lemma C+ and C-}, Lemma \ref{Lemma three conclusions
on C+ and C-} and Inductive rule for the drawing of $E_{blue}^n$, we
have that
\begin{eqnarray} \label{equation Ablue-u and E-black}
& &\nu_{\Gamma_n}\big{(}\ {\rm In}(\bigcup\limits_{t\in
\{-n,n\}}\mathscr{I}_t(U_{1,1}^n)), \ {\rm In}(E(U_1^n\cup
V_1^n))\ \big{)}\nonumber\\
&=&\nu_{\Gamma_n}\big{(}\ {\rm In}(\bigcup\limits_{t\in
\{-n,n\}}\mathscr{I}_t(U_{1,1}^n)), \ {\rm In}(E(U_1^n))\ \big{)}\nonumber\\
&=&2\cdot \frac{1}{2}\cdot\mathcal
{C}_{+}^{(n-5)}=\frac{7}{3}\cdot4^{n-4}-(2n-\frac{14}{3}+\frac{7(1+(-1)^n)}{6})\cdot2^{n-5}
\end{eqnarray}
and
\begin{eqnarray} \label{equation Ablue-v and E-black}
& &\nu_{\Gamma_n}\big{(}\ {\rm In}(\bigcup\limits_{t\in\{-n,n\}}\mathscr{I}_t(V_{1,1}^n)),\ {\rm In}(E(V_1^n))\ \big{)}\nonumber\\
&=&\nu_{\Gamma_n}\big{(}\ {\rm In}(\mathscr{I}_{-n}(V_{1,1}^n)),\ {\rm In}(E(V_1^n))\ \big{)}+\nu_{\Gamma_n}\big{(}\ {\rm In}(\mathscr{I}_{n}(V_{1,1}^n)),\ {\rm In}(E(V_1^n))\ \big{)}\nonumber\\
&=&\nu_{\Gamma_n}\big{(}\ {\rm In}(\mathscr{I}_{-n}(V_{1,1}^n)),\ {\rm In}(E(V_1^n))\ \big{)}+\nu_{\Gamma_n}\big{(}\ \{{\rm In}(\mathscr{I}_n(v_{1,j}^n)):j\in [1,2^{n-5}]\},\ {\rm In}(E(V_1^n))\ \big{)}\nonumber\\
& & +\ \nu_{\Gamma_n}\big{(}\ \{{\rm In}(\mathscr{I}_n(v_{1,j}^n)):j\in [2^{n-5}+1,2^{n-4}]\},\ {\rm In}(E(V_1^n))\ \big{)}\nonumber\\
&=&\frac{1}{2}\cdot\mathcal{C}_{-}^{(n-5)}+\sum\limits_{j\in \mathcal{I}_1^{(n-5)}\cup \mathcal{I}_2^{(n-5)}}\mathscr{C}_{+}^{(n-5)}(\mathcal {E}^{(n-5)},z_j^{(n-5)})+\sum\limits_{j\in \mathcal{I}_3^{(n-5)}\cup \mathcal{I}_4^{(n-5)}}\mathscr{C}_{-}^{(n-5)}(\mathcal {E}^{(n-5)},z_j^{(n-5)})\nonumber\\
&=&\frac{5}{3}\cdot2^{2n-9}-(2n-\frac{17}{3}+\frac{7\cdot (1+(-1)^{n-1})}{6})\cdot2^{n-6} \nonumber\\
& &+\ \frac{35}{3}\cdot2^{2n-13}-(4n-\frac{28}{3}+\frac{7\cdot (1+(-1)^{n})}{3})\cdot2^{n-8} \nonumber\\
& &+\ \frac{43}{3}\cdot2^{2n-13}-(4n-\frac{34}{3}+\frac{7\cdot (1+(-1)^{n-1})}{3})\cdot2^{n-8} \nonumber\\
&=&\frac{79}{3}\cdot2^{2n-12}-(4n-\frac{29}{3}+\frac{7\cdot(1+(-1)^{n-1})}{6})\cdot2^{n-6}.
\end{eqnarray}

It remains to calculate $\nu_{\Gamma_n}({\rm
In}(\bigcup\limits_{t\in \{-n,n\}}\mathscr{I}_t(V_{1,1}^n)),\ {\rm
In}(E(U_1^n)\cup E[U_1^n,V_1^n]))$.

Similarly as \eqref{equation sni1}, \eqref{equation sni2},
\eqref{equation sni3} and \eqref{equation sum of sni}, we can derive
that
\begin{equation} \label{equation tni1}
t_{n,j\times 2^{n-7}+1}=2^{n-8}\cdot t_{8,2j+1} \mbox{ \ \ for all }
\ j\in [0,2^3],
\end{equation}

and that
\begin{equation} \label{equation tni2}
t_{n,j\times 2^{n-7}}=t_{8,2j}+(2^{n-8}-1)\cdot
t_{8,2j+1}+2\cdot(n-8) \mbox{ \ \ for all } \ j\in [1,2^3],
\end{equation}

and that
\begin{equation} \label{equation tni3}
t_{n,j\times 2^{n-7}+2}=(2^{n-8}-1)\cdot
t_{8,2j+1}+t_{8,2j+2}+2\cdot(n-8) \mbox{ \ \ for all } \ j\in
[1,2^3-1],
\end{equation}

and that \begin{equation} \label{equation sum of tni}
\sum_{j=1}^{2^{n-4}+1}t_{n,j}=4^{n-8}\sum_{j=1}^{2^{4}+1}t_{8,j}.
\end{equation}

It is easy to verify the following
\begin{table}[htbp]\label{Table 4}
\centering
{\small\begin{tabular}{|c|c|c|c|c|c|c|c|c|c|c|c|c|c|c|c|c|}
  \hline
  $t_{8,1}$ & $t_{8,2}$ & $t_{8,3}$ & $t_{8,4}$  & $t_{8,5}$  & $t_{8,6}$  & $t_{8,7}$  & $t_{8,8}$  & $t_{8,9}$  & $t_{8,10}$ & $t_{8,11}$ & $t_{8,12}$ & $t_{8,13}$ & $t_{8,14}$ & $t_{8,15}$ & $t_{8,16}$ & $t_{8,17}$ \\ \hline
  0 & 10 & 16 & 22 & 24 & 30 & 32 & 34 & 32 & 38 & 40 & 42 & 40 & 42 & 40 & 38 & 32 \\ \hline
  \end{tabular}}
  \caption{\small{The values of $t_{8,j}$ for $j\in [1,2^4+1]$}}
\end{table}

Notice that $$E(U_1^n)\cup E[U_1^n,V_1^n]=H_{black}^n.$$ By
\eqref{equation definition t i}, \eqref{equation definition t
2exp(n-4)+1}, \eqref{equation t' i}, \eqref{equation t'
2exp(n-4)+1}, \eqref{equation t and t'}, Lemma \ref{lemma counting
of crossings}, Property A, Assertion C and Table 3.2, we conclude
that
\begin{eqnarray}\label{equation E blue and E black final}
& &\nu_{\Gamma_n}\big{(}\ {\rm In}(\bigcup\limits_{t\in
\{-n,n\}}\mathscr{I}_t(V_{1,1}^n)), \ {\rm
In}(E(U_1^n)\cup E[U_1^n,V_1^n])\ \big{)} \nonumber\\
&=&\nu_{\Gamma_n}\big{(}\ {\rm In}(\mathscr{I}_{-n}(V_{1,1}^n)),\
{\rm In}(E(U_1^n)\cup E[U_1^n,V_1^n])\ \big{)}\nonumber\\
& &+\ \nu_{\Gamma_n}\big{(}\ {\rm In}(\mathscr{I}_{n}(V_{1,1}^n)),\ {\rm In}(E(U_1^n))\ \big{)}+\nu_{\Gamma_n}\big{(}\ {\rm In}(\mathscr{I}_{n}(V_{1,1}^n)),\ E[U_{1,1}^n,V_{1,1}^n]\ \big{)}\nonumber\\
&=&\big{(}\
\sum_{j=1}^{2^{n-4}}t_{n,j}+(t_{n,2^{n-6}+1}-t_{n,2^{n-6}})+
        (t_{n,2^{n-6}+2^{n-7}+1}-t_{n,2^{n-6}+2^{n-7}})+ (t_{n,2^{n-5}+1}-t_{n,2^{n-5}})\nonumber\\
& &\ \ \ +\ (t_{n,2^{n-5}+2^{n-7}+1}-t_{n,2^{n-5}+2^{n-7}}) +
        (t_{n,2^{n-5}+2^{n-6}+1}-t_{n,2^{n-5}+2^{n-6}}) \nonumber\\
& &\ \ \ +\ \sum\limits_{j=2^{n-5}+2^{n-6}+2}^{2^{n-4}}(t_{n,j+1}-t_{n,j})\ \big{)}\nonumber\\
& &+\ \big{(}\ 2^{n-5}\cdot t_{n,1}^{'}+2^{n-7}\cdot t_{n,2^{n-5}+1}^{'}+2^{n-8}\cdot t_{n,2^{n-5}+2^{n-7}+1}^{'}+2^{n-8}\cdot t_{n,2^{n-5}+2^{n-6}+1}^{'}+2^{n-6}\cdot t_{n,2^{n-4}+1}^{'}\ \big{)}\nonumber\\
& &+\ \big{(}\ 2\cdot{2^{n-7}\choose2}+2\cdot2\cdot{2^{n-8}\choose2}+2\cdot{2^{n-6}\choose2}\ \big{)}\nonumber\\
&=&\big{(}\
\sum_{j=1}^{2^{n-4}}t_{n,j}+(t_{n,2^{n-6}+1}-t_{n,2^{n-6}})+
        (t_{n,2^{n-6}+2^{n-7}+1}-t_{n,2^{n-6}+2^{n-7}})+ (t_{n,2^{n-5}+1}-t_{n,2^{n-5}})\nonumber\\
& &\ \ \ +\ (t_{n,2^{n-5}+2^{n-7}+1}-t_{n,2^{n-5}+2^{n-7}}) +
        (t_{n,2^{n-5}+2^{n-6}+1}-t_{n,2^{n-5}+2^{n-6}}) \nonumber\\
& &\ \ \ +\ \sum\limits_{j=2^{n-5}+2^{n-6}+2}^{2^{n-4}}(t_{n,j+1}-t_{n,j})\ \big{)}\nonumber\\
& &+\ \big{(}\ 2^{n-5}\cdot t_{n,1}+2^{n-7}\cdot t_{n,2^{n-5}+1}+2^{n-8}\cdot t_{n,2^{n-5}+2^{n-7}+1}+2^{n-8}\cdot t_{n,2^{n-5}+2^{n-6}+1}+2^{n-6}\cdot t_{n,2^{n-4}+1}\ \big{)}\nonumber\\
& &+\ \big{(}\ 2\cdot{2^{n-7}\choose2}+2\cdot2\cdot{2^{n-8}\choose2}+2\cdot{2^{n-6}\choose2}\ \big{)}\nonumber\\
&=&\big{(}\ \sum_{j=1}^{2^{n-4}+1}t_{n,j}+\sum_{j=2}^{6}t_{n,j\times 2^{n-7}+1}-\sum_{j=2}^{6}t_{n,j\times 2^{n-7}}-t_{n,6\times 2^{n-7}+2}\ \big{)}\nonumber\\
& &+\ \big{(}\ 2^{n-7}\cdot t_{n,4\times 2^{n-7}+1}+2^{n-8}\cdot t_{n,5\times 2^{n-7}+1}+2^{n-8}\cdot t_{n,6\times 2^{n-7}+1}+2^{n-6}\cdot t_{n,8\times 2^{n-7}+1}\ \big{)}\nonumber\\
& &+\ \big{(}\ 11\cdot2^{2n-15}-2^{n-5}\ \big{)}\nonumber\\
&=&\big{(}\ 4^{n-8}\sum_{j=1}^{2^{4}+1}t_{8,j}+\sum_{j=2}^{6}2^{n-8}\cdot t_{8,2j+1}-\sum_{j=2}^{6}(t_{8,2j}+(2^{n-8}-1)\cdot t_{8,2j+1}+2\cdot(n-8))\nonumber\\
& &\ \ \ -\ ((2^{n-8}-1)\cdot t_{8,13}+t_{8,14}+2\cdot (n-8))\ \big{)}\nonumber\\
& &+\ \big{(}\ 2^{n-7}\cdot2^{n-8}\cdot t_{8,9}+2^{n-8}\cdot2^{n-8}\cdot t_{8,11}+2^{n-8}\cdot2^{n-8}\cdot t_{8,13}+2^{n-6}\cdot2^{n-8}\cdot t_{8,17}\ \big{)}\nonumber\\
& &+\ \big{(}\ 11\cdot2^{2n-15}-2^{n-5}\ \big{)}\nonumber\\
&=&\big{(}\ 512\cdot2^{2n-16}+168\cdot2^{n-8}-(166+168\cdot(2^{n-8}-1)+10\cdot(n-8))\nonumber\\
& &\ \ \ -\ (40\cdot (2^{n-8}-1)+42+2\cdot (n-8))\ \big{)}\nonumber \\
& &+\ \big{(}\ 32\cdot2^{2n-15}+40\cdot2^{n-16}+40\cdot2^{2n-16}+32\cdot2^{n-14}\ \big{)}\nonumber\\
& &+\ \big{(}\ 11\cdot2^{2n-15}-2^{n-5}\ \big{)}\nonumber\\
&=&403\cdot2^{2n-15}-3\cdot2^{n-4}-12n+96.
\end{eqnarray}

Therefore, it follows from \eqref{equation E-blue and E-black
general}, \eqref{equation Ablue-u and E-black}, \eqref{equation
Ablue-v and E-black} and \eqref{equation E blue and E black final}
that
\begin{eqnarray} \label{equation E-blue and E-black}
\nu_{\Gamma_n}(E_{blue}^n,E_{black}^n)
&=&8\cdot\big{(}\ \frac{7}{3}\cdot4^{n-4}-(2n-\frac{14}{3}+\frac{7(1+(-1)^n)}{6})\cdot2^{n-5}\nonumber\\
& &\ \ \ \ \ +\ 403\cdot2^{2n-15}-3\cdot2^{n-4}-12n+96 \nonumber\\
& &\ \ \ \ \ +\ \frac{79}{3}\cdot2^{2n-12}-(4n-\frac{29}{3}+\frac{7\cdot(1+(-1)^{n-1})}{6})\cdot2^{n-6}\ \big{)}\nonumber\\
&=&8\cdot\big{(}\ \frac{2737}{3}\cdot2^{2n-15}-(8n-\frac{14}{3}+\frac{7\cdot (1+(-1)^n)}{6})\cdot 2^{n-6}-12n+96\ \big{)}\nonumber\\
&=&\frac{2737}{3}\cdot2^{2n-12}-(8n-\frac{14}{3}+\frac{7\cdot
(1+(-1)^n)}{6})\cdot 2^{n-3}-96n+768.
\end{eqnarray}

By \eqref{equation E-blue}, \eqref{equation E-red}, \eqref{equation
general E-black}, \eqref{equation E-blue and E-red}, \eqref{equation
E-red and E-black} and \eqref{equation E-blue and E-black}, for
$n\geq 8$, we have that
$$\begin{array}{rlll}
\nu_{\Gamma_n}(AQ_n)&=&\nu_{\Gamma_n}(E_{blue}^n)+\nu_{\Gamma_n}(E_{red}^n)+\nu_{\Gamma_n}(E_{black}^n)\\
& &+\ \nu_{\Gamma_n}(E_{red}^n, E_{blue}^n)+\nu_{\Gamma_n}(E_{red}^n, E_{black}^n)+\nu_{\Gamma_n}(E_{blue}^n, E_{black}^n)\\
&=&\big{(}\ 11\cdot2^{2n-13}+2^{n-3}\ \big{)} \\
& &+\ \big{(}\ 71\cdot2^{2n-10}-5\cdot2^{n-4}\ \big{)} \\
& &+\ \big{(}\ 59\cdot 2^{2n-8}-\big{(}4n^2-9n-6\big{)}\cdot
2^{n-3}-7\cdot(1+(-1)^{n-1})\cdot2^{n-4}\ \big{)} \\
& &+\ \big{(}\ 19\cdot2^{2n-8}-3\cdot2^{n-3}-48 \ \big{)}  \\
& &+\ \big{(}\ 389\cdot2^{2n-11}-(n-1)\cdot2^{n}+7\cdot(1+(-1)^{n-1})\cdot2^{n-5}\ \big{)} \\
& &+\ \big{(}\ \frac{2737}{3}\cdot2^{2n-12}-(8n-\frac{14}{3}+\frac{7\cdot(1+(-1)^n)}{6})\cdot 2^{n-3}-96n+768\ \big{)} \\
&=&\frac{19367}{3}\cdot 2^{2n-13}-(8n^2+14n-\frac{85}{3})\cdot 2^{n-4}-(\frac{28}{3}+\frac{7\cdot (1+(-1)^{n-1})}{3})\cdot2^{n-5}-96n+720\\
&=&\frac{19367}{3}\cdot 2^{2n-13}-(8n^2+14n-\frac{71}{3})\cdot 2^{n-4}-\frac{7\cdot (1+(-1)^{n-1})}{3}\cdot2^{n-5}-96n+720\\
&<&\frac{19367}{3}\cdot 2^{2n-13}-(8n^2+14n-\frac{71}{3})\cdot 2^{n-4}\\
&<&\frac{19368}{3}\cdot 2^{2n-13}-(8n^2+14n-\frac{72}{3})\cdot 2^{n-4}\\
&=&807\cdot 4^{n-5}-(4n^2+7n-12)\cdot 2^{n-3}\\
&<&\frac{26}{32}\cdot4^{n}-(2n^2+\frac{7}{2}n-6)\cdot2^{n-2}. \\
\end{array}$$
This completes the proof of Theorem \ref{Theorem Upper Bound for
general n}. \qed

\section{Concluding remarks}

In this section, we make a study of the crossing number of $AQ_n$
for $n\leq 7$. It is clear that $cr(AQ_1)=cr(AQ_2)=0$. In Figure 4.1
(1), we show a drawing  of $AQ_3$ with 4 crossings, it means
$cr(AQ_3)\leq 4$. Since $AQ_3$ contains a subgraph isomorphic to
$K_{4,4}$, we have $cr(AQ_3)\geq cr(K_{4,4})=4$. Hence,
\begin{proposition}
$cr(AQ_3)=4$.
\end{proposition}

\begin{figure}[ht]
\centering
\includegraphics[scale=1.0]{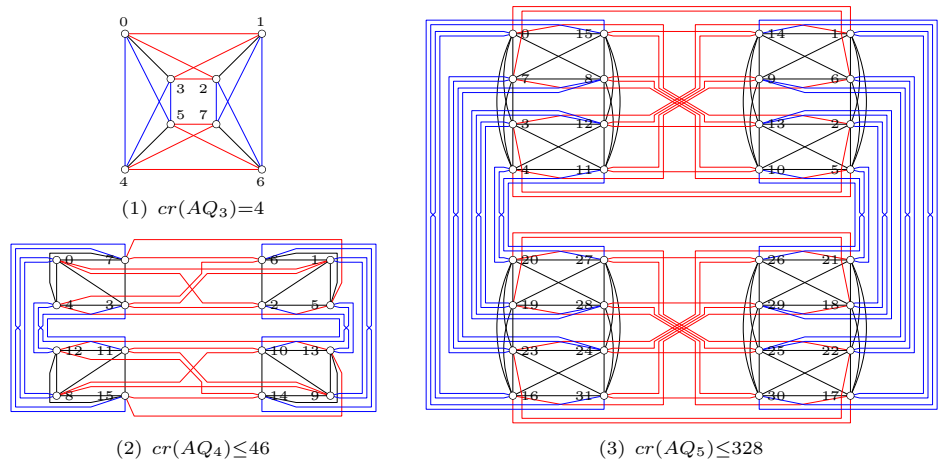}
\caption{\small{Some drawings of $AQ_3$, $AQ_4$ and $AQ_5$}}
\end{figure}

Meanwhile, in Figure 4.1 (2)-(3), Figure 4.2 and Figure 4.3 we show
the drawings of $AQ_4,AQ_5,AQ_6$ and $AQ_7$ with 46, 328, 1848 and
9112 crossings, respectively. We remark that the calculations of
crossings in the following drawings which are given in Figures
4.1-4.3 are similar to $\nu_{\Gamma_n}(AQ_n)$ for $n\geq 8$, and
omit them here. Hence, we have the following
\begin{proposition}
$cr(AQ_4)\leq 46$, $cr(AQ_5)\leq 328$, $cr(AQ_6)\leq 1848$,
$cr(AQ_7)\leq 9112$.
\end{proposition}

In the final of this paper, by just applying the same technique of
congestions proposed by Leighton \cite{L81}, we can obtain the following lower bound:
$$cr(AQ_n)>\frac{4^n}{5\times(1+2^{2-n})^2}-(4n^2+4n+\frac{17}{5})2^{n-1}.$$

\newpage

\begin{figure}[h]
\centering \vspace{100pt}
\includegraphics[scale=1.0]{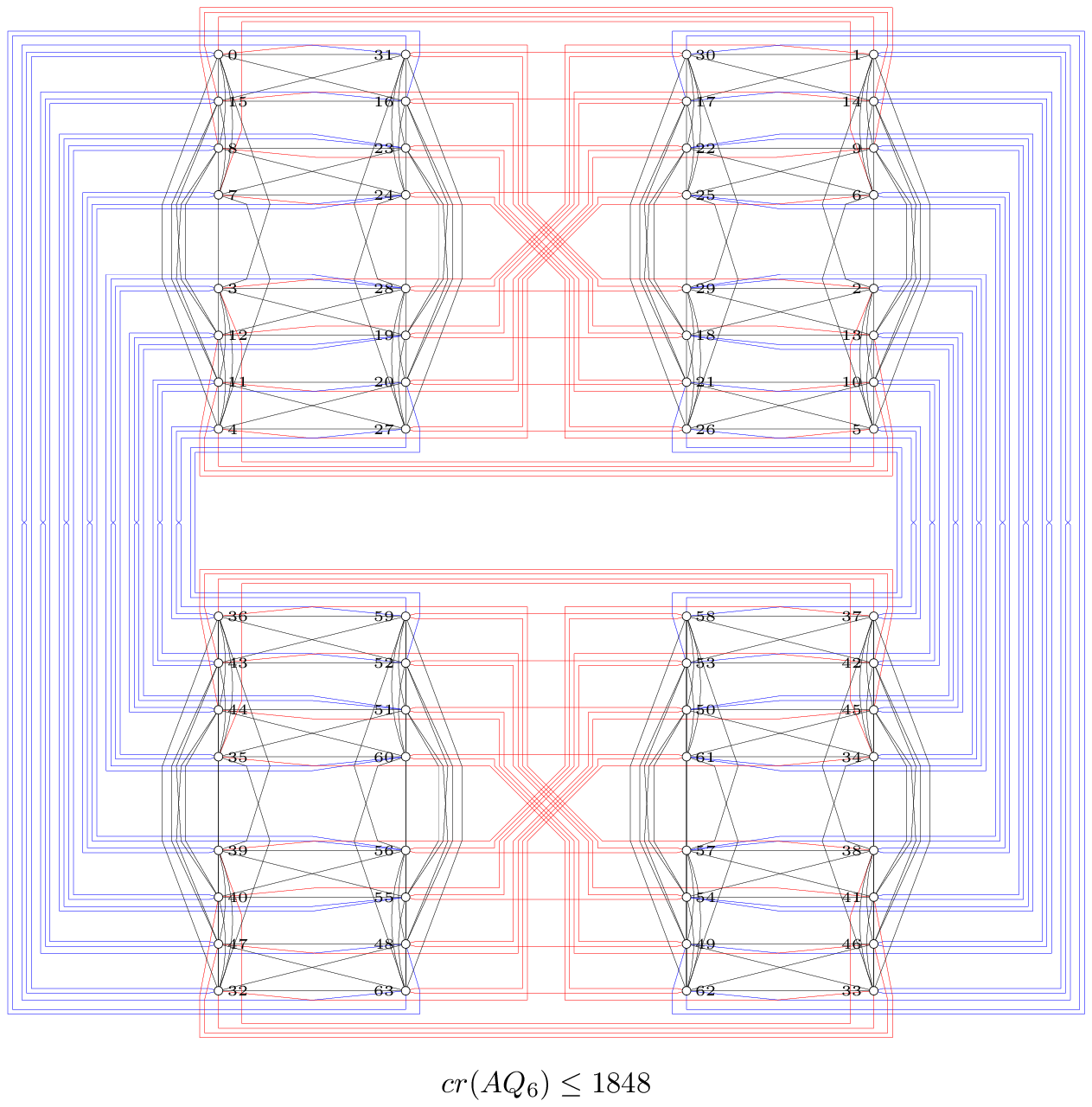}
\caption{\small{A drawing of $AQ_6$ with 1848 crossings}}
\end{figure}

\newpage

\begin{figure}[h]
\centering
\includegraphics[scale=1.0]{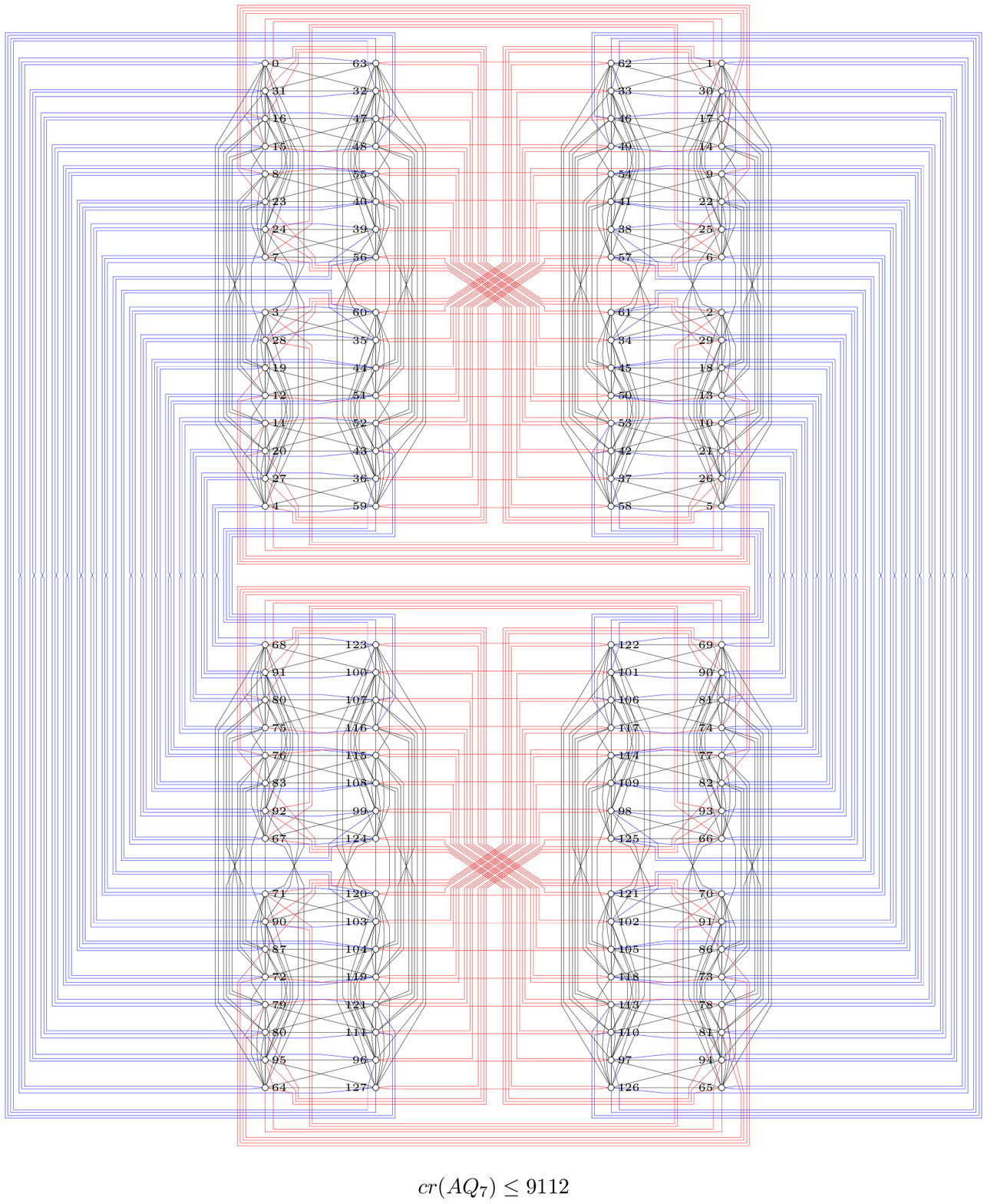}
\caption{\small{A drawing of $AQ_7$ with 9112 crossings}}
\end{figure}


\newpage








\begin{thebibliography}{99}

\bibitem{BL84}
S.N. Bhatt, F.T. Leighton,
\newblock A framework for solving VLSI graph layout problems,
\newblock {\it J. Comput. System Sci.} 28 (1984) 300--343.

\bibitem{BA84}
L. Bhuyan, D.P. Agrawal,
\newblock Generalized hypercubes and hyperbus structure for a computer network,
\newblock {\it IEEE Trans. Comput.} 33 (1984) 323--333.

\bibitem{CCWH09}
H.C. Chan, J.M. Chang, Y.L. Wang, S.J. Horng,
\newblock Geodesic-pancyclicity and fault-tolerant panconnectivity of augmented cubes,
\newblock {\it Appl. Math. Comput.} 207 (2009) 333--339.

\bibitem{CH12}
N.W. Chang, S.Y. Hsieh,
\newblock Conditional diagnosability of augmented cubes under the PMC model,
\newblock {\it IEEE Transactions on Dependable and Secure Computing} 9 (2012) 46--60.

\bibitem{CS02}
S.A. Choudum, V. Sunitha,
\newblock Augmented Cubes,
\newblock {\it Networks} 40 (2002) 71--84.

\bibitem{DR95}
A.M. Dean, R.B. Richter,
\newblock The crossing number of $C_4\times C_4$,
\newblock {\it J. Graph Theory} 19 (1995) 125--129.

\bibitem{Egg70}
R.B. Eggleton, R.K. Guy,
\newblock The crossing number of the $n$-cube,
\newblock {\it Notices Amer. Math. Soc.} 17 (1970) 757--757.

\bibitem{EG73}
P. Erd\H{o}s, R.K. Guy,
\newblock Crossing number problems,
\newblock {\it Amer. Math. Monthly} 80 (1973) 52--58.

\bibitem{FF00}
L. Faria, C.M.H. de Figueiredo,
\newblock On Eggleton and Guy's conjectured upper bound for the crossing number of the $n$-cube,
\newblock {\it Math. Slovaca} 50 (2000) 271--287.

\bibitem{FFSV08}
L. Faria, C.M.H. de Figueiredo, O. Sykora, I. Vrt'o,
\newblock An improved upper bound on the crossing number of the hypercube,
\newblock {\it J. Graph Theory} 59 (2008) 145--159.

\bibitem{GJ83}
M.R. Garey, D.S. Johnson,
\newblock Crossing number is NP-complete,
\newblock {\it SIAM J. Alg. Disc. Math.} 4 (1983) 312--316.

\bibitem{Guy60}
R.K. Guy,
\newblock A combinatorial problem,
\newblock {\it Nabla(Bull. Malayan Math. Soc.)} 7 (1960) 68--72.

\bibitem{HH12}
W.S. Hong, S.Y. Hsieh,
\newblock Strong diagnosability and conditional diagnosability of augmented cubes under the comparison diagnosis model,
\newblock {\it IEEE Transactions on Reliability} 61 (2012) 140--148.

\bibitem{HC10}
S.Y. Hsieh, Y.R. Cian,
\newblock Conditional edge-fault Hamiltonicity of augmented cubes,
\newblock {\it Inform. Sci.} 180 (2010) 2596--2617.

\bibitem{HS07}
S.Y. Hsieh, J.Y. Shiu,
\newblock Cycle embedding of augmented cubes,
\newblock {\it Appl. Math. Comput.} 191 (2007) 314--319.

\bibitem{HCTH05}
H.C. Hsu, L.C. Chiang, J.J.M. Tan, L.H. Hsu,
\newblock Fault hamiltonicity of augmented cubes,
\newblock {\it Parallel Comput.} 31 (2005) 130--145.

\bibitem{HLT07}
H.C. Hsu, P.L. Lai, C.H. Tsai,
\newblock Geodesic pancyclicity and balanced pancyclicity of augmented cubes,
\newblock {\it Inform. Process. Lett.} 101 (2007) 227--232.


\bibitem{LTTH09}
C.M. Lee, Y.H. Teng, Jimmy J.M. Tanc, L.H. Hsu,
\newblock Embedding Hamiltonian paths in augmented cubes with a required vertex in a fixed position,
\newblock {\it Computers and Mathematics with Applications} 58 (2009) 1762--1768.

\bibitem{L81}
F.T. Leighton,
\newblock New lower bound techniques for VLSI,
\newblock {\it Math. Systems Theory} 17 (1984) 47--70.

\bibitem{L83}
F.T. Leighton,
\newblock Complexity Issues in VLSI,
\newblock {\it Found. Comput. Ser.,} MIT Press, Cambridge, MA (1983).

\bibitem{L92}
F.T. Leighton,
\newblock Introduction to Parallel Algorithms and Architecture: Arrays, Trees, Hypercubes,
\newblock Morgan Kaufmann, San Mateo, CA (1992).

\bibitem{LiYaZh}
X.H. Lin, Y.S. Yang, W.P. Zheng, L. Shi, W.M. Lu,
\newblock The crossing numbers of generalized Petersen graphs with small order,
\newblock {\it Discrete Appl. Math.} 157 (2009) 1016--1023.

\bibitem{MLX07}
M. Ma, G. Liu, J.M. Xu,
\newblock Panconnectivity and edge-fault-tolerant pancyclicity of augmented cubes,
\newblock {\it Parallel Comput.} 33 (2007) 35--42.

\bibitem{M91}
T. Madej,
\newblock Bounds for the crossing number of the $n$-cube,
\newblock {\it J. Graph Theory} 15 (1991) 81--97.


\bibitem{PanRich07}
S. Pan, R.B. Richter,
\newblock The crossing number of $K_{11}$ is 100,
\newblock {\it J. Graph Theory} 56 (2007) 128--134.

\bibitem{RichterThomassen95}
R.B. Richter, C. Thomassen,
\newblock Intersections of curve systems and the crossing number of $C_5\times C_5$,
\newblock {\it Discrete Comput. Geom.} 13 (1995) 149--159.

\bibitem{S05}
G. Salazar,
\newblock On the crossing numbers of loop networks and generalized Petersen graphs,
\newblock {\it Discrete Math.} 302 (2005) 243--253.

\bibitem{SV93}
O. Sykora, I. Vrt'o,
\newblock On crossing numbers of hypercubes and cube connected cycles,
\newblock {\it BIT} 33 (1993) 232--237.

\bibitem{Turan77}
P. Tur\'{a}n,
\newblock A note of welcome,
\newblock {\it J. Graph Theory} 1 (1977) 7--9.

\bibitem{Tutte70}
W.T. Tutte,
\newblock Toward a theory of crossing numbers,
\newblock {\it J. Combinatorial Theory} 8 (1970) 45--53.

\bibitem{WMX07}
W.W. Wang, M.J. Ma, J.M. Xu,
\newblock Fault-tolerant pancyclicity of augmented cubes,
\newblock {\it Inform. Process. Lett.} 103 (2007) 52--56.

\bibitem{XX07}
M. Xu, J.M. Xu,
\newblock The forwarding indices of augmented cubes,
\newblock {\it Inform. Process. Lett.} 101 (2007) 185--189.


\bibitem{YangWang}
Y.S. Yang, G.Q. Wang, H.L. Wang, Y. Zhou,
\newblock  The Erd\H{o}s and Guy's conjectured equality on the crossing number of hypercubes,
\newblock arXiv:1201.4700v1.

\end{thebibliography}
\end{document}